\documentclass[12pt]{article}

\usepackage{authblk}

\title{Smoothing effect and large time behavior of solutions to 
nonlinear elastic wave equations \\
with viscoelastic term}
\author[1]{Yoshiyuki Kagei}
\author[2]{Hiroshi Takeda}
\affil[1]{Department of Mathematics, 
Tokyo Institute of Technology, 
Meguro-ku, Tokyo, 152-8551, Japan}
\affil[2]{Department of Intelligent Mechanical Engineering, 
Faculty of Engineering, Fukuoka Institute of Technology, 
3-30-1 Wajiro-higashi, Higashi-ku, Fukuoka, 811-0295, Japan}
\date{}

\usepackage{amssymb,amsmath,amsthm}
\usepackage[dvips]{graphicx}
\usepackage{bm} 

\newcommand{\R}{\mathbb R}

\newcommand{\supp}{\mathop{\mathrm{supp}}\nolimits}

\newtheorem{thm}{Theorem}[section]
\newtheorem{cor}[thm]{Corollary}
\newtheorem{prop}[thm]{Proposition}
\newtheorem{lem}[thm]{Lemma}
\theoremstyle{remark}
\newtheorem{rem}[thm]{Remark}

\theoremstyle{definition}

\begin{document}
\maketitle

\numberwithin{equation}{section}

\begin{abstract}
The Cauchy problem for a nonlinear elastic wave equations with viscoelastic damping terms is considered on the 3 dimensional whole space. Decay and smoothing properties of the solutions are investigated when the initial data are sufficiently small; 
and asymptotic profiles as $t \to \infty$ are also derived. 
\end{abstract}


\newpage
\section{Introduction}

In this paper we consider the Cauchy problem for the system of quasi-linear elastic equations
with strong damping:
\begin{equation} \label{eq:1.1}
\left\{
\begin{split}
& \partial_{t}^{2} u -\mu \Delta u - (\lambda + \mu) \nabla {\rm div} u -\nu \Delta \partial_{t} u  =F(u), \quad t>0, \quad x \in \R^{3}, \\
& u(0,x)=f_{0}(x), \quad \partial_{t} u(0,x)=f_{1}(x) , \quad x \in \R^{3}, 
\end{split}
\right.
\end{equation}
where $u={}^t\! (u_{1}, u_{2}, u_{3})$ is the unknown function; and
$f_{j}={}^t\! (f_{j1}, f_{j2}, f_{j3})$ $(j=0,1)$ are initial data.
Here and in what follows the superscript ${}^t\! \cdot$ stands for the transpose of the matrix. 
We assume that the Lam\'e constants satisfy 
\begin{equation} \label{eq:1.222}
\mu>0, \quad \lambda + 2 \mu >0, 
\end{equation}
and the viscosity parameter $\nu$ is positive. 
We also assume that the nonlinear term $F(u)$ is given by $\nabla u \nabla D u$,
where $\nabla$ is the spatial gradient and $D$ represents the $t,x$ gradient.
As is mentioned in \cite{J-S}, the Cauchy problem \eqref{eq:1.1} serves as a simplified, dimensionless model for viscoelasticity.   
It is also known the relationship with the fluid dynamics (see \cite{Ponce}, \cite{GMcM}, \cite{A} and \cite{PoF} for the detail).
Here we note that
for more general setting, the existence of sufficiently smooth and small solutions are well-investigated (see e.g. \cite{KS} for bounded domain, \cite{KPS} and \cite{ST} for unbouded domains). 

On the other hand, 
our concern in this paper is large time behavior of the solution to \eqref{eq:1.1}, including sharp time decay properties.
When $\lambda +\mu=0$, the interaction of each components of the solutions are from the nonlinear term only.
In this case, 
Ponce \cite{Ponce} proved the existence of global solutions in $L^{2}$-Sobolev spaces:
$$
C([0,\infty); \dot{H}^{s_{0}} \cap \dot{H}^{1} ) \cap C^{1}([0,\infty); H^{s_{0}} ) \cap C^{2}([0,\infty); H^{s_{0}-2} ) 
$$
with an integer $s_{0} > \frac{7}{2}$ for small initial data  $(f_{0}, f_{1}) \in \{ \dot{H}^{s_{0}} \cap \dot{H}^{1} \cap \dot{W}^{1,1} \}^{3} \times  \{H^{s_{0}} \cap {L}^{1} \}^{3}$
and obtained decay estimates 
\begin{equation}  \label{eq:1.333}
\| \nabla^{\alpha} u(t) \|_{L^{q}(\R^{3})} \le Ct^{-\frac{3}{2}(1-\frac{1}{q})+\frac{1}{2}-\frac{\alpha}{2}}
\end{equation}
for $2 \le q \le \infty$ and $1 \le \alpha \le 3$.
The method of proof in \cite{Ponce} is based on the energy method for the higher order derivatives of the solution and the 
$L^{p}$-$L^{q}$ type estimates for the fundamental solution of the linearized equation that are obtained by the parabolic aspect of \eqref{eq:1.1}.
He also obtained faster decay properties with the special case $f_{1}(x)= \nabla g_{1}(x)$ and the nonlinear term is given by the divergence form, 
applying same method.
Jonov-Sideris \cite{J-S} considered \eqref{eq:1.1} with $\lambda +\mu=0$, when the nonlinear terms are decomposed into two parts;
$$
F(u) =F_{0}(u)+ \delta F_{1}(u),  
$$
where $F_{0}(u)$ satisfies the Klainerman null condition (cf. \cite{Klainerman}, \cite{Chris}, \cite{Sideris} and \cite{Agemi})
and the constant $|\delta|$ is sufficiently small, which means that $\delta F_{1}(u)$ is a small perturbation of $F_{0}(u)$.
They quantify the influence of the parameters $\nu$, $\delta$ and the size of the initial data to ensure the global existence of the solutions in the weighted $L^{2}$-Sobolev space with high regularity.
As a result, they have decay properties of the solutions. 
Their proof relied on the weighted energy method based on the vector field of wave equations.
For this kind of argument, we also refer to \cite{Alinhac} and \cite{Strauss}.

In the contrast, we deal with the case $\lambda +\mu \neq 0$ without any structural condition on the nonlinear terms like null condition and investigate the smoothing effect and large time behavior of small solutions to \eqref{eq:1.1}.

We shall obtain the global solutions and decay properties without ``derivative-loss'', which is our first result.
In what follows, we only consider the case where the nonlinear term $F(u)$ contains only $x$ derivatives. 
The precise statements of our results are formulated as follows.
%
\begin{thm} \label{thm:1.1}
Suppose that $F(u) =\nabla u \nabla^{2} u$.
Let $(f_{0}, f_{1}) \in \{ \dot{H}^{3} \cap \dot{W}^{1,1} \}^{3} \times  \{H^{1} \cap {L}^{1} \}^{3}$ with sufficiently small norms. 
Then there exists a unique global solution to \eqref{eq:1.1} in the class
$$
\{ C([0,\infty); \dot{H}^{3} \cap \dot{H}^{1} ) \cap C^{1}([0,\infty); H^{1}) \}^{3}
$$
satisfying the following time decay properties:
\begin{equation} \label{eq:1.2}
\begin{split}
\| \nabla^{\alpha} u(t) \|_{L^{2}(\R^{3})} & 
\le C (1+t)^{-\frac{1}{4}-\frac{\alpha}{2}}, \quad 1 \le \alpha \le 3, \\
\| \partial_{t}  \nabla^{\alpha} u(t) \|_{L^{2}(\R^{3})} & 
\le C (1+t)^{-\frac{3}{4}-\frac{\alpha}{2}}, \quad 0 \le \alpha \le 1
\end{split}
\end{equation}
for $t \ge 0$. 
\end{thm}
\begin{rem}
	In Theorem \ref{thm:1.1}, 
	we used the facts that $\dot{H}^{3} \cap \dot{W}^{1,1}=\dot{H}^{3} \cap \dot{H}^{1} \cap \dot{W}^{1,1}$ and 
	$\dot{H}^{k} \cap L^{1}=H^{k} \cap L^{1}$ for $k=0,1$, which are shown later in preliminaries. 
\end{rem}
Next we state the smoothing effect of global solutions.
\begin{thm} \label{thm:1.3}
	The global solution $u(t)$ constructed in Theorem \ref{thm:1.1} satisfies 
	\begin{equation*}
		\begin{split}
			u \in \{ C^{1}( (0,\infty); \bigcup_{2 \le p < 6} \dot{W}^{2,p}) 
			\cap W^{1, \infty}( 0,\infty; W^{1,\infty}) \cap C^{2} ((0,\infty); \bigcup_{2 \le p < 6} L^{p} ) \}^{3}
		\end{split}
	\end{equation*}
	 and 
	 \begin{align}
	 		\| \nabla^{\alpha} u(t) \|_{L^{\infty}(\R^{3})} & 
	 		\le C (1+t)^{-\frac{3}{2}-\frac{\alpha}{2}}, \quad 0 \le \alpha \le 1 \label{eq:1.3}
	 \end{align}
	for $t \ge 0$ and 
 \begin{align}
	\| \nabla^{2} \partial_{t} u(t) \|_{L^{p}(\R^{3})} & \le C(1+t)^{-\frac{7}{4}+\frac{1}{p}} t^{-\frac{5}{4}+\frac{3}{2p} }, \quad 2 \le p <6, \label{eq:1.4}\\
	\| \nabla^{\alpha} \partial_{t}  u(t) \|_{L^{\infty}(\R^{3})} & 
	\le C(1+t)^{-\frac{7}{4}} t^{-\frac{1}{4}-\frac{\alpha}{2}},\quad 0 \le \alpha \le 1,    \label{eq:1.5} \\
	\| \partial_{t}^{2}  u(t) \|_{L^{p}(\R^{3})} & 
	\le C (1+t)^{-\frac{5}{4}+\frac{1}{p}} t^{-\frac{5}{4}+\frac{3}{2p}}, \quad 2 \le p <6   \label{eq:1.6} 
\end{align}
for $t>0$. 
\end{thm}
\begin{rem}
	Theorems \ref{thm:1.1} and \ref{thm:1.3} improves the results of \cite{Ponce} in the following sense.
	The estimates \eqref{eq:1.2} and \eqref{eq:1.3} yield
	\begin{equation}  \label{eq:1.999}
\| \nabla u(t) \|_{L^{q}(\R^{3})} \le C(1+t)^{-\frac{3}{2}(1-\frac{1}{q})+\frac{1}{q}-\frac{1}{2}}, \quad  2 \le q \le \infty
\end{equation}
	by the interpolation, which shows the faster decay properties than the estimate \eqref{eq:1.333}.
	We also note that the estimate \eqref{eq:1.3} is not obtained by the Sobolev embedding $ \| g \|_{L^{\infty}(\R^{3})} \le C \| g \|_{L^{2}(\R^{3})}^{\frac{1}{4}} \| \nabla^{2} g \|_{L^{2}(\R^{3})}^{\frac{3}{4}}$ and suggests the hyperbolic aspect of \eqref{eq:1.1}.
	The proof is based on the estimates of the fundamental solutions to \eqref{eq:1.1}, which will be presented in later.
Moreover our assumption for the regularity of initial data is slightly weaker than \cite{Ponce}. 
And, as is seen in \eqref{eq:1.2}-\eqref{eq:1.6}, 
we sharply estimate all norms of the function spaces which $u$ belongs to.
This fact means that our assumption doesn't have extra regularity for the initial data.
We close this remark with an explanation of the hyperbolic aspect of \eqref{eq:1.1}. 
In this paper we consider the Cauchy problem \eqref{eq:1.1} with \eqref{eq:1.222} in the framework of $L^{2}$-Sobolev spaces; 
and we establish quantitative estimates for smoothing and decay properties which reflect both the damping aspect (i.e., the parabolic aspect) and the dispersive aspect (i.e., the hyperbolic aspect) of the system \eqref{eq:1.1}. 
Indeed,
for example, 
we observe the hyperbolic aspect of the system \eqref{eq:1.1} in the diffusion waves $G_{0}^{(\beta)}(t,x)$ and $G_{1}^{(\beta)}(t,x)$,
which appear in the asymptotic profiles of the solutions to \eqref{eq:1.1}. 
Here $G_{j}^{(\beta)}(t,x)$ for $j=0,1$ are defined by
\begin{equation} \label{eq:1.7}
\begin{split}
G_{j}^{(\beta)}(t,x)
:=\mathcal{F}^{-1}[
	\mathcal{G}^{(\beta)}_{j}(t,\xi)
	]
\end{split}
\end{equation}
and 
\begin{equation*} 
\begin{split}
\mathcal{G}^{(\beta)}_{0}(t,\xi) := e^{-\frac{\nu |\xi|^{2}}{2}t} \cos (\beta |\xi| t), \quad
\mathcal{G}^{(\beta)}_{1}(t,\xi) := e^{-\frac{\nu |\xi|^{2}}{2}t} \frac{\sin (\beta |\xi| t)}{ \beta |\xi| }
\end{split}
\end{equation*}
with the parameter $\beta>0$.
Hoff and Zumbrun \cite{H-Z1} considered the compressible Navier-Stokes equations whose linearized semigroup at the motionless state takes a similar form to the fundamental solution for the linearized equation for \eqref{eq:1.1}; 
and it was shown in \cite{H-Z1} that a difference between the diffusion waves $G_{j}^{(\beta)}(t,x)$ and the heat kernel 
(i.e., $G_{j}^{(\beta)}(t,x)$ with $\beta=0$) appears in quantitative estimates in $L^{p}$-norms for $p \neq 2$. 
See also \cite{H-Z1, H-Z2, K-S, Shibata}.
We apply the observations on the diffusion waves in \cite{H-Z1, H-Z2, K-S, Shibata} and establish our improvement of the decay property \eqref{eq:1.999} for the quasi-linear hyperbolic system \eqref{eq:1.1}. 
	%
\end{rem}
\begin{rem}
We obtain the estimates \eqref{eq:1.3}-\eqref{eq:1.6}, which indicates the smoothing effect of the global solution 
in the sense of integrability.
We emphasize that this regularity gain isn't easily expected since our nonlinear interaction is given by quasi-linear,  
even if the linear principal part includes the viscoelastic term $-\nu \Delta \partial_{t} u$.
Indeed, as is well-known, hyperbolic problem with quasi-linear nonlinear term may cause ``derivative-loss'' to the solution.

For the proof of \eqref{eq:1.3}-\eqref{eq:1.6}, 
the basic idea is to use both parabolic aspect and hyperbolic aspect of the fundamental solutions to \eqref{eq:1.1}. 
More precisely,
two types of the estimates for the high frequency parts of the fundamental solutions to \eqref{eq:1.1}, 
$L^{p}$-$L^{p}$ type estimates (Proposition \ref{prop:4.1}) and $L^{\infty}$-$L^{2}$ type estimates (Lemma \ref{Lem:4.2}), 
play essential role.
$L^{p}$-$L^{p}$ type estimates are shown by the combination of the parabolic smoothing  (i.e. parabolic aspect) 
and use of the cancellation in the integration by parts by the oscillation integral (i.e. hyperbolic aspect), 
while $L^{\infty}$-$L^{2}$ type estimates come from the parabolic smoothing  (i.e. parabolic aspect) and the H\"{o}lder inequality 
with the fact $|\xi|^{-2} \in L^{2}(|\xi| \ge 1)$ in $\R^{3}$. 
We also note that the regularity of the solution guaranteed in Theorem \ref{thm:1.3} is essential to obtain Theorem \ref{thm:1.5}.

As for the smoothing effect, 
we mention an interesting work by Ghisi, Gobbino and Haraux \cite{GGH},
where abstract linear hyperbolic equations with strong damping are considered in a Hilbert space setting and the smoothing effect of solutions are described by the fractional power of a nonnegative self-adjoint operator. 

\end{rem}
We next claim the approximation formulas of the global solutions $u$, obtained in Theorems \ref{thm:1.1}, 
in the topology observed in Theorem \ref{thm:1.3}.
For this purpose, we denote $\mathcal{I}_{3} \in M(\R;3)$ is the identity matrix and 
\begin{equation} \label{eq:1.8}
\mathcal{P}:=\displaystyle\frac{\xi}{|\xi|} \otimes \frac{\xi}{|\xi|}.
\end{equation}
We also define the 3-d valued constant vectors depending on the initial data and the nonlinear term as follows. 
\begin{equation*}
\begin{split}
m_{j}= {}^t\! (m_{j1}, m_{j2}, m_{j3}), \quad M[u]:={}^t\! (M_{1}[u], M_{2}[u], M_{3}[u]), 
\end{split}
\end{equation*}
where 
\begin{equation*}
\begin{split}
m_{0k}:= \displaystyle\int_{\R^{3}} \nabla f_{0k}(x) dx, \quad m_{1k}:= \displaystyle\int_{\R^{3}} f_{1k}(x) dx
\end{split}
\end{equation*}
and 
\begin{equation*}
\begin{split}
M_{k}[u] := \displaystyle\int_{0}^{\infty} \int_{\R^{3}} F_{k}(u)(\tau, y) dy\, d \tau 
\end{split}
\end{equation*}
for $k=1,2,3$.
Using the above notation, we define the functions $G$, $H$ and $\tilde{G}$ by
\begin{equation} \label{eq:1.9}
\begin{split}
 G(t,x):= 
& \nabla^{-1} \mathcal{F}^{-1} \left[ 
\left( \mathcal{G}^{(\sqrt{\lambda+2 \mu})}_{0}(t,\xi) -\mathcal{G}^{(\sqrt{\mu})}_{0}(t,\xi) \right)\mathcal{P} 
+
\mathcal{G}^{(\sqrt{\mu})}_{0}(t,\xi) \mathcal{I}_{3}
\right] m_{0} \\
& +
\mathcal{F}^{-1} \left[ 
\left( \mathcal{G}^{(\sqrt{\lambda+2 \mu})}_{1}(t,\xi) -\mathcal{G}^{(\sqrt{\mu})}_{1}(t,\xi) \right)\mathcal{P} 
+
\mathcal{G}^{(\sqrt{\mu})}_{1}(t,\xi) \mathcal{I}_{3}
\right] (m_{1}+M[u]),
\end{split}
\end{equation}
\begin{equation} \label{eq:1.10}
\begin{split}
& H(t,x):= \\ 
& \nabla^{-1} \mathcal{F}^{-1} \left[ 
\left( (\lambda+2 \mu)  
\mathcal{G}^{(\sqrt{\lambda+2 \mu})}_{1}(t,\xi) 
-\mu \mathcal{G}^{(\sqrt{\mu} )}_{1}(t,\xi) 
\right)\mathcal{P} 
+ \mu \mathcal{G}^{(\sqrt{\mu} )}_{1}(t,\xi) \mathcal{I}_{3}
\right] m_{0} \\
& +
\mathcal{F}^{-1} \left[ 
\left(  \mathcal{G}^{(\sqrt{\lambda+2 \mu})}_{0}(t,\xi) -\mathcal{G}^{(\sqrt{\mu} )}_{0}(t,\xi) \right) \mathcal{P} 
+
\mathcal{G}^{(\sqrt{\mu} )}_{0}(t,\xi) \mathcal{I}_{3}
\right]  (m_{1}+M[u])
\end{split}
\end{equation}
and
\begin{equation} \label{eq:1.11}
	\begin{split}
		& \tilde{G}(t,x) := \\
		& -\Delta \nabla^{-1}
		\mathcal{F}^{-1} \left[
		\left(
		(\lambda+2 \mu) 
		\mathcal{G}^{(\sqrt{\lambda+2 \mu})}_{0}(t,\xi) 
		- \mu  \mathcal{G}^{(\sqrt{\mu} )}_{0}(t,\xi) 
		\right)\mathcal{P} 
		+ \mu  
		\mathcal{G}^{(\sqrt{\mu} )}_{0}(t,\xi) \mathcal{I}_{3}
		\right]  m_{0} \\
		& -\Delta
		\mathcal{F}^{-1} \biggl[
		\left(
		(\lambda+2 \mu) 
		\mathcal{G}^{(\sqrt{\lambda+2 \mu})}_{1}(t,\xi) 
		- \mu  \mathcal{G}^{(\sqrt{\mu} )}_{1}(t,\xi) 
		\right)\mathcal{P} \\
		&		
		+ \mu  
		\mathcal{G}^{(\sqrt{\mu} )}_{1}(t,\xi) \mathcal{I}_{3}
		\biggr] (m_{1}+M[u]),
	\end{split}
\end{equation}
respectively,
where $\mathcal{F}^{-1}$ represents the Fourier inverse transform. 
Here we formulate the asymptotic behavior of the solution of \eqref{eq:1.1} as $t \to \infty$.

\begin{thm} \label{thm:1.5}
The global solution $u(t)$ of \eqref{eq:1.1} constructed in Theorem \ref{thm:1.1} satisfies
\begin{align}
	& \| \nabla^{\alpha} (u(t)-G(t)) \|_{L^{2}(\R^{3})} = o(t^{-\frac{1}{4}-\frac{\alpha}{2}}), \quad 1 \le \alpha \le 3, \label{eq:1.12} \\
	& \| \nabla^{\alpha} (u(t)-G(t)) \|_{L^{\infty}(\R^{3})} =o(t^{-\frac{3}{2}-\frac{\alpha}{2}}), \quad 0 \le \alpha \le 1, \label{eq:1.13}  \\
	& \| \nabla^{\alpha} (\partial_{t}u(t) -H(t)) \|_{L^{2}(\R^{3})} =o(t^{-\frac{3}{4}-\frac{\alpha}{2}}), \quad 0 \le \alpha \le 2, \label{eq:1.14}  \\
& \| \nabla^{2} (\partial_{t} u(t) -H(t)) \|_{L^{p}(\R^{3})} =o( t^{-\frac{5}{2}(1-\frac{1}{p})-\frac{1}{2}} ), \quad 2 \le p <6, \label{eq:1.15}  \\
& \| \nabla^{\alpha} (\partial_{t} u(t) -H(t))\|_{L^{\infty}(\R^{3})}
=o( t^{-2-\frac{\alpha}{2}}),\quad 0 \le \alpha \le 1, \label{eq:1.16}  \\
& \| \partial_{t}^{2} u(t) -\tilde{G}(t) \|_{L^{p}(\R^{3})}  =o( t^{-\frac{5}{2}(1-\frac{1}{p})} ), \quad 2 \le p <6 \label{eq:1.17} 
\end{align}
as $t \to \infty$.
\end{thm}
\begin{rem}
We can derive the corresponding results for the nonlinear term $F(u)= \nabla u \nabla \partial_{t} u$.
More precisely, for $(f_{0}, f_{1}) \in \{ \dot{H}^{3} \cap \dot{W}^{1,1} \}^{3} \times  \{ H^{2} \cap {L}^{1} \}^{3}$ with sufficiently small norms,
we have the existence of global solutions with decay properties, and smoothing effect of global solutions and their asymptotic profiles as $t \to \infty$. 
It is worth pointing out that in this case, we conclude that 
\begin{equation*}
		\begin{split}
			u \in \{ C^{1}( (0,\infty); \dot{W}^{2,6}) 
			\cap W^{1, \infty}( 0,\infty; W^{1,\infty}) \cap C^{2} ((0,\infty);  L^{6} ) \}^{3},
		\end{split}
	\end{equation*}
which is the reason why we separate the results corresponding to the nonlinear terms $F(u)= \nabla u \nabla^{2} u$ and $F(u)= \nabla u \nabla \partial_{t} u$.
On the other hand, the proof is parallel and we only focus on the case $F(u)= \nabla u \nabla^{2} u$ in this paper.  
\end{rem}

This paper is organized as follows.
In section 2, we firstly review the well-known facts on fundamental solutions to the strongly damped linear wave equations and wave equations. 
We also introduce notation and useful estimates used throughout this paper.
Sections 3 and 4 are devoted to the derivation of time decay properties of the fundamental solutions to \eqref{eq:1.1}. 
Our main results, Theorems \ref{thm:1.1}, \ref{thm:1.3} and \ref{thm:1.5} are proved in sections 5-7. 
For the completeness of the proof, we will show the estimate \eqref{eq:3.36} as an appendix.
\section{Preliminaries}
\subsection{Notation}
We shall explain our basic notation.
Let $\hat{f}$ denote the Fourier transform of $f$
defined by
\begin{align*}
\hat{f}(\xi) := (2 \pi)^{-\frac{3}{2}}
\int_{\R^{3}} e^{-i x \cdot \xi} f(x) dx.
\end{align*}
Also, let $\mathcal{F}^{-1}[f]$ or $\check{f}$ denote the inverse
Fourier transform.

The norm of $f \in L^{p}(\R^{3})$ is denoted by $\| f \|_{p}$ for $1 \le p \le \infty$. 
For $k \ge 0$ and $1 \le p \le \infty$, let $W^{k,p}(\R^{3})$ be the usual Sobolev spaces
\begin{equation*}
W^{k,p}(\R^{3})
  :=\Big\{ f:\R^{3} \to \R;
        \| f \|_{W^{k,p}(\R^{3})} 
        := \| f \|_{p}+
         \|  \nabla_{x}^{k} f \|_{p}< \infty 
     \Big\}.
\end{equation*}
When $p=2$, we denote $W^{k,2}(\R^{3}) = H^{k}(\R^{3})$.
For the notation of the function spaces, the domain $\R^{3}$ is often abbreviated.

We will denote by $\dot{W}^{k,p}$ and $\dot{H}^{k}$
the corresponding homogeneous Sobolev spaces, respectively.  

\subsection{Estimates for the strongly damped linear wave equations}
In this subsection, we summarize the estimates for the fundamental solutions of the strong damped wave equation:
\begin{equation} \label{eq:2.1}
\left\{
\begin{split}
& \partial_{t}^{2} w -\beta^{2} \Delta w -\nu \Delta \partial_{t} w  = f(t,x), \quad t>0, \quad x \in \R^{3}, \\
& w(0,x)=w_{0}(x), \quad \partial_{t} w(0,x)=w_{1}(x) , \quad x \in \R^{3}, 
\end{split}
\right.
\end{equation}
where $w=w(t,x):(0,\infty)\times \R^{3}\to \R$ and $\beta>0$.
Now introducing the characteristic roots as
\begin{equation*}
\sigma_{\pm}^{(\beta)}:=\frac{-\nu|\xi|^{2} \pm \sqrt{\nu^{2}|\xi|^{4}-4 \beta^{2} |\xi|^{2}}}{2},
\end{equation*}
we define the functions by  
\begin{equation*}
\begin{split}
\mathcal{K}_{0}^{(\beta)}(t,\xi):= 
\frac{
	-\sigma_{-}^{(\beta)}e^{\sigma_{+}^{(\beta)}t}+\sigma_{+}^{(\beta)} e^{\sigma_{-}^{(\beta)}t}
		}{\sigma_{+}^{(\beta)}-\sigma_{-}^{(\beta)}},
\end{split}
\end{equation*}
\begin{equation*}
\begin{split}
	\mathcal{K}_{1}^{(\beta)}(t,\xi):= 
	\frac{
		e^{\sigma_{+}^{(\beta)}t}-e^{\sigma_{-}^{(\beta)}t}
	}{\sigma_{+}^{(\beta)}-\sigma_{-}^{(\beta)}}.
\end{split}
\end{equation*}
We also define the smooth cut-off functions as follows;
$\chi_{j}= \chi_{j}(\xi) \in C^{\infty}(\R^{3})$ $(j=L,M,H)$ satisfying 
\begin{equation*}
\begin{split}
\chi_{L}:= 
\begin{cases}
& 1 \quad (|\xi|\le \frac{c_{0}}{2}), \\
& 0 \quad (|\xi|\ge c_{0}),
\end{cases}
\end{split}
\end{equation*}
\begin{equation*}
\begin{split}
\chi_{H}:= 
\begin{cases}
& 0 \quad (|\xi|\le c_{1}), \\
& 1 \quad (|\xi|\ge 2 c_{1})
\end{cases}
\end{split}
\end{equation*}
and 
\begin{equation*}
\chi_{M}=1-\chi_{L}-\chi_{H}
\end{equation*}
for some $c_{0}>0$ with sufficiently small and $c_{1}>0$, which we choose later.
For simplicity of the notation, we use the evolution operators defined by 
\begin{equation} \label{eq:2.2}
\begin{split}
K_{j}^{(\beta)}(t)g & := \mathcal{F}^{-1}[
\mathcal{K}_{j}^{(\beta)}(t, \xi)\hat{g} 
], \\
K_{jk}^{(\beta)}(t)g & := \mathcal{F}^{-1}[
\mathcal{K}_{j}^{(\beta)}(t,\xi)\chi_{k} \hat{g} 
], \\
G_{jk}^{(\beta)}(t) \ast g & := \mathcal{F}^{-1}[
\mathcal{G}_{j}^{(\beta)}(t,\xi)\chi_{k} \hat{g} 
]
\end{split}
\end{equation}
for $j=0,1$ and $k=L,M,H$.

Then we immediately have the representation formula to the solution of \eqref{eq:2.1}.
\begin{lem} \label{Lem:2.1}
	The solution $w$ of single damped wave equation \eqref{eq:2.1} satisfies 
	\begin{equation} \label{eq:2.3}
	\hat{w}(t) = \mathcal{K}_{0}^{(\beta)}(t, \xi) \hat{w}_{0} +\mathcal{K}_{1}^{(\beta)}(t, \xi)\hat{w}_{1} 
	+ \int_{0}^{t} \mathcal{K}_{1}^{(\beta)}(t-\tau, \xi) \hat{f}(\tau) d \tau.
	\end{equation}
\end{lem}
The representation formulas of the low frequency parts of the fundamental solution is important to derive the asymptotic profiles of the solutions for \eqref{eq:1.1}.
\begin{lem}[\cite{Ponce}, \cite{Shibata}, \cite{D-R}, \cite{I-T}] \label{Lem:2.3}
	The fundamental solutions of \eqref{eq:2.1}, $\mathcal{K}_{0}^{(\beta)}(t,\xi)$ and
	$\mathcal{K}_{1}^{(\beta)}(t,\xi)$, 
	satisfy the following representation formula in ${\rm supp}\, \chi_{L}$: 
	\begin{equation} \label{eq:2.4}
	\begin{split}
	& \mathcal{K}_{0}^{(\beta)}(t,\xi)\chi_{L} = \nu |\xi|^{2} \mathcal{K}_{1}^{(\beta)}(t,\xi)\chi_{L}  
	+\mathcal{K}_{00}^{(\beta)}(t,\xi)\chi_{L}, 
	\end{split}
	\end{equation}
	\begin{equation} \label{eq:2.5}
		\begin{split} 
			& \mathcal{K}_{1}^{(\beta)}(t,\xi)\chi_{L} =
			\frac{
				e^{-\frac{\nu|\xi|^{2} t}{2}} \sin (t \beta |\xi| \phi_{\nu,\beta})
			}{\beta |\xi| \phi_{\nu, \beta}} \chi_{L},
		\end{split}
	\end{equation}
	where
	$
	\phi_{\nu,\beta}=\phi_{\nu,\beta}(\xi):= \sqrt{1-\frac{\nu^{2} |\xi|^{2}}{4 \beta^{2}}}
	$ and 
	\begin{equation} \label{eq:2.6}
		\begin{split}
			& \mathcal{K}_{00}^{(\beta)}(t,\xi):=e^{-\frac{\nu|\xi|^{2} t}{2}} \cos (t \beta |\xi| \phi_{\nu,\beta}). 
		\end{split}
	\end{equation}
\end{lem}
The decay properties of the fundamental solutions \eqref{eq:2.1} are described as follows:
\begin{lem}[\cite{Ponce}, \cite{Shibata}, \cite{K-S}] \label{Lem:2.2}
	Let $1 \le q \le p \le \infty$, $\ell \ge \tilde{\ell} \ge 0$ and $\alpha \ge \tilde{\alpha} \ge 0$. 
	Then it holds that
	\begin{align} 
	& 
	\left\|
	\partial_{t}^{\ell} 
	\nabla^{\alpha}
	K_{0L}^{(\beta)}(t)g
	\right\|_{p} 
	\le C(1+t)^{-\frac{3}{2}(\frac{1}{q}-\frac{1}{p})-(\frac{1}{q}-\frac{1}{p})+\frac{1}{2} -\frac{\ell-\tilde{\ell}+ \alpha-\tilde{\alpha}}{2}}
	\| \nabla^{\tilde{\alpha}+\tilde{\ell}} g \|_{q}, \label{eq:2.7}  \\
	& \left\| 
	\partial_{t}^{\ell} 
	\nabla^{\alpha}
	K_{1L}^{(\beta)}(t) g
	\right\|_{p} 
	\le C(1+t)^{-\frac{3}{2}(\frac{1}{q}-\frac{1}{p})-(\frac{1}{q}-\frac{1}{p})+1 -\frac{\ell-\tilde{\ell}+ \alpha-\tilde{\alpha}}{2}}
	\| \nabla^{\tilde{\alpha}+\tilde{\ell}} g \|_{q}, \label{eq:2.8} \\ 
		& 
		\left\|
		\partial_{t}^{\ell} 
		\nabla^{\alpha}
		G_{0L}^{(\beta)}(t) \ast g
		\right\|_{p} 
		\le C(1+t)^{-\frac{3}{2}(\frac{1}{q}-\frac{1}{p})-(\frac{1}{q}-\frac{1}{p})+\frac{1}{2} -\frac{\ell-\tilde{\ell}+ \alpha-\tilde{\alpha}}{2}}
		\| \nabla^{\tilde{\alpha}+\tilde{\ell}} g \|_{q}, \label{eq:2.9}  \\
		& \left\| 
		\partial_{t}^{\ell} 
		\nabla^{\alpha}
		G_{1L}^{(\beta)}(t) \ast g
		\right\|_{p} 
		\le C(1+t)^{-\frac{3}{2}(\frac{1}{q}-\frac{1}{p})-(\frac{1}{q}-\frac{1}{p})+1 -\frac{\ell-\tilde{\ell}+ \alpha-\tilde{\alpha}}{2}}
		\| \nabla^{\tilde{\alpha}+\tilde{\ell}} g \|_{q} \label{eq:2.10} 
	\end{align}
	for $t \ge 0$.
\end{lem}
An easy computation shows that the middle and high frequency part of for $G_{0}^{(\beta)}(t,x)$ and $G_{1}^{(\beta)}(t,x)$ and their Riesz transform decays sufficiently fast so that it does not effect the asymptotic profiles.
This fact is summarized as follows;
\begin{lem} \label{Lem:2.4}
	Let $\alpha, \ell \ge 0$, $1 \le p \le \infty$ and $t>0$.
	Then it holds that 
	\begin{equation} \label{eq:2.11}
		\begin{split}
			& \| 
			\partial^{\ell}_{t} \nabla^{\alpha} \mathcal{R}_{a} \mathcal{R}_{b} \mathcal{F}^{-1} [\mathcal{G}_{0}^{(\beta)}(t,\xi) (\chi_{M} +\chi_{H}) ]\|_{p} 
			+
			\| \partial^{\ell}_{t} \nabla^{\alpha}  \mathcal{F}^{-1} [\mathcal{G}_{0}^{(\beta)}(t,\xi) (\chi_{M} +\chi_{H}) ]\|_{p} \\
			& \le C e^{-ct} t^{-\frac{3}{2}(1-\frac{1}{p})-\frac{\alpha+\ell}{2}}
		\end{split}
	\end{equation}
	and 
	\begin{equation} \label{eq:2.12}
		\begin{split}
			& \|
			\partial^{\ell}_{t} \nabla^{\alpha} \mathcal{R}_{a} \mathcal{R}_{b} \mathcal{F}^{-1} [\mathcal{G}_{1}^{(\beta)}(t,\xi) (\chi_{M} +\chi_{H}) ]\|_{p} 
			+\|
			\partial^{\ell}_{t} \nabla^{\alpha} \mathcal{F}^{-1} [\mathcal{G}_{1}^{(\beta)}(t,\xi) (\chi_{M} +\chi_{H}) ]\|_{p} \\
			& \le C e^{-ct} t^{-\frac{3}{2}(1-\frac{1}{p})-\frac{\alpha+\ell}{2}},
		\end{split}
	\end{equation}
	where 
	\begin{equation*} 
		\begin{split}
			\mathcal{R}_{a} g:= \mathcal{F}^{-1} \left[\frac{\xi_{a}}{|\xi|} \hat{g} \right]
		\end{split}
	\end{equation*}
	for $a=1,2,3$.
\end{lem}
%

\subsection{Estimates for the linear wave equation}
In this subsection, 
we recall well-known estimates for the solutions to the Cauchy problem of wave equations:  
\begin{equation} \label{eq:2.13}
\left\{
\begin{split}
& \partial_{t}^{2} w -\beta^{2} \Delta w   = 0, \quad t>0, \quad x \in \R^{3}, \\
& w(0,x)=w_{0}(x), \quad \partial_{t} w(0,x)=w_{1}(x) , \quad x \in \R^{3}, 
\end{split}
\right.
\end{equation}
where $w=w(t,x):(0,\infty)\times \R^{3}\to \R$ and $\beta>0$.
Now we define the fundamental solutions to \eqref{eq:2.13};
\begin{equation*} 
W^{(\beta)}_{0}(t)g := \mathcal{F}^{-1}[\cos (t \beta |\xi|) \hat{g}],
\end{equation*}
\begin{equation*}
W^{(\beta)}_{1}(t)g := \mathcal{F}^{-1}\left[\frac{\sin (t \beta |\xi|)}{\beta |\xi|} \hat{g} \right].
\end{equation*}
We firstly state estimates for the $L^{\infty}$ norms of the fundamental solutions of \eqref{eq:2.13}. 
\begin{lem} \label{Lem:2.5}
	There exists a constant $C>0$ such that 
	\begin{equation} \label{eq:2.14}
	\| W^{(\beta)}_{0}(t)g \|_{\infty} \le Ct^{-1} \| g \|_{\dot{W}^{2,1}}, 
	\end{equation}
	\begin{equation} \label{eq:2.15}
	\| W^{(\beta)}_{1}(t)g \|_{\infty} \le Ct^{-1} \| g \|_{\dot{W}^{1,1}}
	\end{equation}
	for $t>0$.
\end{lem}
Lemma \ref{Lem:2.5} is well-known.
For the proof, see e.g. \cite{S-S}. 

The following lemma states the $L^{p}$ boundedness of the fundamental solutions to \eqref{eq:2.13}. 
\begin{lem}
	Let $1 \le p \le \infty$, $\alpha$, $\ell \ge 0$ and $\gamma>0$. There exists a constant $C>0$ such that 
	\begin{equation} \label{eq:2.16}
		\| \partial_{t}^{\ell} W^{(\beta)}_{0}(t)g \|_{p} \le C ( \| \nabla^{\ell} g \|_{p}+t \| \nabla^{\ell+1} g \|_{p}), 
	\end{equation} 
	\begin{equation} \label{eq:2.17}
		\| \partial_{t}^{\ell}  W^{(\beta)}_{1}(t)g \|_{p} \le C t \| \nabla^{\ell} g \|_{p}
	\end{equation}
	for $t >0$.
\end{lem}
	The estimates \eqref{eq:2.16}-\eqref{eq:2.17} are also well-known and the direct consequence of the representation formula of the fundamental solutions (cf.\cite{E}).
%
%
\subsection{Useful estimates}
In this subsection, we recall the basic estimates, 
which are frequently used in what follows. 
We begin with the $L^{p}$-$L^{p}$ boundedness of the Riesz transform. 
\begin{lem}
	Let $1<p<\infty$. 
	There exists $C>0$ such that 
	\begin{equation} \label{eq:2.18}
		\begin{split}
			\| \mathcal{R}_{a} g\|_{p} \le C \| g \|_{p}.
		\end{split}
	\end{equation}
\end{lem}
For the proof, see e.g. \cite{Gr}.

The following estimate is useful to obtain the decay properties of low frequency parts of our problem.

\begin{lem}[\cite{K-S}]
	Let $\ell \ge 0$, $\alpha \ge 0$ and $\ell+\alpha \ge 1$. 
	There exists a constant $C>0$ such that 
	\begin{equation} \label{eq:2.19}
		\begin{split}
			\| \partial_{t}^{\ell} \nabla^{\alpha} \mathcal{R}_{a}\mathcal{R}_{b} \mathcal{F}^{-1}[e^{-\frac{\nu t |\xi|^{2}}{2}} \chi_{L}] \|_{1} \le C(1+t)^{-\frac{\alpha}{2}-\ell}
		\end{split}
	\end{equation}
	for $t \ge 0$.
\end{lem}
Finally, we summarize the interpolation estimates.
\begin{lem} \label{lem:2.9}
	There exists a constant $C>0$ such that 
	\begin{align}
			& \| g \|_{L^{1}(\R^{3})} \le C \| g \|_{L^{2}(\R^{3})}^{\frac{1}{4}} \| x^{2} u \|_{L^{2}(\R^{3})}^{\frac{3}{4}}
			=C \| \hat{g} \|_{L^{2}(\R^{3})}^{\frac{1}{4}} \| \nabla_{\xi}^{2} \hat{g} \|_{L^{2}(\R^{3})}^{\frac{3}{4}},  \label{eq:2.20} \\
			& \| g \|_{L^{\infty}(\R^{3})} \le C \| g \|_{L^{2}(\R^{3})}^{\frac{1}{4}} \| \nabla^{2} g \|_{L^{2}(\R^{3})}^{\frac{3}{4}}, \label{eq:2.21} \\
			& \| \nabla g \|_{L^{2p}(\R^{3})} \le C \| g \|_{L^{\infty}(\R^{3})}^{\frac{1}{2}} \| \nabla^{2} g \|_{L^{p}(\R^{3})}^{\frac{1}{2}}, \quad 1\le p< \infty, \label{eq:2.22} \\
		& \|  g \|_{L^{6}(\R^{3})} \le C \| \nabla g \|_{L^{2}(\R^{3})},  \label{eq:2.23} \\
		&  \| \nabla g \|_{L^{2}(\R^{3})} \le C \| \nabla g \|_{L^{1}(\R^{3})}+ C \| \nabla^{3} g \|_{L^{2}(\R^{3})}, \label{eq:2.24}
	\end{align}
	where $C$ is independent of $g$.
\end{lem}
\begin{proof}
	The proof of \eqref{eq:2.20}-\eqref{eq:2.23} is well-known. 
	See e.g. \cite{BTW} for \eqref{eq:2.20} and \cite{C} for \eqref{eq:2.21}-\eqref{eq:2.23}. 
	Here we only show the estimate \eqref{eq:2.24}.
	By the Plancherel formula and the Hausdorff-Young inequality, we see that 
	\begin{equation*}
		\begin{split}
			\| \nabla g \|_{2} & \le  C \| \xi \chi_{L} \hat{g} \|_{2} +\| \xi (\chi_{M} +\chi_{H})  \hat{g} \|_{2} \le  C \| \mathcal{F}^{-1} [\chi_{L}] \ast \mathcal{F}^{-1} [\xi \hat{g}] \|_{2} 
			+\| \xi^{3} (\chi_{M} +\chi_{H})  \hat{g} \|_{2} \\
			& \le C \| \mathcal{F}^{-1} [\chi_{L}]  \|_{2} \|\mathcal{F}^{-1} [\xi \hat{g}] \|_{1} + \| \xi^{3} \hat{g} \|_{2} \\
			& \le C \| \nabla g \|_{1} + \| \nabla^{3} g \|_{2},
		\end{split}
	\end{equation*}
	which is the desired estimate.
	We complete the proof of Lemma \ref{lem:2.9}.
\end{proof} 
\section{Linear estimates for low frequency parts}
\subsection{Approximation of the low frequency parts}
This subsection is devoted to the proof of approximation formulas of the evolution operators $K_{0L}^{(\beta)}(t)$ and $K_{1L}^{(\beta)}(t)$ with the Riesz transforms.
For this aim, we introduce the notation 
\begin{align}
\mathbb{K}_{00}^{(\beta)}(t,x) & := \mathcal{R}_{a} \mathcal{R}_{b} \mathcal{F}^{-1}[\mathcal{K}_{00L}^{(\beta)}(t,\xi)\chi_{L} ], \label{eq:3.1} \\
\mathbb{K}_{1}^{(\beta)}(t,x) & :=\mathcal{R}_{a} \mathcal{R}_{b} \mathcal{F}^{-1}[\mathcal{K}_{1L}^{(\beta)}(t,\xi)\chi_{L} ], \label{eq:3.2}  \\ 
\mathbb{G}_{j}^{(\beta)}(t,x) & :=\mathcal{R}_{a} \mathcal{R}_{b} \mathcal{F}^{-1}[\mathcal{G}_{jL}^{(\beta)}(t,\xi)\chi_{L} ] \label{eq:3.3} 
\end{align}
for $j=0,1$ and $a,b=1,2,3$.
Now we formulate the main result of this subsection, which states that the large time behavior of the functions $\mathbb{K}_{00}^{(\beta)}(t,x)$ and 
$\mathbb{K}_{1}^{(\beta)}(t,x)$ are described by $\mathbb{G}_{0}^{(\beta)}(t,x)$ and $\mathbb{G}_{1}^{(\beta)}(t,x)$, depending on the order of the time derivatives.
\begin{prop} \label{Prop:3.1}
	Let $\alpha\ge \tilde{\alpha} \ge 0$, $\ell \ge \tilde{\ell} \ge 0$, $m \ge 0$, $1 \le q \le p\le \infty$ and $t \ge 0$.
	Then it holds that 
	\begin{equation} \label{eq:3.4}
		\begin{split}
			& \left\| \nabla^{\alpha} \left(
			\partial^{\ell}_{t} \mathbb{K}_{00}^{(\beta)}(t) \ast g -(-1)^{\frac{\ell}{2}} \beta^{\ell} \nabla^{\ell} \mathbb{G}_{0}^{(\beta)}(t) \ast g  
			\right)
			\right\|_{p} \\
			& \le C(1+t)^{-\frac{3}{2}(\frac{1}{q}-\frac{1}{p})-(\frac{1}{q}-\frac{1}{p})- \frac{\ell-\tilde{\ell}+ \alpha-\tilde{\alpha}}{2}} \| \nabla^{\tilde{\alpha}+\tilde{\ell}} g \|_{q}
		\end{split}
	\end{equation}
	for $\ell=2m$,
	\begin{equation} \label{eq:3.5}
		\begin{split}
			& \left\| \nabla^{\alpha} \left(
			\partial^{\ell}_{t} \mathbb{K}_{00}^{(\beta)}(t) \ast g -(-1)^{\frac{\ell+1}{2}} \beta^{\ell+1} \nabla^{\ell+1} \mathbb{G}_{1}^{(\beta)}(t)\ast g 
			\right)
			\right\|_{p} \\
			& \le C(1+t)^{-\frac{3}{2}(\frac{1}{q}-\frac{1}{p})-(\frac{1}{q}-\frac{1}{p})- \frac{\ell-\tilde{\ell}+ \alpha-\tilde{\alpha}}{2}} \| \nabla^{\tilde{\alpha}+\tilde{\ell}} g \|_{q}
		\end{split}
	\end{equation}
	for $\ell=2m+1$,
	\begin{equation} \label{eq:3.6}
		\begin{split}
			& \left\| \nabla^{\alpha} \left(
			\partial^{\ell}_{t} \mathbb{K}_{1}^{(\beta)}(t) \ast g - (-1)^{\frac{\ell}{2}} \beta^{\ell} \nabla^{\ell} \mathbb{G}_{1}^{(\beta)}(t) \ast g
			\right)
			\right\|_{p} \\ 
			& \le C(1+t)^{-\frac{3}{2}(\frac{1}{q}-\frac{1}{p})-(\frac{1}{q}-\frac{1}{p})+\frac{1}{2}- \frac{\ell-\tilde{\ell}+ \alpha-\tilde{\alpha}}{2}} \| \nabla^{\tilde{\alpha}+\tilde{\ell}} g \|_{q}
		\end{split}
	\end{equation}
	for $\ell=2m$ and 
	\begin{equation} \label{eq:3.7}
		\begin{split}
			& \left\| \nabla^{\alpha} \left(
			\partial^{\ell}_{t} \mathbb{K}_{1}^{(\beta)}(t) \ast g - (-1)^{\frac{\ell-1}{2}} \beta^{\ell-1} \nabla^{\ell-1} \mathbb{G}_{0}^{(\beta)}(t) \ast g
			\right)
			\right\|_{p} \\
			& \le C(1+t)^{-\frac{3}{2}(\frac{1}{q}-\frac{1}{p})-(\frac{1}{q}-\frac{1}{p})+\frac{1}{2}-\frac{\ell-\tilde{\ell}+ \alpha-\tilde{\alpha}}{2}} \| \nabla^{\tilde{\alpha}+\tilde{\ell}} g \|_{q}
		\end{split}
	\end{equation}
	for $\ell=2m+1$.
\end{prop}
Proposition \ref{Prop:3.1} follows from Lemmas \ref{Lem:3.2} and \ref{Lem:3.3} below 
by the virtue of the Riesz-Thorin interpolation theorem (cf. \cite{B}).
\begin{lem} \label{Lem:3.2}
	Let $\alpha\ge \tilde{\alpha} \ge 0$, $\ell \ge \tilde{\ell} \ge 0$, $m \ge 0$ and $t \ge 0$.
	Then it holds that 
	\begin{equation} \label{eq:3.8}
		\begin{split}
			& \left\| \nabla^{\alpha} \left(
			\partial^{\ell}_{t} \mathbb{K}_{00}^{(\beta)}(t) \ast g -(-1)^{\frac{\ell}{2}} \beta^{\ell} \nabla^{\ell} \mathbb{G}_{0}^{(\beta)}(t) \ast g  
			\right)
			\right\|_{\infty} \le C(1+t)^{-\frac{5}{2}- \frac{\ell-\tilde{\ell}+ \alpha-\tilde{\alpha}}{2}} \| \nabla^{\tilde{\alpha}+\tilde{\ell}} g \|_{1}
		\end{split}
	\end{equation}
	for $\ell=2m$,
	\begin{equation} \label{eq:3.9}
		\begin{split}
			& \left\| \nabla^{\alpha} \left(
			\partial^{\ell}_{t} \mathbb{K}_{00}^{(\beta)}(t) \ast g -(-1)^{\frac{\ell+1}{2}} \beta^{\ell+1} \nabla^{\ell+1} \mathbb{G}_{1}^{(\beta)}(t)\ast g 
			\right)
			\right\|_{\infty} \\
			& \le C(1+t)^{-\frac{5}{2}- \frac{\ell-\tilde{\ell}+ \alpha-\tilde{\alpha}}{2}} \| \nabla^{\tilde{\alpha}+\tilde{\ell}} g \|_{1}
		\end{split}
	\end{equation}
	for $\ell=2m+1$,
	\begin{equation} \label{eq:3.10}
		\begin{split}
			& \left\| \nabla^{\alpha} \left(
			\partial^{\ell}_{t} \mathbb{K}_{1}^{(\beta)}(t) \ast g - (-1)^{\frac{\ell}{2}} \beta^{\ell} \nabla^{\ell} \mathbb{G}_{1}^{(\beta)}(t) \ast g
			\right)
			\right\|_{\infty}  \le C(1+t)^{-2- \frac{\ell-\tilde{\ell}+ \alpha-\tilde{\alpha}}{2}} \| \nabla^{\tilde{\alpha}+\tilde{\ell}} g \|_{1}
		\end{split}
	\end{equation}
	for $\ell=2m$ and 
	\begin{equation} \label{eq:3.11}
		\begin{split}
			& \left\| \nabla^{\alpha} \left(
			\partial^{\ell}_{t} \mathbb{K}_{1}^{(\beta)}(t) \ast g - (-1)^{\frac{\ell-1}{2}} \beta^{\ell-1} \nabla^{\ell-1} \mathbb{G}_{0}^{(\beta)}(t) \ast g
			\right)
			\right\|_{\infty} \\
			& \le C(1+t)^{-2- \frac{\ell-\tilde{\ell}+ \alpha-\tilde{\alpha}}{2}} \| \nabla^{\tilde{\alpha}+\tilde{\ell}} g \|_{1}
		\end{split}
	\end{equation}
	for $\ell=2m+1$.
\end{lem}
%
\begin{lem} \label{Lem:3.3}
	Let $\alpha\ge \tilde{\alpha} \ge 0$, $\ell \ge \tilde{\ell} \ge 0$, $m \ge 0$, $1 \le p\le \infty$ and $t \ge 0$.
	Then it holds that 
	\begin{equation} \label{eq:3.12}
		\begin{split}
			& \left\| \nabla^{\alpha} \left(
			\partial^{\ell}_{t} \mathbb{K}_{00}^{(\beta)}(t) \ast g -(-1)^{\frac{\ell}{2}} \beta^{\ell} \nabla^{\ell} \mathbb{G}_{0}^{(\beta)}(t) \ast g  
			\right)
			\right\|_{p} \le C(1+t)^{- \frac{\ell-\tilde{\ell}+ \alpha-\tilde{\alpha}}{2}} \| \nabla^{\tilde{\alpha}+\tilde{\ell}} g \|_{p}
		\end{split}
	\end{equation}
	for $\ell=2m$,
	\begin{equation} \label{eq:3.13}
		\begin{split}
			& \left\| \nabla^{\alpha} \left(
			\partial^{\ell}_{t} \mathbb{K}_{00}^{(\beta)}(t) \ast g -(-1)^{\frac{\ell+1}{2}} \beta^{\ell+1} \nabla^{\ell+1} \mathbb{G}_{1}^{(\beta)}(t)\ast g 
			\right)
			\right\|_{p} \\
			& \le C(1+t)^{- \frac{\ell-\tilde{\ell}+ \alpha-\tilde{\alpha}}{2}} \| \nabla^{\tilde{\alpha}+\tilde{\ell}} g \|_{p}
		\end{split}
	\end{equation}
	for $\ell=2m+1$,
	\begin{equation} \label{eq:3.14}
		\begin{split}
			& \left\| \nabla^{\alpha} \left(
			\partial^{\ell}_{t} \mathbb{K}_{1}^{(\beta)}(t) \ast g - (-1)^{\frac{\ell}{2}} \beta^{\ell} \nabla^{\ell} \mathbb{G}_{1}^{(\beta)}(t) \ast g
			\right)
			\right\|_{p} \le C(1+t)^{\frac{1}{2}- \frac{\ell-\tilde{\ell}+ \alpha-\tilde{\alpha}}{2}} \| \nabla^{\tilde{\alpha}+\tilde{\ell}} g \|_{p}
		\end{split}
	\end{equation}
	for $\ell=2m$ and 
	\begin{equation} \label{eq:3.15}
		\begin{split}
			& \left\| \nabla^{\alpha} \left(
			\partial^{\ell}_{t} \mathbb{K}_{1}^{(\beta)}(t) \ast g - (-1)^{\frac{\ell-1}{2}} \beta^{\ell-1} \nabla^{\ell-1} \mathbb{G}_{0}^{(\beta)}(t) \ast g
			\right)
			\right\|_{p} \\
			& \le C(1+t)^{\frac{1}{2}- \frac{\ell-\tilde{\ell}+ \alpha-\tilde{\alpha}}{2}} \| \nabla^{\tilde{\alpha}+\tilde{\ell}} g \|_{p}
		\end{split}
	\end{equation}
	for $\ell=2m+1$.
\end{lem}
\begin{proof}[Proof of Lemma \ref{Lem:3.2}]
	We only prove the estimate \eqref{eq:3.10} under the assumption on $\ell=2m$, $m \ge 1$.
	The same argument works for \eqref{eq:3.9}, \eqref{eq:3.11} and \eqref{eq:3.12}. 
	
	We decompose $\partial_{t}^{\ell} \mathbb{K}_{1}^{(\beta)}(t) \ast g$ into 4 parts;
	\begin{equation} \label{eq:3.16}
		\begin{split}
			\partial^{\ell}_{t} \mathbb{K}_{1L}^{(\beta)}(t) \ast g
			= 
			A_{\ell, 1} +A_{\ell,2}+A_{\ell, 3}+A_{\ell, 4},
		\end{split}
	\end{equation}
	where 
	\begin{equation*}
		\begin{split}
			A_{\ell, 1} & := \mathcal{F}^{-1} \left[
			e^{-\frac{\nu |\xi|^{2} t}{2}} (\beta |\xi| \phi_{\nu,\beta})^{\ell-1} \sin (\beta |\xi| \phi_{\nu,\beta} t) \frac{\xi_{a}\xi_{b}}{|\xi|^{2}} \chi_{L} \hat{g} 
			\right], \\
			A_{\ell, 2} & := \ell \mathcal{F}^{-1} \left[
			e^{-\frac{\nu |\xi|^{2} t}{2}} \left(
			-\frac{\nu |\xi|^{2}}{2}
			\right)
			(\beta |\xi| \phi_{\nu,\beta})^{\ell-2} (-1)^{\frac{\ell-2}{2}} \cos (\beta |\xi| \phi_{\nu,\beta} t) \frac{\xi_{a}\xi_{b}}{|\xi|^{2}} \chi_{L} \hat{g} 
			\right], 
		\end{split}
	\end{equation*}
	\begin{equation*}
		\begin{split}
			& A_{\ell, 3}:= \\
			& \sum_{ \substack{ 0 \le j \le \ell-2 \\ j:even } } \binom{\ell}{j} \mathcal{F}^{-1} \left[
			e^{-\frac{\nu |\xi|^{2} t}{2}}
			\left(
			-\frac{\nu |\xi|^{2}}{2}
			\right)^{\ell-j}
			(\beta |\xi| \phi_{\nu,\beta})^{j-1}(-1)^{\frac{j}{2}} \sin (\beta |\xi| \phi_{\nu,\beta} t) \frac{\xi_{a}\xi_{b}}{|\xi|^{2}} \chi_{L} \hat{g} 
			\right], \\
			& A_{\ell, 4}:= \\
			& \sum_{ \substack{1 \le j \le \ell-3 \\ j:odd } }\binom{\ell}{j} \mathcal{F}^{-1} \left[
			e^{-\frac{\nu |\xi|^{2} t}{2}}
			\left(
			-\frac{\nu |\xi|^{2}}{2}
			\right)^{\ell-j}
			(\beta |\xi| \phi_{\nu,\beta})^{j-1}(-1)^{\frac{j-1}{2}} \cos (\beta |\xi| \phi_{\nu,\beta} t) \frac{\xi_{a}\xi_{b}}{|\xi|^{2}} \chi_{L} \hat{g} 
			\right],
		\end{split}
	\end{equation*}
	where $\binom{\ell}{j}:=\frac{\ell !}{j!(\ell-j)!}$.
	We firstly deal with $A_{\ell, 1}$. 
	To do this, we observe that 
	\begin{equation} \label{eq:3.17}
		\begin{split}
			A_{\ell,1}-(-1)^{\frac{\ell}{2}} \beta^{\ell} \nabla^{\ell} \mathbb{G}_{1}^{(\beta)}(t) \ast g
			=A_{\ell,1,1}+A_{\ell,1,2}+A_{\ell,1,3},
		\end{split}
	\end{equation}
	where
	\begin{equation*}
		\begin{split}
			& A_{\ell,1,1} := 
			\mathcal{F}^{-1} \left[
			e^{-\frac{\nu |\xi|^{2} t}{2}} 
			(\beta |\xi|)^{\ell-1}
			(\phi_{\nu,\beta}^{\ell-1} -1) \sin (\beta |\xi| \phi_{\nu,\beta} t) \frac{\xi_{a}\xi_{b}}{|\xi|^{2}} \chi_{L} \hat{g} 
			\right], \\
			& A_{\ell,1,2} := \\
			& \mathcal{F}^{-1} \left[
			e^{-\frac{\nu |\xi|^{2} t}{2}} (\beta |\xi|)^{\ell-1} \left( \sin (\beta |\xi| \phi_{\nu,\beta} t)- \sin (\beta |\xi| t)- \beta |\xi| t(\phi_{\nu,\beta}-1)\cos (\beta |\xi| t)\right) \frac{\xi_{a}\xi_{b}}{|\xi|^{2}} \chi_{L} \hat{g} 
			\right], \\
			& A_{\ell,1,3} := 
		    \mathcal{F}^{-1} \left[
			e^{-\frac{\nu |\xi|^{2} t}{2}} (\beta |\xi|)^{\ell-1} \beta |\xi| t(\phi_{\nu,\beta}-1)\cos (\beta |\xi| t) \frac{\xi_{a}\xi_{b}}{|\xi|^{2}} \chi_{L} \hat{g} 
			\right].
		\end{split}
	\end{equation*}
Then noting that $\phi_{\nu,\beta}^{-1} -1=-\phi_{\nu,\beta}^{-1}(\phi_{\nu,\beta} -1)$,
\begin{equation} \label{eq:3.18}
	\phi_{\nu,\beta} -1 =O(|\xi|^{2})
\end{equation}
and
\begin{equation} \label{eq:3.19}
	\phi_{\nu,\beta}^{\ell-1} -1=(\phi_{\nu,\beta} -1)(1+\phi_{\nu,\beta}+ \cdots +\phi_{\nu,\beta}^{\ell-2}) =O(|\xi|^{2})
\end{equation}
for $\ell \ge 2$ as $|\xi| \to 0$, 
we have
\begin{equation} \label{eq:3.20}
	\begin{split}
		\| \nabla^{\alpha}A_{\ell,1,1} \|_{\infty} \le C \left\|
		e^{-c(1+t)|\xi|^{2}} |\xi|^{\ell+\alpha+1} \chi_{L} \hat{g} 
		\right\|_{1} 
		\le C (1+t)^{-2-\frac{\ell-\tilde{\ell}}{2}-\frac{\alpha-\tilde{\alpha}}{2}}
		\| \nabla^{\tilde{\alpha}+\tilde{\ell}} g \|_{1}.
	\end{split}
\end{equation}
For the estimate of $A_{\ell,1,2}$,
we apply the fact that 
\begin{align} \label{eq:3.21}
	& |\sin (\beta |\xi| \phi_{\nu,\beta} t)- \sin (\beta |\xi| t)- \beta |\xi| t(\phi_{\nu,\beta}-1)\cos (\beta |\xi| t)| \le C (t |\xi|^{3})^{2}=Ct^{2}|\xi|^{6}
\end{align}
to see that
\begin{equation} \label{eq:3.22}
	\begin{split}
	 \| \nabla^{\alpha} A_{\ell,1,2} \|_{\infty} 
		%
		%
		\le C t^{2} \left\|
		e^{-c(1+t)|\xi|^{2}} |\xi|^{\alpha+\ell+5} \chi_{L} \hat{g} 
		\right\|_{1} \le C (1+t)^{-2-\frac{\ell-\tilde{\ell}}{2}-\frac{\alpha-\tilde{\alpha}}{2}}
		\| \nabla^{\tilde{\alpha}+\tilde{\ell}} g \|_{1}.
	\end{split}
\end{equation}
We now show the estimate for $A_{\ell,1,3}$.
When $t \ge 1$, we see that
\begin{equation*} 
	\begin{split}
		\| \nabla^{\alpha} A_{\ell,1,3} \|_{\infty} 
		& \le Ct \left\| W_{0}^{(\beta)} (t)
		\mathcal{F}^{-1} \left[
		\xi^{\alpha} e^{-\frac{\nu |\xi|^{2} t}{2}} |\xi|^{\ell} (\phi_{\nu,\beta}-1) \frac{\xi_{a}\xi_{b}}{|\xi|^{2}} \chi_{L} \hat{g} 
		\right] \right\|_{\infty} \\
		& \le C \left\| \mathcal{F}^{-1} \left[\xi^{\alpha} 
		e^{-\frac{\nu |\xi|^{2} t}{2}} |\xi|^{\ell+4} \frac{1}{\phi_{\nu,\beta}+1} \frac{\xi_{a}\xi_{b}}{|\xi|^{2}} \chi_{L} \hat{g} 
		\right] \right\|_{1} \\
		& \le C t^{-2-\frac{\ell-\tilde{\ell}}{2}-\frac{\alpha-\tilde{\alpha}}{2}}
		\| \nabla^{\tilde{\alpha}+\tilde{\ell}} g \|_{1}
	\end{split}
\end{equation*}
by \eqref{eq:3.18}, \eqref{eq:2.14} and \eqref{eq:2.19}.
On the other hand, when $0 \le t \le 1$, we easily have
\begin{equation*}
	\begin{split}
		\| \nabla^{\alpha} A_{\ell,1,3} \|_{\infty} \le C \left\| |\xi|^{\alpha +\ell}
		 \chi_{L} \hat{g} 
		\right\|_{1} \le C 
		\| \nabla^{\tilde{\alpha}+\tilde{\ell}} g \|_{1}.
	\end{split}
\end{equation*}
Thus we obtain 
\begin{equation} \label{eq:3.23}
	\begin{split}
		\| \nabla^{\alpha} A_{\ell,1,3} \|_{\infty} 
		\le C (1+t)^{-2-\frac{\ell-\tilde{\ell}}{2}-\frac{\alpha-\tilde{\alpha}}{2}}
		\| \nabla^{\tilde{\alpha}+\tilde{\ell}} g \|_{1}.
	\end{split}
\end{equation}
Combining \eqref{eq:3.17}, \eqref{eq:3.20}, \eqref{eq:3.22} and \eqref{eq:3.23}, we conclude that
\begin{equation} \label{eq:3.24}
	\begin{split}
		\| \nabla^{\alpha} A_{\ell,1} \|_{\infty} 
		\le C (1+t)^{-2-\frac{\ell-\tilde{\ell}}{2}-\frac{\alpha-\tilde{\alpha}}{2}}
		\| \nabla^{\tilde{\alpha}+\tilde{\ell}} g \|_{1}.
	\end{split}
\end{equation}
	Next, we prove the estimate for $A_{\ell,2}$.
	To this end, we decompose $A_{\ell,2}$ into two parts;
	\begin{equation} \label{eq:3.25}
		\begin{split}
			A_{\ell, 2} = A_{\ell, 2,1}+A_{\ell, 2,2},
		\end{split}
	\end{equation}
where 
\begin{equation*}
	\begin{split}
		& A_{\ell, 2,1} := \\
		& \ell \mathcal{F}^{-1} \left[
		e^{-\frac{\nu |\xi|^{2} t}{2}} \left(
		-\frac{\nu |\xi|^{2}}{2}
		\right)
		(\beta |\xi| \phi_{\nu,\beta})^{\ell-2} (-1)^{\frac{\ell-2}{2}} (\cos (\beta |\xi| \phi_{\nu,\beta} t) -\cos (\beta |\xi| t))\frac{\xi_{a}\xi_{b}}{|\xi|^{2}} \chi_{L} \hat{g} 
		\right]
	\end{split}
\end{equation*}
\begin{equation*}
	\begin{split}
		A_{\ell, 2,2} := \ell \mathcal{F}^{-1} \left[
		e^{-\frac{\nu |\xi|^{2} t}{2}} \left(
		-\frac{\nu |\xi|^{2}}{2}
		\right)
		(\beta |\xi| \phi_{\nu,\beta})^{\ell-2} (-1)^{\frac{\ell-2}{2}}\cos (\beta |\xi| t) \frac{\xi_{a}\xi_{b}}{|\xi|^{2}} \chi_{L} \hat{g} 
		\right].
	\end{split}
\end{equation*}
We show the estimate for $A_{\ell,2,1}$.
Observing that 
\begin{align} \label{eq:3.26}
	|\cos (t \beta |\xi| \phi_{\nu,\beta})-\cos (t \beta |\xi| )| \le C t |\xi|^{3},
\end{align}
we see
\begin{equation} \label{eq:3.27}
	\begin{split}
		\| \nabla^{\alpha} A_{\ell, 2,1} \|_{\infty} \le C t \left\| 
		e^{-C(1+t)|\xi|^{2}} |\xi|^{\ell+\alpha+3}  \chi_{L} \hat{g} 
		\right\|_{1} \le C(1+t)^{-2-\frac{\ell-\tilde{\ell}}{2}-\frac{\alpha-\tilde{\alpha}}{2}} \| \nabla^{\tilde{\alpha}+\tilde{\ell}} g\|_{1}.
	\end{split}
\end{equation}
For $A_{\ell,2,2}$, 
as in the proof of \eqref{eq:3.23}, we apply the estimates \eqref{eq:2.14} and \eqref{eq:2.19} to have 
\begin{equation*}
	\begin{split}
		\| \nabla^{\alpha} A_{\ell, 2,2} \|_{\infty} 
		& \le Ct^{-2-\frac{\ell-\tilde{\ell}}{2}-\frac{\alpha-\tilde{\alpha}}{2}} \| \nabla^{\tilde{\alpha}+\tilde{\ell}} g\|_{1} 
	\end{split}
\end{equation*}
for $t \ge 1$.
When $0 \le t \le 1$, a direct calculation gives
\begin{equation*}
	\begin{split}
		\| \nabla^{\alpha} A_{\ell, 2,2} \|_{\infty} 
		\le  C\left\| 
		|\xi|^{\alpha+\ell} e^{-c(1+t)|\xi|^{2}}  \chi_{L} \hat{g} 
		\right\|_{1} \le C \| \nabla^{\tilde{\alpha}+\tilde{\ell}} g\|_{1}. 
	\end{split}
\end{equation*}
These estimates show 
\begin{equation} \label{eq:3.28}
	\begin{split}
		\| \nabla^{\alpha} A_{\ell, 2,2} \|_{\infty} 
		\le C (1+t)^{-2-\frac{\ell-\tilde{\ell}}{2}-\frac{\alpha-\tilde{\alpha}}{2}} \| \nabla^{\tilde{\alpha}+\tilde{\ell}} g\|_{1}
	\end{split}
\end{equation}
	for $t \ge 0$.
	Therefore we get 
	\begin{equation} \label{eq:3.29}
		\begin{split}
			\| \nabla^{\alpha} A_{\ell, 2} \|_{\infty} 
			\le C (1+t)^{-2-\frac{\ell-\tilde{\ell}}{2}-\frac{\alpha-\tilde{\alpha}}{2}} \| \nabla^{\tilde{\alpha}+\tilde{\ell}} g\|_{1}
		\end{split}
	\end{equation}
by \eqref{eq:3.25}, \eqref{eq:3.27} and \eqref{eq:3.28}.
	The estimates for $A_{\ell, 3}$ and $A_{\ell, 4}$ are easily shown, since they are the remainder factors. 
	More precisely, we have 
	\begin{equation} \label{eq:3.30}
		\begin{split}
			\| \nabla^{\alpha} A_{\ell, 3} \|_{\infty}+\| \nabla^{\alpha} A_{\ell, 4} \|_{\infty} 
			\le C (1+t)^{-2-\frac{\ell-\tilde{\ell}}{2}-\frac{\alpha-\tilde{\alpha}}{2}}
			\| \nabla^{\tilde{\alpha}+\tilde{\ell}} g \|_{1}.
		\end{split}
	\end{equation}
Summing up \eqref{eq:3.16} with the estimates \eqref{eq:3.24}, \eqref{eq:3.29} and \eqref{eq:3.30}, we obtain the desired estimate \eqref{eq:3.10} for $\ell=2m$, $m \ge 1$.
We complete the proof of Lemma \ref{Lem:3.2}.
\end{proof}
%
\begin{proof}[Proof of Lemma \ref{Lem:3.3}]
	We only prove the estimate \eqref{eq:3.14} for $\ell=2m$ with $m \ge 1$.
	The same argument works for \eqref{eq:3.12}, \eqref{eq:3.13} and \eqref{eq:3.15}.  
	
	We begin with the point-wise estimates in the Fourier space.  
	By direct calculations, we have
	\begin{align}
		& |\nabla_{\xi} \cos (t \beta |\xi| \phi_{\nu,\beta})|  +|\nabla_{\xi} \sin (t \beta |\xi| \phi_{\nu,\beta})| \le Ct, \label{eq:3.31} \\
		& |\nabla_{\xi}^{2} \cos (t \beta |\xi| \phi_{\nu,\beta})|+ |\nabla_{\xi}^{2} \sin (t \beta |\xi| \phi_{\nu,\beta})| \le C (t^{2}+t |\xi|^{-1}), \label{eq:3.32}\\
		& |\nabla_{\xi} (\cos (t \beta |\xi| \phi_{\nu,\beta})-\cos (t \beta |\xi| ) )| \le C (t^{2} |\xi|^{3} +t |\xi|^{2} ), \label{eq:3.33}\\
		& |\nabla_{\xi}^{2} (\cos (t \beta |\xi| \phi_{\nu,\beta})-\cos (t \beta |\xi| ) )| \le C (t^{2} |\xi|^{2} +t|\xi| ), \label{eq:3.34}
	\end{align}
	\begin{equation}
		\begin{split}
			& |\nabla_{\xi} \{ \sin (\beta |\xi| \phi_{\nu,\beta} t)- \sin (\beta |\xi| t)- \beta |\xi| t(\phi_{\nu,\beta}-1)\cos (\beta |\xi| t) \}| \\
			&\le C (t^{3}|\xi|^{6} +t^{2} |\xi|^{5}+t|\xi|^{2} ), \label{eq:3.35}
		\end{split}
	\end{equation}
	\begin{equation}
		\begin{split}
			& |\nabla_{\xi}^{2} \{ \sin (\beta |\xi| \phi_{\nu,\beta} t)- \sin (\beta |\xi| t)- \beta |\xi| t(\phi_{\nu,\beta}-1)\cos (\beta |\xi| t) \}| \\
			&\le C (t^{4}|\xi|^{6} +t^{2}|\xi|^{2} +t |\xi|) \label{eq:3.36}
		\end{split}
	\end{equation}
	and
	\begin{align} \label{eq:3.37}
		\left| \nabla_{\xi}^{k} ( 
		e^{-\frac{\nu |\xi|^{2} t}{2}} O(|\xi|^{\gamma}))
		\right| \le C e^{-c(1+t)|\xi|^{2}} |\xi|^{\gamma-k}
	\end{align}
	for $k \in \mathbb{N}$ on $\supp \chi_{L}$. 
	The proof of the estimates \eqref{eq:3.31}-\eqref{eq:3.37} is straightforward. For the completeness, we will show the estimate \eqref{eq:3.36} in the appendix.
We are now in a position to show the estimate \eqref{eq:3.14}. 
Recalling the decomposition \eqref{eq:3.16}, we estimate $A_{\ell,j}$ ($j=1,2,3,4$) separately. 
Here, we use \eqref{eq:3.17} to have 
\begin{equation} \label{eq:3.38}
	\begin{split}
		& \| \nabla^{\alpha} (A_{\ell,1}-(-1)^{\frac{\ell}{2}} \beta^{\ell} \nabla^{\ell} \mathbb{G}_{1}^{(\beta)}(t) \ast g) \|_{p} \\
		& \le \| \nabla^{\alpha}  A_{\ell,1,1} \|_{p} +\| \nabla^{\alpha}  A_{\ell,1,2} \|_{p}+\| \nabla^{\alpha}  A_{\ell,1,3} \|_{p} \\
		& \le ( \| \mathcal{F}^{-1} [ \chi_{L} \mathcal{A}_{\ell,1,1} ]\|_{1} 
		+\| \mathcal{F}^{-1} [\chi_{L} \mathcal{A}_{\ell,1,2} ]\|_{1}) \| \nabla^{\tilde{\alpha}+\tilde{\ell}} g \|_{p} +\| \nabla^{\alpha}  A_{\ell,1,3} \|_{p}, 
	\end{split}
\end{equation}
where 
	\begin{equation*}
	\begin{split}
		& \mathcal{A}_{\ell,1,1} := \xi^{\alpha-\tilde{\alpha}}
		e^{-\frac{\nu |\xi|^{2} t}{2}} 
		(\beta |\xi|)^{\ell-\tilde{\ell}-1}
		(\phi_{\nu,\beta}^{\ell-1} -1) \sin (\beta |\xi| \phi_{\nu,\beta} t) \frac{\xi_{a}\xi_{b}}{|\xi|^{2}},
	\end{split}
\end{equation*}
	\begin{equation*}
	\begin{split}
		& \mathcal{A}_{\ell,1,2} \\
		& := \xi^{\alpha-\tilde{\alpha}}
		e^{-\frac{\nu |\xi|^{2} t}{2}} (\beta |\xi|)^{\ell-\tilde{\ell}-1} \left( \sin (\beta |\xi| \phi_{\nu,\beta} t)- \sin (\beta |\xi| t)- \beta |\xi| t(\phi_{\nu,\beta}-1)\cos (\beta |\xi| t)\right) \frac{\xi_{a}\xi_{b}}{|\xi|^{2}}.
	\end{split}
\end{equation*}
It is easy to see that 
\begin{equation} \label{eq:3.39}
	\begin{split}
		& \| \chi_{L} \mathcal{A}_{\ell,1,1}  \|_{2} \le C \|
		e^{-c(1+t) |\xi|^{2}} 
		\beta |\xi|^{\alpha-\tilde{\alpha}+\ell-\tilde{\ell}+1} \chi_{L}  \|_{2}
		\le C (1+t)^{-\frac{5}{4}- \frac{\ell-\tilde{\ell}+ \alpha-\tilde{\alpha}}{2}}
	\end{split}
\end{equation}
by \eqref{eq:3.19}.
On the other hand, 
a direct calculation shows 
\begin{equation*}
	\begin{split}
		\nabla_{\xi}^{2} \mathcal{A}_{\ell,1,1} & = \nabla_{\xi}^{2} 
		\left( 
		e^{-\frac{\nu |\xi|^{2} t}{2}}
		O(|\xi|^{\alpha-\tilde{\alpha}+\ell-\tilde{\ell}+1}) 
		\right) \sin (\beta |\xi| \phi_{\nu,\beta} t) \\
		& +2\nabla_{\xi}
		\left( 
		e^{-\frac{\nu |\xi|^{2} t}{2}}
		O(|\xi|^{\alpha-\tilde{\alpha}+\ell-\tilde{\ell}+1}) 
		\right) \nabla_{\xi}\sin (\beta |\xi| \phi_{\nu,\beta} t) \\
		&
		+
		\left( 
		e^{-\frac{\nu |\xi|^{2} t}{2}}
		O(|\xi|^{\alpha-\tilde{\alpha}+\ell-\tilde{\ell}+1}) 
		\right) \nabla_{\xi}^{2}\sin (\beta |\xi| \phi_{\nu,\beta} t). 
	\end{split}
\end{equation*}
Thus, 
noting that $\supp \nabla_{\xi} \chi_{L} \cup \supp \nabla_{\xi}^{2} \chi_{L}$ does not include the neighborhood of $\xi=0$,
we have  
\begin{equation} \label{eq:3.40}
	\begin{split}
		\| \nabla_{\xi}^{2} (\chi_{L} \mathcal{A}_{\ell,1,1})  \|_{2} & \le 
		C \| \chi_{L} \nabla_{\xi}^{2} \mathcal{A}_{\ell,1,1} \|_{2} +C e^{-ct} \\
		& \le C (1+t)
		 \|
		e^{-c(1+t) |\xi|^{2}} |\xi|^{\alpha-\tilde{\alpha}+\ell-\tilde{\ell}-1} \chi_{L}  \|_{2}+C e^{-ct} \\
		& \le C (1+t)^{\frac{3}{4}- \frac{\ell-\tilde{\ell}+ \alpha-\tilde{\alpha}}{2}}
	\end{split}
\end{equation}
by \eqref{eq:3.18}, \eqref{eq:3.31}, \eqref{eq:3.32} and \eqref{eq:3.37}.
Using the estimate \eqref{eq:2.20}, \eqref{eq:3.39} and \eqref{eq:3.40}, 
we obtain
\begin{equation} \label{eq:3.41}
	\begin{split}
		\| \mathcal{F}^{-1} [\chi_{L} \mathcal{A}_{\ell,1,1} ]  \|_{1} & \le 
		C\| \chi_{L} \mathcal{A}_{\ell,1,1}  \|_{2}^{\frac{1}{4}} \| \nabla_{\xi}^{2} (\chi_{L} \mathcal{A}_{\ell,1,1})  \|_{2}^{\frac{3}{4}} \\
		& \le 
		C \| \chi_{L} \mathcal{A}_{\ell,1,1}  \|_{2}^{\frac{1}{4}} \| \chi_{L} \nabla_{\xi}^{2} \mathcal{A}_{\ell,1,1} \|_{2}^{\frac{3}{4}} +C e^{-ct} \\
		& \le C (1+t)^{\frac{1}{4}- \frac{\ell-\tilde{\ell}+ \alpha-\tilde{\alpha}}{2}}.
	\end{split}
\end{equation}
%
By a similar argument, we easily have the estimate for  $\| \mathcal{F}^{-1} [\chi_{L} \mathcal{A}_{\ell,1,2} ]  \|_{1}$ as 
\begin{equation} \label{eq:3.42}
	\begin{split}
		\| \mathcal{F}^{-1} [\chi_{L} \mathcal{A}_{\ell,1,2} ]  \|_{1} \le C (1+t)^{\frac{1}{4}- \frac{\ell-\tilde{\ell}+ \alpha-\tilde{\alpha}}{2}}.
	\end{split}
\end{equation}
%
To show the estimate for $\| \nabla^{\alpha}  A_{\ell,1,3} \|_{p}$, we apply the estimates \eqref{eq:3.18}, \eqref{eq:2.16} and \eqref{eq:2.19} to have
\begin{equation} \label{eq:3.43}
	\begin{split}
		 \| \nabla^{\alpha}  A_{\ell,1,3} \|_{p} & \le C t \left\| W^{(\beta)}_{0}(t) \nabla^{\alpha+\ell-(\tilde{\alpha}+\tilde{\ell})+2} e^{\frac{\nu t \Delta}{2}} \mathcal{R}_{a} \mathcal{R}_{b}
		 \mathcal{F}^{-1} \left[
		 \frac{\chi_{L}}{1+\phi_{\nu,\beta}}
		 \right] \ast \nabla^{\tilde{\alpha}+\tilde{\ell}}g \right\|_{p} \\
		 & \le C
		 t\left\| \nabla^{\alpha+\ell-(\tilde{\alpha}+\tilde{\ell})+2} e^{\frac{\nu t \Delta}{2}} \mathcal{R}_{a} \mathcal{R}_{b}
		 \mathcal{F}^{-1} \left[
		 \frac{\chi_{L}}{1+\phi_{\nu,\beta}}
		 \right] \ast \nabla^{\tilde{\alpha}+\tilde{\ell}}g \right\|_{p} \\
		 & +C t^{2} \left\|  \nabla^{\alpha+\ell-(\tilde{\alpha}+\tilde{\ell})+3} e^{\frac{\nu t \Delta}{2}} \mathcal{R}_{a} \mathcal{R}_{b}
		 \mathcal{F}^{-1} \left[
		 \frac{\chi_{L}}{1+\phi_{\nu,\beta}}
		 \right] \ast \nabla^{\tilde{\alpha}+\tilde{\ell}}g \right\|_{p} \\
		 & \le C (1+t)^{\frac{1}{2}- \frac{\ell-\tilde{\ell}+ \alpha-\tilde{\alpha}}{2}}
		 \|\nabla^{\tilde{\alpha}+\tilde{\ell}}g \|_{p}.
	\end{split}
\end{equation}
Then, it follows from the estimates \eqref{eq:3.38}, \eqref{eq:3.41}, \eqref{eq:3.42} and \eqref{eq:3.43} that 
\begin{equation} \label{eq:3.44}
	\begin{split}
		& \| \nabla^{\alpha} (A_{\ell,1}-(-1)^{\frac{\ell}{2}} \beta^{\ell} \nabla^{\ell} \mathbb{G}_{1}^{(\beta)}(t) \ast g) \|_{p} \le C (1+t)^{\frac{1}{2}- \frac{\ell-\tilde{\ell}+ \alpha-\tilde{\alpha}}{2}}.
	\end{split}
\end{equation}
	The second term in \eqref{eq:3.16}, $\| \nabla^{\alpha} A_{\ell,2} \|_{p}$, is estimated as 
	\begin{equation} \label{eq:3.45}
	\begin{split}
		\| \nabla^{\alpha} A_{\ell,2} \|_{p} \le C (1+t)^{\frac{1}{2}- \frac{\ell-\tilde{\ell}+ \alpha-\tilde{\alpha}}{2}} \| \nabla^{\tilde{\alpha}+\tilde{\ell}} g \|_{p}.
	\end{split}
\end{equation}
	Indeed, this case follows by the same method as in the derivation of the estimate \eqref{eq:3.44}.
	Finally, we prove the estimate for $\| \nabla^{\alpha}  A_{\ell,3} \|_{p} +\| \nabla^{\alpha}  A_{\ell,4} \|_{p}$.  
	Noting that 
	\begin{equation} \label{eq:3.46}
		\begin{split}
			\mathcal{A}_{\ell, 3}:=\sum_{ \substack{ 0 \le j \le \ell-2 \\ j:even } } \binom{\ell}{j} \mathcal{A}_{\ell, 3, j}, \quad 
			\mathcal{A}_{\ell, 4}:=
			\sum_{ \substack{1 \le j \le \ell-3 \\ j:odd } }\binom{\ell}{j} 
			\mathcal{A}_{\ell, 4,j},
		\end{split}
	\end{equation}
	where
	\begin{equation*}
		\begin{split}
	 \mathcal{A}_{\ell, 3, j}:= \xi^{\alpha-\tilde{\alpha}}
			e^{-\frac{\nu |\xi|^{2} t}{2}}
			\left(
			-\frac{\nu }{2}
			\right)^{\ell-j} |\xi|^{2 \ell-\tilde{\ell}-j-1}
			(\beta  \phi_{\nu,\beta})^{j-1}(-1)^{\frac{j}{2}} \sin (\beta |\xi| \phi_{\nu,\beta} t) \frac{\xi_{a}\xi_{b}}{|\xi|^{2}}
		\end{split}
	\end{equation*}
	and 
	\begin{equation*}
		\begin{split}
			\mathcal{A}_{\ell, 4, j}:= \xi^{\alpha-\tilde{\alpha}}
			e^{-\frac{\nu |\xi|^{2} t}{2}}
			\left(
			-\frac{\nu }{2}
			\right)^{\ell-j} |\xi|^{2 \ell-\tilde{\ell}-j-1}
			(\beta  \phi_{\nu,\beta})^{j-1}(-1)^{\frac{j}{2}} \cos (\beta |\xi| \phi_{\nu,\beta} t) \frac{\xi_{a}\xi_{b}}{|\xi|^{2}},
		\end{split}
	\end{equation*}%
	we apply same argument again to see that 
\begin{equation*} 
	\begin{split}
		\| \mathcal{F}^{-1} [\chi_{L} \mathcal{A}_{\ell,3} ]  \|_{1} +\| \mathcal{F}^{-1} [\chi_{L} \mathcal{A}_{\ell,4} ]  \|_{1} \le C (1+t)^{\frac{1}{4}- \frac{\ell-\tilde{\ell}+ \alpha-\tilde{\alpha}}{2}},
	\end{split}
\end{equation*}
which implies the estimate 
\begin{equation} \label{eq:3.47}
	\begin{split}
		\| \nabla^{\alpha} A_{\ell,3} \|_{p}+ \| \nabla^{\alpha} A_{\ell,4} \|_{p} \le C (1+t)^{\frac{1}{4}- \frac{\ell-\tilde{\ell}+ \alpha-\tilde{\alpha}}{2}} \| \nabla^{\tilde{\alpha}+\tilde{\ell}} g \|_{p}.
	\end{split}
\end{equation}
%
Therefore we conclude the desired estimate \eqref{eq:3.14} for $\ell=2m$, $m \ge 1$ by the combination of \eqref{eq:3.44}, \eqref{eq:3.45} and \eqref{eq:3.47}.
We complete the proof of Lemma \ref{Lem:3.3}.
\end{proof}

\subsection{Low frequency part}
At first, we discuss the decay properties of the fundamental solutions to \eqref{eq:1.1}, defined by \eqref{eq:3.1} and \eqref{eq:3.2}.
\begin{prop} \label{Prop:3.4}
	Let $\alpha\ge \tilde{\alpha} \ge 0$, $\ell \ge \tilde{\ell} \ge 0$, $m \ge 0$, $1 \le q \le p\le \infty$ and $t \ge 0$.
	Then it holds that 
	\begin{equation} \label{eq:3.48}
		\begin{split}
			\left\| 
			(\partial^{\ell}_{t} \nabla^{\alpha} \mathbb{K}_{00}^{(\beta)})(t) \ast g
			\right\|_{p} 
			 \le C(1+t)^{-\frac{3}{2}(\frac{1}{q}-\frac{1}{p})-(\frac{1}{q}-\frac{1}{p})+\frac{1}{2}- \frac{\ell-\tilde{\ell}+ \alpha-\tilde{\alpha}}{2}} \| \nabla^{\tilde{\alpha}+\tilde{\ell}} g \|_{q}
		\end{split}
	\end{equation}
	and
	\begin{equation} \label{eq:3.49}
		\begin{split}
			 \left\| 
			(\partial^{\ell}_{t} \nabla^{\alpha} \mathbb{K}_{1}^{(\beta)})(t) \ast g 
			\right\|_{p} 
			 \le C(1+t)^{-\frac{3}{2}(\frac{1}{q}-\frac{1}{p})-(\frac{1}{q}-\frac{1}{p})+1- \frac{\ell-\tilde{\ell}+ \alpha-\tilde{\alpha}}{2}} \| \nabla^{\tilde{\alpha}+\tilde{\ell}} g \|_{q},
		\end{split}
	\end{equation}
	where $(p,q) \neq (1,1), (\infty,\infty)$ if $\ell+\alpha=0$.
\end{prop}
For the proof of Proposition \ref{Prop:3.4}, the following lemma is useful.
\begin{lem} \label{Lem:3.5}
	Let $\alpha \ge 0$, $\ell \ge 0$ and $t>0$.
	Then it holds that 
	\begin{equation}  \label{eq:3.50}
		\begin{split}
			\left\| \partial^{\ell}_{t} \nabla^{\alpha} 
			\mathbb{G}_{0}^{(\beta)}(t) 
			\right\|_{p} 
			\le
		C (1+t)^{-\frac{3}{2}(1-\frac{1}{p})-(1-\frac{1}{p})+\frac{1}{2}-\frac{\ell+\alpha}{2}},
		\end{split}
	\end{equation}
	\begin{equation}  \label{eq:3.51}
		\begin{split}
			\left\| \partial^{\ell}_{t} \nabla^{\alpha} 
			\mathbb{G}_{1}^{(\beta)}(t) 
			\right\|_{p} 
			\le C (1+t)^{-\frac{3}{2}(1-\frac{1}{p})-(1-\frac{1}{p})+1-\frac{\ell+\alpha}{2} 
			},
		\end{split}
	\end{equation}
	where $1 < p \le \infty$ for $\ell+\alpha=0$ and $1 \le p \le \infty$ for $\ell+\alpha \ge 1$.
\end{lem}
The proof is straightforward. We omit the detail.
\begin{proof}[Proof of Proposition \ref{Prop:3.4}]
At first, we show the estimate \eqref{eq:3.49} with $\ell=2m$.
Noting the estimates \eqref{eq:3.6} and \eqref{eq:3.51}, we see that
\begin{equation*} 
	\begin{split}
		& \left\| \partial^{\ell}_{t} \nabla^{\alpha} 
		 \mathbb{K}_{1}^{(\beta)}(t) \ast g 
		\right\|_{p}  \\
		& \left\| \nabla^{\alpha} \left(
		\partial^{\ell}_{t} \mathbb{K}_{1}^{(\beta)}(t) \ast g - (-1)^{\frac{\ell}{2}} \beta^{\ell} \nabla^{\ell} \mathbb{G}_{1}^{(\beta)}(t) \ast g
		\right)
		\right\|_{p} +\left\| \nabla^{\alpha} (-1)^{\frac{\ell}{2}} \beta^{\ell} \nabla^{\ell} \mathbb{G}_{1}^{(\beta)}(t) \ast g
		\right\|_{p}\\
		&\le C(1+t)^{-\frac{3}{2}(\frac{1}{q}-\frac{1}{p})-(\frac{1}{q}-\frac{1}{p})+1- \frac{\ell-\tilde{\ell}+ \alpha-\tilde{\alpha}}{2}} \| \nabla^{\tilde{\alpha}+\tilde{\ell}} g \|_{q},
	\end{split}
\end{equation*}
		which is the desired estimate \eqref{eq:3.49} with $\ell=2m$.
The other cases for the estimates \eqref{eq:3.48} and \eqref{eq:3.49} are shown in a similar way. 
We complete the proof of proposition.
\end{proof}
Next we state the expansion formulas of $\partial^{\ell}_{t} \mathbb{G}_{0}^{(\beta)}(t) \ast g$ and $\partial^{\ell}_{t} \mathbb{G}_{1}^{(\beta)}(t) \ast g$ as $t \to \infty$.
\begin{prop} \label{Prop:3.6}
	Let $\alpha, \ell, m \ge 0$ and $g \in L^{1}$.
	Then it holds that 
	\begin{equation} \label{eq:3.52}
		\begin{split}
			\left\| \nabla^{\alpha} \left(
			\partial^{\ell}_{t} \mathbb{G}_{0}^{(\beta)}(t) \ast g - m_{g}(-1)^{\frac{\ell}{2}} \beta^{\ell} \nabla^{\ell} \mathbb{G}_{0}^{(\beta)}(t) 
			\right)
			\right\|_{p} 
			= o(t^{-\frac{3}{2}(1-\frac{1}{p})-(1-\frac{1}{p})+\frac{1}{2}-\frac{\ell+\alpha}{2}})
		\end{split}
	\end{equation}
	for $\ell=2m$,
	\begin{equation} \label{eq:3.53}
		\begin{split}
			\left\| \nabla^{\alpha} \left(
			\partial^{\ell}_{t} \mathbb{G}_{0}^{(\beta)}(t) \ast g - m_{g} (-1)^{\frac{\ell+1}{2}} \beta^{\ell+1} \nabla^{\ell+1} \mathbb{G}_{1}^{(\beta)}(t) 
			\right)
			\right\|_{p} 
			= o(t^{-\frac{3}{2}(1-\frac{1}{p})-(1-\frac{1}{p})+\frac{1}{2}-\frac{\ell+\alpha}{2}})
		\end{split}
	\end{equation}
	for $\ell=2m+1$,
	\begin{equation} \label{eq:3.54}
		\begin{split}
			\left\| \nabla^{\alpha} \left(
			\partial^{\ell}_{t} \mathbb{G}_{1}^{(\beta)}(t) \ast g -  m_{g} (-1)^{\frac{\ell}{2}} \beta^{\ell} \nabla^{\ell} \mathbb{G}_{1}^{(\beta)}(t) 
			\right)
			\right\|_{p} 
			= o(t^{-\frac{3}{2}(1-\frac{1}{p})-(1-\frac{1}{p})+1-\frac{\ell+\alpha}{2} })
		\end{split}
	\end{equation}
	for $\ell=2m$ and 
	\begin{equation} \label{eq:3.55}
		\begin{split}
			\left\| \nabla^{\alpha} \left(
			\partial^{\ell}_{t} \mathbb{G}_{1}^{(\beta)}(t) \ast g - m_{g} (-1)^{\frac{\ell-1}{2}} \beta^{\ell-1} \nabla^{\ell-1} \mathbb{G}_{0}^{(\beta)}(t) 
			\right)
			\right\|_{p} 
			= o(t^{-\frac{3}{2}(1-\frac{1}{p})-(1-\frac{1}{p})+1-\frac{\ell+\alpha}{2} })
		\end{split}
	\end{equation}
	for $\ell=2m+1$, 
	as $t \to \infty$, where $1<p \le \infty$ for $\ell+\alpha=0$ and $1 \le p \le \infty$ for $\ell+\alpha \ge 1$.
	Here $m_{g}$ is defined by 
	\begin{equation} \label{eq:3.56}
		\begin{split}
	m_{g}:=\displaystyle\int_{\R^{3}} g(x) dx.
\end{split}
\end{equation}
\end{prop}
\begin{proof}
	We only prove the estimate \eqref{eq:3.52} for the case $j+\ell=0$ and $1 <p\le \infty$.
	The proofs of the other cases are easier.
	
	We write
	\begin{equation*}
		\begin{split}
			& \mathbb{G}_{0}^{(\beta)}(t)\ast g - m_{g} \mathbb{G}_{0}^{(\beta)}(t) \\ 
			& = \int_{|y|\le t^{\frac{1}{4}}} (\mathbb{G}_{0}^{(\beta)}(t,x-y)-\mathbb{G}_{0}^{(\beta)}(t,x) ) g(y) dy \\
			& +\int_{|y|\ge t^{\frac{1}{4}}} \mathbb{G}_{0}^{(\beta)}(t,x-y) g(y) dy -\int_{|y|\ge t^{\frac{1}{4}}} \mathbb{G}_{0}^{(\beta)}(t,x) g(y) dy.
		\end{split}
	\end{equation*}
	On the other hand, applying the mean value theorem with some $\theta \in [0,1]$, 
	we have 
		\begin{equation*}
		\begin{split}
			& \left\| \int_{|y|\le t^{\frac{1}{4}}} (\mathbb{G}_{0}^{(\beta)}(t,x-y)-\mathbb{G}_{0}^{(\beta)}(t,x) ) g(y) dy \right\|_{p} \\
			&  \le C \int_{|y|\le t^{\frac{1}{4}}} |y| \left\| \nabla \mathbb{G}_{0}^{(\beta)}(t,x-\theta y) \right\|_{L^{p}_{x}} |g(y)| dy \\
			& \le C t^{-\frac{3}{2}(1-\frac{1}{p})-(1-\frac{1}{p})+\frac{1}{4}}  \| g \|_{1}.
		\end{split}
	\end{equation*}
	Here we used the estimate \eqref{eq:3.50}.
	Therefore we arrive at the following estimate: 
	\begin{equation*}
		\begin{split}
			& \| \mathbb{G}_{0}^{(\beta)}(t)\ast g - m_{g} \mathbb{G}_{0}^{(\beta)}(t) \|_{p} \\ 
			& \le C t^{-\frac{3}{2}(1-\frac{1}{p})-(1-\frac{1}{p})+\frac{1}{4}}  \| g \|_{1}
			+C \| \mathbb{G}_{0}^{(\beta)}(t) \|_{p} \int_{|y|\ge t^{\frac{1}{4}}} |g(y)| dy  \\
			& \le C t^{-\frac{3}{2}(1-\frac{1}{p})-(1-\frac{1}{p})+\frac{1}{2}} 
			\left( t^{-\frac{1}{4}} \| g \|_{1} + C t  \int_{|y|\ge t^{\frac{1}{4}}}  |g(y)| dy \right)\\
			& =o(t^{-\frac{3}{2}(1-\frac{1}{p})-(1-\frac{1}{p})+\frac{1}{2}} )
		\end{split}
	\end{equation*}
	as $t \to \infty$, since $g \in L^{1}$.
	We complete the proof of Proposition \ref{Prop:3.6}.
\end{proof}

Finally we can also obtain the approximation formulas of $K_{00L}^{(\beta)}(t) g:= \mathcal{F}^{-1}[\mathcal{K}_{00}^{(\beta)}(t, \xi) \chi_{L}]\ast g$ and $K_{1L}^{(\beta)}(t) \ast g$,
applying the same argument.
\begin{cor} \label{cor:3.7}
	Under the assumption on Proposition \ref{Prop:3.1}, 
	it holds that 
	\begin{equation} \label{eq:3.57}
		\begin{split}
			& \left\| \nabla^{\alpha} \left(
			\partial^{\ell}_{t} K_{00L}^{(\beta)}(t) g -(-1)^{\frac{\ell}{2}} \beta^{\ell} \nabla^{\ell} G_{0L}^{(\beta)}(t) \ast g  
			\right)
			\right\|_{p} \\
			& \le C(1+t)^{-\frac{3}{2}(\frac{1}{q}-\frac{1}{p})-(\frac{1}{q}-\frac{1}{p})- \frac{\ell-\tilde{\ell}+ \alpha-\tilde{\alpha}}{2}} \| \nabla^{\tilde{\alpha}+\tilde{\ell}} g \|_{q}
		\end{split}
	\end{equation}
	for $\ell=2m$,
	\begin{equation} \label{eq:3.58}
		\begin{split}
			& \left\| \nabla^{\alpha} \left(
			\partial^{\ell}_{t} K_{00L}^{(\beta)}(t) g -(-1)^{\frac{\ell+1}{2}} \beta^{\ell+1} \nabla^{\ell+1} G_{1L}^{(\beta)}(t)\ast g 
			\right)
			\right\|_{p} \\
			& \le C(1+t)^{-\frac{3}{2}(\frac{1}{q}-\frac{1}{p})-(\frac{1}{q}-\frac{1}{p})- \frac{\ell-\tilde{\ell}+ \alpha-\tilde{\alpha}}{2}} \| \nabla^{\tilde{\alpha}+\tilde{\ell}} g \|_{q}
		\end{split}
	\end{equation}
	for $\ell=2m+1$,
	\begin{equation} \label{eq:3.59}
		\begin{split}
			& \left\| \nabla^{\alpha} \left(
			\partial^{\ell}_{t} K_{1L}^{(\beta)}(t) g - (-1)^{\frac{\ell}{2}} \beta^{\ell} \nabla^{\ell} G_{1L}^{(\beta)}(t) \ast g
			\right)
			\right\|_{p} \\ 
			& \le C(1+t)^{-\frac{3}{2}(\frac{1}{q}-\frac{1}{p})-(\frac{1}{q}-\frac{1}{p})+\frac{1}{2}- \frac{\ell-\tilde{\ell}+ \alpha-\tilde{\alpha}}{2}} \| \nabla^{\tilde{\alpha}+\tilde{\ell}} g \|_{q}
		\end{split}
	\end{equation}
	for $\ell=2m$ and 
	\begin{equation} \label{eq:3.60}
		\begin{split}
			& \left\| \nabla^{\alpha} \left(
			\partial^{\ell}_{t} K_{1L}^{(\beta)}(t) g - m_{g} (-1)^{\frac{\ell-1}{2}} \beta^{\ell-1} \nabla^{\ell-1} G_{0L}^{(\beta)}(t) \ast g
			\right)
			\right\|_{p} \\
			& \le C(1+t)^{-\frac{3}{2}(\frac{1}{q}-\frac{1}{p})-(\frac{1}{q}-\frac{1}{p})+\frac{1}{2}-\frac{\ell-\tilde{\ell}+ \alpha-\tilde{\alpha}}{2}} \| \nabla^{\tilde{\alpha}+\tilde{\ell}} g \|_{q}
		\end{split}
	\end{equation}
	for $\ell=2m+1$.
\end{cor}
%
\section{Linear estimates for middle and high frequency parts}
%
In this section, we shall show the estimates for the middle and high frequency parts of the fundamental solutions to \eqref{eq:1.1}. 
At first, we summarize the estimates in the $L^{p}$ for $1 \le p \le \infty$, 
which are proved in the forthcoming paper \cite{Kagei-T2}.
\begin{prop} \label{prop:4.1}
	Let $\alpha \ge \tilde{\alpha} \ge 0$, $\ell \ge 2 \tilde{\ell} \ge 0$ and $t>0$. 
	Then, the following estimates are hold:
	\begin{equation}
		\begin{split}
			& \| \partial_{t}^{\ell} \nabla^{\alpha} K_{0H}^{(\beta)}(t)  g  \|_{p} 
			+\| \partial_{t}^{\ell} \nabla^{\alpha} K_{0H}^{(\beta)}(t) \mathcal{R}_{a}  \mathcal{R}_{b} g  \|_{p}  \\
			& \le C e^{-ct} (\| \nabla^{\alpha_{1}} g \|_{p} +t^{-\frac{3}{2}(\frac{1}{q}-\frac{1}{p})-\frac{\alpha-\tilde{\alpha}}{2}-(\ell-\frac{\tilde{\ell}}{2})+1}
			\| \nabla^{\tilde{\alpha}+\tilde{\ell}} g \|_{q}), \quad \alpha_{1} \ge \alpha, \label{eq:4.1} 
		\end{split}
	\end{equation}
	\begin{equation}
		\begin{split}
			& \| \partial_{t}^{\ell} \nabla^{\alpha} K_{1H}^{(\beta)}(t) g \|_{p}
			+\| \partial_{t}^{\ell} \nabla^{\alpha} K_{1H}^{(\beta)}(t) \mathcal{R}_{a}  \mathcal{R}_{b}g \|_{p}  \\
			& \le C e^{-ct} (\| \nabla^{\alpha_{1}} g \|_{p} 
			+t^{-\frac{3}{2}(\frac{1}{q}-\frac{1}{p})-\frac{\alpha-\tilde{\alpha}}{2}-(\ell-\frac{\tilde{\ell}}{2})+1}\| \nabla^{\tilde{\alpha}+\tilde{\ell}} g \|_{q}),
			\quad \alpha_{1} \ge \max\{0, \alpha-2 \}
			\label{eq:4.2}
		\end{split}
	\end{equation}
	for $1 < p< \infty$ and $1 \le q \le p$ and 
	\begin{equation}
		\begin{split}
			& \| \partial_{t}^{\ell} \nabla^{\alpha} K_{0H}^{(\beta)}(t)  g  \|_{p} 
			+\| \partial_{t}^{\ell} \nabla^{\alpha} K_{0H}^{(\beta)}(t) \mathcal{R}_{a}  \mathcal{R}_{b} g  \|_{p}  \\
			& \le C e^{-ct} (\| \nabla^{\alpha+2} g \|_{p} +t^{-\frac{3}{2}(\frac{1}{q}-\frac{1}{p})-\frac{\alpha-\tilde{\alpha}}{2}-(\ell-\frac{\tilde{\ell}}{2})+1}
			\| \nabla^{\tilde{\alpha}+\tilde{\ell}} g \|_{q}), \label{eq:4.3} 
		\end{split}
	\end{equation}
	\begin{equation}
		\begin{split}
			& \| \partial_{t}^{\ell} \nabla^{\alpha} K_{1H}^{(\beta)}(t) g \|_{p}
			+\| \partial_{t}^{\ell} \nabla^{\alpha} K_{1H}^{(\beta)}(t) \mathcal{R}_{a}  \mathcal{R}_{b}g \|_{p}  \\
			& \le C e^{-ct} (\| \nabla^{\alpha} g \|_{p} 
			+t^{-\frac{3}{2}(\frac{1}{q}-\frac{1}{p})-\frac{\alpha-\tilde{\alpha}}{2}-(\ell-\frac{\tilde{\ell}}{2})+1}\| \nabla^{\tilde{\alpha}+\tilde{\ell}} g \|_{q})
			\label{eq:4.4}
		\end{split}
	\end{equation}
	and
	\begin{equation} \label{eq:4.5}
		\begin{split}
			& \| \partial_{t}^{\ell} \nabla^{\alpha} K_{0M}^{(\beta)}(t) \mathcal{R}_{a}  \mathcal{R}_{b}g \|_{p}
			+\| \partial_{t}^{\ell} \nabla^{\alpha} K_{1M}^{(\beta)}(t) \mathcal{R}_{a}  \mathcal{R}_{b}g \|_{p} \\
			& + \| \partial_{t}^{\ell} \nabla^{\alpha} K_{0M}^{(\beta)}(t) g \|_{p}
			+\| \partial_{t}^{\ell} \nabla^{\alpha} K_{1M}^{(\beta)}(t) g \|_{p}
			\le C e^{-ct} \| \nabla^{\tilde{\alpha}} g \|_{q} 
		\end{split}
	\end{equation}
	for $1 \le q \le p \le \infty$. 
\end{prop}
%
The next lemma implies smoothing effect of the fundamental solutions to \eqref{eq:1.1} in $L^{2}$ and $L^{\infty}$, 
which is useful to obtain the estimates in the main results. 
\begin{lem} \label{Lem:4.2}
	{\rm (i)} Let $t \ge 0$. 
	Then there exists a constant $C>0$ such that 
	\begin{align}
		& \| \nabla^{\alpha} \partial_{t}^{\ell} K_{0H}^{(\beta)}(t) g \|_{2} +
		\| \nabla^{\alpha} \partial_{t}^{\ell}  K_{0H}^{(\beta)}(t) \mathcal{R}_{a}  \mathcal{R}_{b} g \|_{2} \le C e^{-ct} \| \nabla^{3} g \|_{2}, \label{eq:4.6} \\
		& \| \nabla^{\alpha} K_{1H}^{(\beta)}(t) g \|_{2} + \| \nabla^{\alpha} K_{1H}^{(\beta)}(t) \mathcal{R}_{a}  \mathcal{R}_{b} g \|_{2} \le C e^{-ct} \| \nabla g \|_{2} \label{eq:4.7} 
	\end{align}
	for $0 \le \alpha \le 3$ and $\ell=0,1$,
	\begin{align}
		\| \nabla^{\alpha} \partial_{t} K_{1H}^{(\beta)}(t) g \|_{2} +\| \nabla^{\alpha} \partial_{t} K_{1H}^{(\beta)}(t) \mathcal{R}_{a}  \mathcal{R}_{b}g \|_{2} \le C e^{-ct} \| \nabla g \|_{2} 
		\label{eq:4.8}
	\end{align}
	for $0 \le \alpha \le 1$ and 
	\begin{align}
		& \| \nabla^{\alpha} \partial_{t}^{\ell} K_{0H}^{(\beta)}(t) g \|_{\infty} +
		\| \nabla^{\alpha} \partial_{t}^{\ell}  K_{0H}^{(\beta)}(t) \mathcal{R}_{a}  \mathcal{R}_{b} g \|_{\infty} \le C e^{-ct} \| \nabla^{3} g \|_{2}, \label{eq:4.9} \\
		& \| \nabla^{\alpha} \partial_{t}^{2} K_{0H}^{(\beta)}(t) g \|_{2} +
		\| \nabla^{\alpha} \partial_{t}^{2}  K_{0H}^{(\beta)}(t) \mathcal{R}_{a}  \mathcal{R}_{b} g \|_{2} \le C e^{-ct} \| \nabla^{3} g \|_{2}, \label{eq:4.10}  \\
& \| \nabla^{\alpha} K_{1H}^{(\beta)}(t) g \|_{\infty} +
\| \nabla^{\alpha} K_{1H}^{(\beta)}(t) \mathcal{R}_{a}  \mathcal{R}_{b} g \|_{\infty} \le C e^{-ct} \| \nabla g \|_{2} \label{eq:4.11}
\end{align}
	for $0 \le \alpha \le 1$ and $\ell=0,1$.
	 \\
	{\rm (ii)} Let $t >0$. Then there exists a constant $C>0$ such that 
	\begin{align}
		& \| \nabla^{2} \partial_{t} K_{1H}^{(\beta)}(t) g \|_{2} +\| \nabla^{2} \partial_{t} K_{1H}^{(\beta)}(t) \mathcal{R}_{a}  \mathcal{R}_{b}g \|_{2} \le C e^{-ct}(1+t^{-\frac{1}{2}}) \| \nabla g \|_{2}, 
		\label{eq:4.12} \\
		& \| \nabla^{\alpha} \partial_{t} K_{1H}^{(\beta)}(t) g \|_{\infty}+\| \nabla^{\alpha} \partial_{t} K_{1H}^{(\beta)}(t) \mathcal{R}_{a}  \mathcal{R}_{b} g \|_{\infty} 
		\le C e^{-ct}(1+t^{-\frac{1}{4}-\frac{\alpha}{2}}) \| \nabla g \|_{2} 
		\label{eq:4.13}
	\end{align}
	for $0 \le \alpha \le  1$ and
	\begin{align}
		\|\partial_{t}^{2} K_{1H}^{(\beta)}(t) g \|_{2} +\|\partial_{t}^{2} K_{1H}^{(\beta)}(t) \mathcal{R}_{a}  \mathcal{R}_{b}g \|_{2} \le C e^{-ct}
		(1+t^{-\frac{1}{2}})
		\| \nabla g \|_{2}. 
		\label{eq:4.14}
	\end{align}
	{\rm (iii)} Let $t \ge 0$. Then there exists a constant $C>0$ such that 
	\begin{align}
		& \| \nabla^{2} \partial_{t} K_{1H}^{(\beta)}(t) g \|_{2} +\| \nabla^{2} \partial_{t} K_{1H}^{(\beta)}(t) \mathcal{R}_{a}  \mathcal{R}_{b}g \|_{2} \le C e^{-ct}\| \nabla^{2} g \|_{2}.
		\label{eq:4.15} 
	\end{align}
\end{lem}
\begin{proof}
	The proof of the estimates \eqref{eq:4.6}-\eqref{eq:4.15} is straightforward. 
	Indeed, we easily have the point-wise estimates of the high frequency parts of the fundamental solutions defined by \eqref{eq:2.2}: 
	\begin{equation*} 
		\begin{split}
			|\partial_{t}^{\ell} \mathcal{K}_{0}^{(\beta)}(t,\xi)| \le Ce^{-ct} (1+|\xi|^{2(\ell-1)} e^{-ct |\xi|^{2}}) \chi_{H},
		\end{split}
	\end{equation*}
	\begin{equation*} 
		\begin{split}
			|\partial_{t}^{\ell} \mathcal{K}_{1}^{(\beta)}(t,\xi)| \le Ce^{-ct}|\xi|^{-2} (1+|\xi|^{2\ell} e^{-ct |\xi|^{2}}) \chi_{H},
		\end{split}
	\end{equation*}
	where $\ell \ge 0$. Therefore we simply apply the H\"older inequality to have the estimates \eqref{eq:4.6}-\eqref{eq:4.15}.
	We complete the proof of Lemma \ref{Lem:4.2}.  
\end{proof}

\section{Global existence of solutions}
In this section, we prove Theorem \ref{thm:1.1}, which claims the existence of the global solution to \eqref{eq:1.1} with decay properties.  
We begin with the representation formula of the solution for problem \eqref{eq:1.1}.
\subsection{Solution formula}
We formulate the Cauchy problem \eqref{eq:1.1} into the integral equation as follows:
\begin{prop} \label{prop:5.1}
	Let $u$ be a solution of \eqref{eq:1.1}. Then it holds that
	\begin{equation} \label{eq:5.1}
		\begin{split}
			\hat{u}(t,\xi) & =\mathcal{K}_{0}^{(\sqrt{\lambda+2 \mu})}(t,\xi) \mathcal{P} \hat{f}_{0}(\xi) 
			+ \mathcal{K}_{0}^{(\sqrt{\mu})}(t,\xi) (\mathcal{I}_{3}-\mathcal{P}) \hat{f}_{0}(\xi) \\
			& +\mathcal{K}_{1}^{(\sqrt{\lambda+2 \mu})}(t,\xi) \mathcal{P} \hat{f}_{1}(\xi) 
			+ \mathcal{K}_{1}^{(\sqrt{\mu})}(t,\xi) (\mathcal{I}_{3}-\mathcal{P}) \hat{f}_{1}(\xi) \\
			& + \int_{0}^{t} \left\{ 
			\mathcal{K}_{1}^{(\sqrt{\lambda+2 \mu})}(t-\tau,\xi) \mathcal{P} 
			+ \mathcal{K}_{1}^{(\sqrt{\mu})}(t-\tau,\xi) (\mathcal{I}_{3}-\mathcal{P}) 
			\right\} \hat{F}(u)(\tau, \xi) d \tau.
		\end{split}
	\end{equation}
\end{prop}
For the proof of Proposition \ref{prop:5.1}, we recall the elementary facts on the matrices.
\begin{lem}\label{lem:5.2}
	There exists
	$\mathcal{Q} \in O(3)$ such that the first column of $\mathcal{Q}$ is given by $\dfrac{\xi}{|\xi|}$ and 
	\begin{equation} \label{eq:5.2}
		\mathcal{Q}^{-1} \mathcal{A} \mathcal{Q} = { \rm diag} ((\lambda+2 \mu) |\xi|^{2}, \mu |\xi|^{2} \mu |\xi|^{2})=:\Lambda,
	\end{equation}
	where 
	\begin{equation*}
		 \mathcal{A}:= -\mu |\xi|^{2} \mathcal{I}_{3} -(\lambda+\mu)|\xi|^{2} \mathcal{P}.
	\end{equation*}
\end{lem}
\begin{lem} \label{lem:5.3}
	Let $\mathcal{B}:= {\rm diag}(\beta, \alpha, \alpha) \in M(3)$, $\mathcal{Q} \in O(3)$ and 
	${\bf q}_{1}$ be the first column of $\mathcal{Q}$. 
	Then it holds that 
	\begin{equation*}
		\mathcal{Q} \mathcal{B} \mathcal{Q}^{-1} = \beta {\bf q}_{1} \otimes {\bf q}_{1} + \alpha (\mathcal{I}_{n} -{\bf q}_{1} \otimes {\bf q}_{1} ).
	\end{equation*}
\end{lem}
The proof of Lemmas \ref{lem:5.2}-\ref{lem:5.3} is elementary. See e.g. \cite{Takeda}. 
\begin{proof}[Proof of Proposition \ref{prop:5.1}]
	We take the Fourier transform and multiply $\mathcal{Q}^{-1}$ from left to \eqref{eq:1.1}, 
	to see that 
	\begin{equation*} 
		\left\{
		\begin{split}
			& \hat{v}'' +\Lambda \hat{v}+\nu |\xi|^{2} \hat{v}'  =\mathcal{Q}^{-1} \hat{F}(u), \quad t>0, \\
			& \hat{v}(0)=\mathcal{Q}^{-1}\hat{f}_{0}, \quad  \hat{v}'(0)=\mathcal{Q}^{-1} \hat{f}_{1}
		\end{split}
		\right.
	\end{equation*}
	by \eqref{eq:5.2},
where $\hat{v}:=\mathcal{Q}^{-1} \hat{u}$.
Then, denoting
\begin{equation*} 
	\mathcal{L}_{j}(t,\xi):= {\rm diag}(\mathcal{K}_{j}^{(\sqrt{\lambda+2 \mu})}(t,\xi),\mathcal{K}_{j}^{(\sqrt{\mu})}(t,\xi),\mathcal{K}_{j}^{(\sqrt{\mu})}(t,\xi))
\end{equation*}
for $j=0,1$,
we apply \eqref{eq:2.3} to have
\begin{equation} \label{eq:5.3}
	\begin{split}
		\hat{v}(t) = \mathcal{L}_{0}(t,\xi) \mathcal{Q}^{-1} \hat{f}_{0} + \mathcal{L}_{1}(t,\xi) \mathcal{Q}^{-1} \hat{f}_{1}
		+ \int_{0}^{t}  \mathcal{L}_{1}(t-\tau,\xi) \mathcal{Q}^{-1} \hat{F}(u)(\tau) d \tau.
	\end{split}
\end{equation}
Therefore, multiplying $\mathcal{Q}$ from left to \eqref{eq:5.3} and using Lemma \ref{lem:5.3} yield \eqref{eq:5.1},
which is the desired result. 
\end{proof}
\subsection{Proof of Theorems \ref{thm:1.1}}
Let $F(u)=\nabla u \nabla^{2} u$. 
We introduce a mapping $\Phi$ on $\mathbb{B}_{R}:=\{ u \in X_{1} ;\| u \|_{X_{1}} \le R \}$,
a closed ball with the radius $R>0$ in the complete metric space $X_{1}$ defined by
\begin{equation*}
	\begin{split}
		X_{1}:=\{ u\in \{C([0,\infty); \dot{H}^{3} \cap \dot{H}^{1}) \cap C^{1}([0,\infty); H^{1}) \}^{3};  \| u \|_{X_{1}}  < \infty \},
	\end{split}
\end{equation*}
equipped with the norm
\begin{equation*}
	\begin{split}
		\| u \|_{X_{1}} & :=\sup_{t \ge 0}
		\{ (1+t)^{\frac{7}{4}} \| \nabla^{3} u(t) \|_{2}  + (1+t)^{\frac{3}{4}} (\| \nabla u(t) \|_{2}+ \|\partial_{t} u(t) \|_{2}) \\
		& \quad + (1+t)^{\frac{5}{4}} \| \nabla \partial_{t} u(t) \|_{2} \}.
	\end{split}
\end{equation*}
More precisely, we denote 
\begin{equation*} 
	\begin{split}
		\Phi[u](t) & =u_{lin}(t)+\Phi_{N}[u](t),
	\end{split}
\end{equation*}
where
\begin{equation*} 
	\begin{split}
		u_{lin}(t) & :=(K_{0}^{(\sqrt{\lambda+2 \mu})}(t)-  K_{0}^{(\sqrt{\mu})}(t)) \mathcal{F}^{-1} [\mathcal{P} \hat{f}_{0}] 
		+ K_{0}^{(\sqrt{\mu})}(t)f_{0} \\
		& +(K_{1}^{(\sqrt{\lambda+2 \mu})}(t)-  K_{1}^{(\sqrt{\mu})}(t)) \mathcal{F}^{-1} [\mathcal{P} \hat{f}_{1}] 
		+ K_{1}^{(\sqrt{\mu})}(t)f_{1},
	\end{split}
\end{equation*}
and 
\begin{equation*} 
	\begin{split}
		\Phi_{N}[u](t) & := \int_{0}^{t}
		(K_{1}^{(\sqrt{\lambda+2 \mu})}(t-\tau)-K_{1}^{(\sqrt{\mu})}(t-\tau)) \mathcal{F}^{-1} [\mathcal{P} \hat{F}(u)(\tau) ] d \tau \\
		& + \int_{0}^{t} K_{1}^{(\sqrt{\mu})}(t-\tau) F(u)(\tau, \xi) d \tau.
	\end{split}
\end{equation*}
In this framework, 
noting that $\| \nabla u(t) \|_{\infty} \le C \| \nabla u(t) \|_{2}^{\frac{1}{4}} \| \nabla^{2} u(t) \|_{2}^{\frac{3}{4}} \le C(1+t)^{-\frac{3}{2}}$,
we have the estimates for the nonlinear term as follows: 
\begin{equation} \label{eq:5.4}
	\| F(u) \|_{1} \le C \| \nabla u \|_{2} \| \nabla^{2} u \|_{2} \le C(1+t)^{-2} \| u \|_{X_{1}}^{2},
\end{equation}
\begin{equation} \label{eq:5.5}
	\| F(u) \|_{2} \le C \| \nabla u \|_{\infty} \| \nabla^{2} u \|_{2} \le C(1+t)^{-\frac{11}{4}}\| u \|_{X_{1}}^{2} 
\end{equation}
by the H\"{o}lder inequality and \eqref{eq:2.21}, and
\begin{equation} \label{eq:5.6}
	\| \nabla F(u) \|_{2} \le C \| \nabla^{2} u(t)\|_{4}^{2} + C \| \nabla u \|_{\infty} \| \nabla^{3} u \|_{2} \le C(1+t)^{-\frac{13}{4}} \| u \|_{X_{1}}^{2}
\end{equation}
by the H\"{o}lder inequality, \eqref{eq:2.21} and \eqref{eq:2.22}.
By virtue of the Banach fixed point theorem, if we prove the estimates 
\begin{align}
	& \| \Phi[u] \|_{X_{1}} \le R, \label{eq:5.7} \\
	& \| \Phi[u]-\Phi[v] \|_{X_{1}} \le \frac{1}{2} \|u-v \|_{X_{1}} \label{eq:5.8}
\end{align}
for all $u,v \in \mathbb{B}_{R}$ with a suitable choice of $R>0$, we can conclude Theorem \ref{thm:1.1}.

Now, 
for the simplicity of the notation, we also denote $Y_{1}:= \{ \dot{H}^{3}  \cap \dot{W}^{1,1}  \}^{3} \times \{ \dot{H}^{1} \cap L^{1}  \}^{3}$.
Then, combining the estimates \eqref{eq:2.7}, \eqref{eq:2.8}, \eqref{eq:3.48}, \eqref{eq:3.49} and \eqref{eq:4.5}-\eqref{eq:4.8}, we arrive at the estimate
\begin{equation} \label{eq:5.9}
	\begin{split}
		 \| u_{lin}(t) \|_{X_{1}} 
		& =\sup_{t \ge 0}
		\{ (1+t)^{\frac{7}{4}}\| \nabla^{3} u_{lin}(t)  \|_{2} + (1+t)^{\frac{3}{4}} \| \nabla u_{lin}(t)  \|_{2} \\
		&
		\quad + (1+t)^{\frac{3}{4}} \| \partial_{t} u_{lin}(t)  \|_{2} + (1+t)^{\frac{5}{4}} \| \nabla \partial_{t} u_{lin}(t)  \|_{2} \} 
		 \le C_{0}  \|f_{0}, f_{1} \|_{Y_{1}}
	\end{split}
\end{equation}
for some $C_{0}>0$.

We next prove the estimate for the nonlinear term $\Phi_{N}[u](t) $.
Applying the estimates  \eqref{eq:2.8}, \eqref{eq:3.49}, \eqref{eq:4.5}, \eqref{eq:4.7} \eqref{eq:5.4} and \eqref{eq:5.6}, we see that
\begin{equation} \label{eq:5.10}
\begin{split}
& \left\| \nabla^{\alpha} \int_{0}^{t} K^{(\beta)}_{1}(t-\tau) \mathcal{R}_{a} \mathcal{R}_{b} F_{j}(u)(\tau) d \tau \right\|_{2} 
+
\left\| \nabla^{\alpha} \int_{0}^{t} K^{(\beta)}_{1}(t-\tau) F_{j}(u)(\tau) d \tau \right\|_{2} \\
& \le C  \int_{0}^{t} (1+t-\tau)^{-\frac{1}{4}-\frac{\alpha}{2}} \| F(u)(\tau)\|_{1} d \tau 
+
C  \int_{0}^{t} e^{-c(t-\tau)} \| \nabla F(u)(\tau)\|_{2} d \tau \\
& \le C  \int_{0}^{t} (1+t-\tau)^{-\frac{1}{4}-\frac{\alpha}{2}} (1+\tau)^{-2} d \tau \| u \|_{X_{1}}^{2} 
+
C  \int_{0}^{t} e^{-c(t-\tau)} (1+\tau)^{-\frac{13}{4}} d \tau \| u \|_{X_{1}}^{2} \\
& \le C(1+t)^{-\frac{1}{4}-\frac{\alpha}{2}} \| u \|_{X_{1}}^{2}, 
\end{split}
\end{equation}
where we used the fact that $-\frac{1}{4}-\frac{\alpha}{2}<-2$ for $1\le \alpha \le 3$.
On the other hand, observing that 
\begin{equation*}
\begin{split}
\partial_{t} \int_{0}^{t} K^{(\beta)}_{1}(t-\tau) \mathcal{R}_{a} \mathcal{R}_{b} F_{j}(u)(\tau) d \tau 
=
\int_{0}^{t} \partial_{t} K^{(\beta)}_{1}(t-\tau) \mathcal{R}_{a} \mathcal{R}_{b} F_{j}(u)(\tau) d \tau,
\end{split}
\end{equation*}
we easily have
\begin{equation} \label{eq:5.11}
\begin{split}
& \left\| \nabla^{\alpha}  \partial_{t} \int_{0}^{t} K^{(\beta)}_{1}(t-\tau) \mathcal{R}_{a} \mathcal{R}_{b} F_{j}(u)(\tau) d \tau \right\|_{2} 
+ \left\| \nabla^{\alpha}  \partial_{t} \int_{0}^{t} K^{(\beta)}_{1}(t-\tau)  F_{j}(u)(\tau) d \tau \right\|_{2} 
\\
& \le C(1+t)^{-\frac{3}{4}-\frac{\alpha}{2}} \| u \|_{X_{1}}^{2}
\end{split}
\end{equation}
for $0 \le \alpha \le 1$ by a similar way.
Summing up the estimates \eqref{eq:5.10} and \eqref{eq:5.11}, we get
\begin{equation} \label{eq:5.12}
	\begin{split}
		& \| \Phi_{N}[u](t) \|_{X_{1}} \\
		& =\sup_{t \ge 0}
		\{ (1+t)^{\frac{7}{4}}\| \nabla^{3} \Phi_{N}[u](t)  \|_{2}  + (1+t)^{\frac{3}{4}} \| \nabla \Phi_{N}[u](t)  \|_{2} \\
		& \quad + (1+t)^{\frac{3}{4}} \| \partial_{t} \Phi_{N}[u](t)  \|_{2} + (1+t)^{\frac{5}{4}} \| \nabla \partial_{t} \Phi_{N}[u](t)  \|_{2}  \} \le C_{1}  \|u \|_{X_{1}}^{2}
	\end{split}
\end{equation}
for some $C_{1}>0$.
Then it follows from \eqref{eq:5.9} and \eqref{eq:5.12} that 
\begin{equation*} 
\begin{split}
\| u \|_{X_{1}} \le C_{0} \| f_{0},f_{1} \|_{Y_{1}}+ C_{1}  \|u \|_{X_{1}}^{2}.
\end{split}
\end{equation*}
This implies that if we choose $R=C_{0} \| f_{0},f_{1} \|_{Y_{1}}$ with sufficiently small $\| f_{0},f_{1} \|_{Y_{1}}$, 
we have \eqref{eq:5.7}.
A similar argument applies to the case \eqref{eq:5.8}.
We omit the detail.
Now we have \eqref{eq:5.7} and \eqref{eq:5.8}, and by the Banach fixed point theorem, the mapping $\Phi$ has a unique fixed point on $\mathbb{B}_{R}$, 
which proves Theorem \ref{thm:1.1}.   
%
%
\section{Smoothing effect of the global solutions}
In this section, we prove Theorem \ref{thm:1.3}.
	%
%
%
\begin{proof}[Proof of Theorem \ref{thm:1.3}]
Let $F(u)=\nabla u \nabla^{2} u$.
We firstly prove the estimate \eqref{eq:1.3}.
The estimates \eqref{eq:2.7}, \eqref{eq:2.8}, \eqref{eq:3.48}, \eqref{eq:3.49}, \eqref{eq:4.5}, \eqref{eq:4.9} and \eqref{eq:4.12} give
\begin{equation} \label{eq:6.8}
\begin{split}
\| \nabla^{\alpha} u_{lin}(t)  \|_{\infty}
\le C (1+t)^{-\frac{3}{2}-\frac{\alpha}{2}} \| f_{0}, f_{1} \|_{Y_{1}} 
\end{split}
\end{equation}
for $0 \le \alpha \le 1$, 
where $Y_{1}:= \{ \dot{H}^{3}  \cap \dot{W}^{1,1}  \}^{3} \times \{ \dot{H}^{1} \cap L^{1}  \}^{3}$.

The estimate for the nonlinear term is shown as follows:
\begin{equation} \label{eq:6.9}
	\begin{split}
		& \left\| \nabla^{\alpha} \int_{0}^{t} K^{(\beta)}_{1}(t-\tau) \mathcal{R}_{a} \mathcal{R}_{b} F_{j}(u)(\tau) d \tau \right\|_{\infty} 
		+ \left\| \nabla^{\alpha}  \int_{0}^{t} K^{(\beta)}_{1}(t-\tau)  F_{j}(u)(\tau) d \tau \right\|_{\infty} 
		\\
		& \le C  \int_{0}^{\frac{t}{2}} (1+t-\tau)^{-\frac{3}{2}-\frac{\alpha}{2}} \| F_{j}(u)(\tau)\|_{1} d \tau + C  \int_{\frac{t}{2}}^{t} (1+t-\tau)^{-\frac{1}{4}-\frac{\alpha}{2}} \| F_{j}(u)(\tau)\|_{2} d \tau\\
		& +
		C  \int_{0}^{t} e^{-c(t-\tau)} \| \nabla F_{j}(u)(\tau)\|_{2} d \tau \\
		& \le C  \int_{0}^{\frac{t}{2}} (1+t-\tau)^{-\frac{3}{2}-\frac{\alpha}{2}} (1+\tau)^{-2} d \tau 
		+
		C \int_{\frac{t}{2}}^{t} (1+t-\tau)^{-\frac{1}{4}-\frac{\alpha}{2}} (1+\tau)^{-\frac{11}{4}} d \tau \\
		& + C  \int_{0}^{t} e^{-c(t-\tau)}  (1+\tau)^{-\frac{13}{4}} d \tau  \\
		& \le C(1+t)^{-\frac{3}{2}-\frac{\alpha}{2}},
	\end{split}
\end{equation}
where $0 \le \alpha \le 1$.
Therefore, summing up the estimates \eqref{eq:6.8} and \eqref{eq:6.9}, 
we obtain the estimate \eqref{eq:1.3}.
%

	Secondly, we prove the estimate \eqref{eq:1.4}.
	We firstly claim that 
\begin{equation} \label{eq:6.1}
	\begin{split}
		\| \nabla^{2} \partial_{t} u_{lin}(t) \|_{p} 
		\le 
		C(1+t)^{-\frac{3}{2}-\frac{1}{2p}} t^{-3(\frac{1}{2}-\frac{1}{p})} \| f_{0}, f_{1} \|_{Y_{1}}
	\end{split}
\end{equation} 
	for $2 \le p \le 6$ and $t>0$.
	Now we apply the estimate \eqref{eq:4.1} to see that
	\begin{equation*} 
		\begin{split}
			\| \nabla^{2} \partial_{t} K_{0H}^{(\beta)}(t) g \|_{6} +
			\| \nabla^{2} \partial_{t}  K_{0H}^{(\beta)}(t) \mathcal{R}_{a}  \mathcal{R}_{b} g \|_{6} 
			& \le C e^{-ct}(\| \nabla^{2}  g \|_{6} +\| \nabla^{3}  g \|_{2} ) \le C e^{-ct} \| \nabla^{3}  g \|_{2}. 
		\end{split}
	\end{equation*}
	Thus noting the estimates \eqref{eq:4.6} and $\| f \|_{p} \le \| f \|_{2}^{\frac{3}{p}-\frac{1}{2}} \| f \|_{6}^{3(\frac{1}{2}-\frac{1}{p})}$ for $2 \le p \le 6$, 
	we obtain 
	\begin{equation} \label{eq:6.2}
		\begin{split}
			& \| \nabla^{2} \partial_{t} K_{0H}^{(\beta)}(t) g \|_{p} +
			\| \nabla^{2} \partial_{t}  K_{0H}^{(\beta)}(t) \mathcal{R}_{a}  \mathcal{R}_{b} g \|_{p} \le C e^{-ct} \| \nabla^{3}  g \|_{2}. 
		\end{split}
	\end{equation}
	Then it follows from \eqref{eq:3.48}, \eqref{eq:4.5} and \eqref{eq:6.2} that
\begin{equation} \label{eq:6.4}
	\begin{split}
		\| \nabla^{2} \partial_{t} K_{0}^{(\beta)}(t) g \|_{p} +
		\| \nabla^{2} \partial_{t}  K_{0}^{(\beta)}(t) \mathcal{R}_{a}  \mathcal{R}_{b} g \|_{p}  
		\le C  (1+t)^{-\frac{5}{2}(1-\frac{1}{p})-\frac{1}{2}} (\| \nabla g \|_{1} +\| \nabla^{3} g \|_{2}).
	\end{split}
\end{equation}
	Similarly, using \eqref{eq:4.2}, we also get 
	\begin{equation*} 
		\begin{split}
			\| \nabla^{2} \partial_{t} K_{1H}^{(\beta)}(t) g \|_{6} +
			\| \nabla^{2} \partial_{t}  K_{1H}^{(\beta)}(t) \mathcal{R}_{a}  \mathcal{R}_{b} g \|_{6} 
			& \le C e^{-ct}(\| g \|_{6} +t^{-\frac{3-\alpha}{2} } \| \nabla^{\alpha} g \|_{2} ) \\
			& 
			\le C e^{-ct}(1+t^{-\frac{3-\alpha}{2} }) \| \nabla^{\alpha}  g \|_{2} 
		\end{split}
	\end{equation*}
	for $1 \le \alpha \le 2$.
	Thus, when $2 \le p \le 6$, we see that
	\begin{equation} \label{eq:6.3}
		\begin{split}
			& \| \nabla^{2} \partial_{t} K_{1H}^{(\beta)}(t) g \|_{p} +
			\| \nabla^{2} \partial_{t}  K_{1H}^{(\beta)}(t) \mathcal{R}_{a}  \mathcal{R}_{b} g \|_{p} 
			\le C e^{-ct}(1+t^{-\frac{7}{4}+\frac{3}{2p}+\frac{\alpha}{2} }) \| \nabla^{\alpha} g \|_{2} 
		\end{split}
	\end{equation}
	for $1 \le \alpha \le 2$ by \eqref{eq:4.12} and \eqref{eq:4.15}. 
The estimates \eqref{eq:3.49}, \eqref{eq:4.15} and \eqref{eq:6.3} show
	\begin{equation} \label{eq:6.5}
		\begin{split}
			& \| \nabla^{2} \partial_{t} K_{1}^{(\beta)}(t) g \|_{p} +
			\| \nabla^{2} \partial_{t}  K_{1}^{(\beta)}(t) \mathcal{R}_{a}  \mathcal{R}_{b} g \|_{p} 
			\\
			& \le C (1+t)^{-\frac{5}{4}+\frac{1}{p}-\frac{\alpha}{2}} t^{-\frac{7}{4}+\frac{3}{2p}+\frac{\alpha}{2} } (\|  g \|_{1} +\| \nabla^{\alpha} g \|_{2})
		\end{split}
	\end{equation}
	for $1 \le \alpha \le 2$.
	The estimates \eqref{eq:6.4} and \eqref{eq:6.5} imply the estimate \eqref{eq:6.1}.
	For the nonlinear term, 
	observing that  
	\begin{equation}  \label{eq:6.6} 
		\begin{split}
			\| F(u)(\tau)\|_{p} & \le \| F(u)(\tau)\|_{2}^{\frac{3}{p}-\frac{1}{2}} \| F(u)(\tau)\|_{6}^{3(\frac{1}{2}-\frac{1}{p})}  \\
			& \le C \| F(u)(\tau)\|_{2}^{\frac{3}{p}-\frac{1}{2}} \| \nabla F(u)(\tau)\|_{2}^{3(\frac{1}{2}-\frac{1}{p})}  
			\le C(1+\tau)^{-\frac{7}{2}+\frac{3}{2p}}
		\end{split}
	\end{equation}
	for $2 \le p \le 6$ by the H\"oloder inequality, \eqref{eq:2.23}, \eqref{eq:5.5} and \eqref{eq:5.6}, 
	we apply the estimates \eqref{eq:2.8}, \eqref{eq:3.49}, \eqref{eq:4.2} and \eqref{eq:4.5} to have
	\begin{equation} \label{eq:6.7} 
		\begin{split}
			& \left\| \nabla^{2} \partial_{t} \int_{0}^{t} K^{(\beta)}_{1}(t-\tau) \mathcal{R}_{a} \mathcal{R}_{b} F_{j}(u)(\tau) d \tau \right\|_{p} 
			+
			\left\| \nabla^{2} \partial_{t} \int_{0}^{t} K^{(\beta)}_{1}(t-\tau) F_{j}(u)(\tau) d \tau \right\|_{p} \\
			& \le C \int_{0}^{\frac{t}{2}} (1+t-\tau)^{-\frac{5}{2}(1-\frac{1}{p})-\frac{1}{2}} \| F(u)(\tau)\|_{1} d \tau \\
			& +C \int_{\frac{t}{2}}^{t} (1+t-\tau)^{-\frac{5}{2}(\frac{1}{2}-\frac{1}{p})} \| \nabla F(u)(\tau)\|_{2} d \tau \\
			& + C \int_{0}^{t} e^{-c(t-\tau)}\{ \| F(u)(\tau)\|_{p} +(t-\tau)^{-\frac{3}{2}(\frac{1}{2}-\frac{1}{p})-\frac{1}{2}}\| \nabla F(u)(\tau)\|_{2} \} d \tau \\
			& \le C  \int_{0}^{\frac{t}{2}} (1+t-\tau)^{-\frac{5}{2}(1-\frac{1}{p})-\frac{1}{2}} (1+\tau)^{-2} d \tau 
			+ C \int_{\frac{t}{2}}^{t} (1+t-\tau)^{-\frac{5}{2}(\frac{1}{2}-\frac{1}{p})} (1+\tau)^{-\frac{13}{4}} d \tau \\
			& +
			C  \int_{0}^{t} e^{-c(t-\tau)} \{ 
			(1+\tau)^{-\frac{7}{2}+\frac{3}{2p}}+(t-\tau)^{-\frac{3}{2}(\frac{1}{2}-\frac{1}{p})-\frac{1}{2}} (1+\tau)^{-\frac{13}{4}}
				 \} d \tau \\
			& \le C(1+t)^{-\frac{5}{2}(1-\frac{1}{p})-\frac{1}{2}}
			\end{split}
	\end{equation}
	for $2 \le p <6$,
	since  we used the fact that $-\frac{3}{2}(\frac{1}{2}-\frac{1}{p})-\frac{1}{2} >-1$.
	Combining the estimates \eqref{eq:6.1} and \eqref{eq:6.7},
	we arrive at the desired estimate \eqref{eq:1.4}.

Next, we prove the estimate \eqref{eq:1.5}.
We apply a similar argument to the derivation of \eqref{eq:6.8} to have
\begin{equation} \label{eq:6.10}
	\begin{split}
		\| \nabla^{\alpha} \partial_{t} u_{lin}(t)  \|_{\infty}
\le C (1+t)^{-\frac{3}{2}-\frac{\alpha}{2}} \| f_{0}, f_{1} \|_{Y_{1}} 
		& 
		\le C (1+t)^{-\frac{7}{4}} t^{-\frac{1}{4}-\frac{\alpha}{2}} \| f_{0}, f_{1} \|_{Y_{1}}
	\end{split}
\end{equation}
for $0 \le \alpha \le 1$.
The estimate for the nonlinear term is obtained as in \eqref{eq:6.9};
\begin{equation} \label{eq:6.11}
	\begin{split}
		& \left\| \nabla^{\alpha}  \partial_{t} \int_{0}^{t} K^{(\beta)}_{1}(t-\tau) \mathcal{R}_{a} \mathcal{R}_{b} F_{j}(u)(\tau) d \tau \right\|_{\infty} 
		+ \left\| \nabla^{\alpha}  \partial_{t} \int_{0}^{t} K^{(\beta)}_{1}(t-\tau)  F_{j}(u)(\tau) d \tau \right\|_{\infty} \\
		& \le C(1+t)^{-2-\frac{\alpha}{2}}
	\end{split}
\end{equation}
for $0 \le \alpha \le 1$ by \eqref{eq:2.8}, \eqref{eq:3.49}, \eqref{eq:4.5} and \eqref{eq:4.13}.
Combining the estimates \eqref{eq:6.10}-\eqref{eq:6.11} with \eqref{eq:5.1}, 
we arrive at the estimate \eqref{eq:1.5}.
Finally, we show the estimate \eqref{eq:1.6}.
Again,
the estimates \eqref{eq:2.7}, \eqref{eq:2.8}, \eqref{eq:3.48}, \eqref{eq:3.49}, \eqref{eq:4.5}, \eqref{eq:4.10} and \eqref{eq:4.13} yield
\begin{equation} \label{eq:6.12}
	\begin{split}
		\| \partial_{t}^{2} u_{lin}(t) \|_{p} 
		\le C (1+t)^{-(1-\frac{1}{p})-\frac{1}{4}} t^{-\frac{3}{2}(\frac{1}{2}-\frac{1}{p})-\frac{1}{2}} \| f_{0}, f_{1} \|_{Y_{1}}
	\end{split}
\end{equation}
for $2 \le p \le 6$.
To estimate the nonlinear term, 
noting that $\partial_{t} \mathcal{K}^{(\beta)}_{1}(0,\xi)=1$,
$$
\partial_{t}^{2} \int_{0}^{t} K^{(\beta)}_{1}(t-\tau) \mathcal{R}_{a} \mathcal{R}_{b} F_{j}(u)(\tau) d \tau
= \mathcal{R}_{a} \mathcal{R}_{b} F_{j}(u)(t) +\int_{0}^{t} \partial_{t}^{2}  K^{(\beta)}_{1}(t-\tau) \mathcal{R}_{a} \mathcal{R}_{b} F_{j}(u)(\tau) d \tau, 
$$
$$
\partial_{t}^{2} \int_{0}^{t} K^{(\beta)}_{1}(t-\tau) F_{j}(u)(\tau) d \tau
=  F_{j}(u)(t) +\int_{0}^{t} \partial_{t}^{2}  K^{(\beta)}_{1}(t-\tau)  F_{j}(u)(\tau) d \tau, 
$$
and \eqref{eq:6.6},
we get
\begin{equation} \label{eq:6.13}
	\begin{split}
		& \left\| \partial_{t}^{2} \int_{0}^{t} K^{(\beta)}_{1}(t-\tau) \mathcal{R}_{a} \mathcal{R}_{b} F_{j}(u)(\tau) d \tau \right\|_{p} 
		+ \left\| \partial_{t}^{2} \int_{0}^{t} K^{(\beta)}_{1}(t-\tau)  F_{j}(u)(\tau) d \tau \right\|_{p} 
		\\
		& \le C \| F_{j}(u)(t) \|_{p} \\
		& +C \int_{0}^{\frac{t}{2}} (1+t-\tau)^{-\frac{5}{2}(1-\frac{1}{p})} \| F(u)(\tau)\|_{1} d \tau 
		 +C \int_{\frac{t}{2}}^{t}   (1+t-\tau)^{-\frac{5}{2}(\frac{1}{2}-\frac{1}{p})}  \| F(u)(\tau)\|_{2} d \tau \\
		& +C  \int_{0}^{t} e^{-c(t-\tau)}(\| F(u)(\tau)\|_{p} +(t-\tau)^{-\frac{3}{2}(\frac{1}{2}-\frac{1}{p})-\frac{1}{2}} \| \nabla F(u)(\tau)\|_{2} ) d \tau \\
		& \le 
		C (1+t)^{-\frac{5}{2}(1-\frac{1}{p})-1-\frac{1}{p}} + C  \int_{0}^{\frac{t}{2}} (1+t-\tau)^{-\frac{5}{2}(1-\frac{1}{p})} (1+\tau)^{-2} d \tau \\
		& + C  \int_{\frac{t}{2}}^{t}   (1+t-\tau)^{-\frac{5}{2}(\frac{1}{2}-\frac{1}{p})} (1+\tau)^{-\frac{11}{4}} d \tau \\
		&+
		C  \int_{0}^{t} e^{-c(t-\tau)} \{ (1+\tau)^{-\frac{5}{2}(1-\frac{1}{p})-1-\frac{1}{p}}  +(t-\tau)^{-\frac{3}{2}(\frac{1}{2}-\frac{1}{p})-\frac{1}{2}} (1+\tau)^{-\frac{13}{4}} \}  d \tau  \\
		& \le C(1+t)^{-\frac{5}{2}(1-\frac{1}{p})}
	\end{split}
\end{equation}
by \eqref{eq:2.8}, \eqref{eq:3.49}, \eqref{eq:4.5}, \eqref{eq:4.15}, \eqref{eq:5.4}-\eqref{eq:5.6}, where $2 \le p <6$.
Therefore we can conclude the desired estimate \eqref{eq:1.6} by the combination of \eqref{eq:6.12} and \eqref{eq:6.13}.
We complete the proof of Theorem \ref{thm:1.3}. 
\end{proof}
%
%
\section{Asymptotic profiles of the solutions}
In this section, we study the large time behavior of the solutions obtained in Theorems \ref{thm:1.1} and \ref{thm:1.3}.
We first deal with asymptotic profiles of the linear solution $u_{lin}(t)$ as $t \to \infty$, 
which is simply reformulation of the linear estimates proved in previous sections.
Secondly, we describe asymptotic profiles of nonlinear terms in the integral equation \eqref{eq:5.1}. 
We conclude this section by proving Theorem \ref{thm:1.5}.

In the sequel, we use the following notation for the simplicity.
\begin{equation} \label{eq:7.1}
	\begin{split}
		G(t,x) & = G_{0,lin}(t,x) +G_{1,lin}(t,x)+G_{N}(t,x),\\
		H(t,x) & = H_{0,lin}(t,x) +H_{1,lin}(t,x) +H_{N}(t,x),\\
		\tilde{G}(t,x) & = \tilde{G}_{0,lin}(t,x)+\tilde{G}_{1,lin}(t,x) +\tilde{G}_{N}(t,x),
	\end{split}
\end{equation}
where 
\begin{equation*}
	\begin{split}
		G_{0,lin}(t,x)& := 
		\nabla^{-1} \mathcal{F}^{-1} \left[ 
		\left( \mathcal{G}^{(\sqrt{\lambda+2 \mu})}_{0}(t,\xi) -\mathcal{G}^{(\sqrt{\mu})}_{0}(t,\xi) \right)\mathcal{P} m_{0}
		+
		\mathcal{G}^{(\sqrt{\mu})}_{0}(t,\xi) m_{0}
		\right], \\
		G_{1,lin}(t,x)& := 
		\mathcal{F}^{-1} \left[ 
		\left( \mathcal{G}^{(\sqrt{\lambda+2 \mu})}_{1}(t,\xi) -\mathcal{G}^{(\sqrt{\mu})}_{1}(t,\xi) \right)\mathcal{P} m_{1}
		+
		\mathcal{G}^{(\sqrt{\mu})}_{1}(t,\xi) m_{1}
		\right], \\ 
		 G_{N}(t,x)& := 
		\mathcal{F}^{-1} \left[ 
		\left( \mathcal{G}^{(\sqrt{\lambda+2 \mu})}_{1}(t,\xi) -\mathcal{G}^{(\sqrt{\mu})}_{1}(t,\xi) \right)\mathcal{P} M[u]
		+
		\mathcal{G}^{(\sqrt{\mu})}_{1}(t,\xi) M[u]
		\right], 
	\end{split}
\end{equation*}
\begin{equation*} 
	\begin{split}
		& H_{0,lin}(t,x)\\
		& := -\Delta 
		\nabla^{-1} \mathcal{F}^{-1} \left[ 
		\left( (\lambda+2 \mu)  
		\mathcal{G}^{(\sqrt{\lambda+2 \mu})}_{1}(t,\xi) 
		-\mu \mathcal{G}^{(\sqrt{\mu} )}_{1}(t,\xi) 
		\right)\mathcal{P} m_{0}
		+ \mu \mathcal{G}^{(\sqrt{\mu} )}_{1}(t,\xi) m_{0}
		\right], 
	\end{split}
\end{equation*}
\begin{equation*} 
	\begin{split}
		H_{1,lin}(t,x)& := 
		\mathcal{F}^{-1} \left[ 
		\left(  \mathcal{G}^{(\sqrt{\lambda+2 \mu})}_{0}(t,\xi) -\mathcal{G}^{(\sqrt{\mu} )}_{0}(t,\xi) \right) \mathcal{P} m_{1}
		+
		\mathcal{G}^{(\sqrt{\mu} )}_{0}(t,\xi) m_{1} 
		\right], \\
		H_{N}(t,x)& := 
		\mathcal{F}^{-1} \left[ 
		\left( \mathcal{G}^{(\sqrt{\lambda+2 \mu})}_{0}(t,\xi) -\mathcal{G}^{(\sqrt{\mu})}_{0}(t,\xi) \right)\mathcal{P} M[u]
		+
		\mathcal{G}^{(\sqrt{\mu})}_{0}(t,\xi) M[u]
		\right], 
	\end{split}
\end{equation*}
and
\begin{equation*} 
	\begin{split}
		& \tilde{G}_{0,lin}(t,x) \\
		& := -\Delta \nabla^{-1}
		\mathcal{F}^{-1} \left[
		\left(
		(\lambda+2 \mu) 
		\mathcal{G}^{(\sqrt{\lambda+2 \mu})}_{0}(t,\xi) 
		- \mu  \mathcal{G}^{(\sqrt{\mu} )}_{0}(t,\xi) 
		\right)\mathcal{P} m_{0}
		+ \mu  
		\mathcal{G}^{(\sqrt{\mu} )}_{0}(t,\xi) m_{0}
		\right], 
	\end{split}
\end{equation*}
\begin{equation*} 
	\begin{split}
\tilde{G}_{1,lin}(t,x)& := -\Delta
\mathcal{F}^{-1} \left[
\left(
(\lambda+2 \mu) 
\mathcal{G}^{(\sqrt{\lambda+2 \mu})}_{1}(t,\xi) 
- \mu  \mathcal{G}^{(\sqrt{\mu} )}_{1}(t,\xi) 
\right)\mathcal{P} m_{1} 
+ \mu  
\mathcal{G}^{(\sqrt{\mu} )}_{1}(t,\xi) m_{1}
\right], 
	\end{split}
\end{equation*}
\begin{equation*} 
	\begin{split}
		& \tilde{G}_{N}(t,x)\\
		& := -\Delta
		\mathcal{F}^{-1} \left[
		\left(
		(\lambda+2 \mu) 
		\mathcal{G}^{(\sqrt{\lambda+2 \mu})}_{1}(t,\xi) 
		- \mu  \mathcal{G}^{(\sqrt{\mu} )}_{1}(t,\xi) 
		\right)\mathcal{P} M[u] 
		+ \mu  
		\mathcal{G}^{(\sqrt{\mu} )}_{1}(t,\xi) M[u]
		\right].
	\end{split}
\end{equation*}
\subsection{Asymptotic profiles of the linear solutions}
In this subsection, our goal is to show that $u_{lin}(t)$ (resp. $ \partial_{t} u_{lin}(t)$, $\partial_{t}^{2} u_{lin}(t)$) is approximated by $G_{0, lin}(t)+G_{1, lin}(t)$ 
(resp. $H_{0, lin}(t)+H_{1, lin}(t)$, $\tilde{G}_{0, lin}(t)+\tilde{G}_{1, lin}(t)$) as $t \to \infty$.
For this purpose, the following proposition plays an essential role. 
\begin{prop} \label{prop:7.1}
		Let $\alpha, \ell, m \ge 0$ and $g \in L^{1}$.
		Then it holds that 
		\begin{equation} \label{eq:7.2}
			\begin{split}
				& \left\| \nabla^{\alpha} \left(
				\partial^{\ell}_{t} K_{0}^{(\beta)}(t) \mathcal{R}_{a} \mathcal{R}_{b} g - m_{g}(-1)^{\frac{\ell}{2}} \beta^{\ell} \nabla^{\ell} \mathcal{R}_{a} \mathcal{R}_{b} G_{0}^{(\beta)}(t) 
				\right)
				\right\|_{p} \\
				&
				\le o(t^{-\frac{3}{2}(1-\frac{1}{p})-(1-\frac{1}{p})+\frac{1}{2}-\frac{\ell+\alpha}{2}})  +  C\| \nabla^{\alpha}  
				\partial^{\ell}_{t} K_{0H}^{(\beta)}(t) \mathcal{R}_{a} \mathcal{R}_{b} g \|_{p} 
			\end{split}
		\end{equation}
		for $\ell=2m$,
		\begin{equation} \label{eq:7.3}
			\begin{split}
				& \left\| \nabla^{\alpha} \left(
				\partial^{\ell}_{t} K_{0}^{(\beta)}(t) \mathcal{R}_{a} \mathcal{R}_{b}  g - m_{g} (-1)^{\frac{\ell+1}{2}} \beta^{\ell+1} \nabla^{\ell+1} \mathcal{R}_{a} \mathcal{R}_{b} G_{1}^{(\beta)}(t) 
				\right)
				\right\|_{p} \\
				&
				\le o(t^{-\frac{3}{2}(1-\frac{1}{p})-(1-\frac{1}{p})+\frac{1}{2}-\frac{\ell+\alpha}{2}}) + C\| \nabla^{\alpha}  
				\partial^{\ell}_{t} K_{0H}^{(\beta)}(t) \mathcal{R}_{a} \mathcal{R}_{b} g \|_{p} 
			\end{split}
		\end{equation}
		for $\ell=2m+1$,
		\begin{equation} \label{eq:7.4}
			\begin{split}
				& \left\| \nabla^{\alpha} \left(
				\partial^{\ell}_{t} K_{1}^{(\beta)}(t) \mathcal{R}_{a} \mathcal{R}_{b}  g -  m_{g} (-1)^{\frac{\ell}{2}} \beta^{\ell} \nabla^{\ell} \mathcal{R}_{a} \mathcal{R}_{b} G_{1}^{(\beta)}(t) 
				\right)
				\right\|_{p} \\
				&
				\le o(t^{-\frac{3}{2}(1-\frac{1}{p})-(1-\frac{1}{p})+1-\frac{\ell+\alpha}{2} }) +  C\| \nabla^{\alpha}  
				\partial^{\ell}_{t} K_{1H}^{(\beta)}(t) \mathcal{R}_{a} \mathcal{R}_{b} g \|_{p} 
			\end{split}
		\end{equation}
		for $\ell=2m$ and 
		\begin{equation} \label{eq:7.5}
			\begin{split}
				& \left\| \nabla^{\alpha} \left(
				\partial^{\ell}_{t} K_{1}^{(\beta)}(t) \mathcal{R}_{a} \mathcal{R}_{b}  g - m_{g} (-1)^{\frac{\ell-1}{2}} \beta^{\ell-1} \nabla^{\ell-1} \mathcal{R}_{a} \mathcal{R}_{b} G_{0}^{(\beta)}(t) 
				\right)
				\right\|_{p} \\
				&
				\le o(t^{-\frac{3}{2}(1-\frac{1}{p})-(1-\frac{1}{p})+1-\frac{\ell+\alpha}{2} })
				+  C\| \nabla^{\alpha}  
				\partial^{\ell}_{t} K_{1H}^{(\beta)}(t) \mathcal{R}_{a} \mathcal{R}_{b} g \|_{p} 
			\end{split}
		\end{equation}
		for $\ell=2m+1$, 
		as $t \to \infty$, where $1<p \le \infty$ for $\ell+\alpha=0$ and $1 \le p \le \infty$ for $\ell+\alpha \ge 1$.
		Here $m_{g}$ is defined by \eqref{eq:3.56}.
\end{prop}
\begin{proof}[Proof of Proposition \ref{prop:7.1}]
	The estimates \eqref{eq:7.2}-\eqref{eq:7.5} are shown in a same way.
	Here we only prove the estimate \eqref{eq:7.2}.
	Now recalling
	\begin{equation*} 
		\begin{split}
			\mathbb{G}_{0}(t,x) = \mathcal{R}_{a} \mathcal{R}_{b} G_{0}^{(\beta)}(t) \ast \mathcal{F}^{-1}[\chi_{L}],
		\end{split}
	\end{equation*} 
	and 
	\begin{equation*} 
		\begin{split}
			\partial_{t}^{\ell} K_{0L}(t) \mathcal{R}_{a} \mathcal{R}_{b} g
			= \partial_{t}^{\ell} \mathbb{K}_{00}(t) \ast g -\Delta \partial_{t}^{\ell} \mathbb{K}_{1}(t) \ast g
		\end{split}
	\end{equation*} 
	by \eqref{eq:2.4},
	we decompose the integrand as 
	\begin{equation*} 
		\begin{split}
			& \nabla^{\alpha} 
			(\partial^{\ell}_{t} K_{0}^{(\beta)}(t) \mathcal{R}_{a} \mathcal{R}_{b} g - m_{g}(-1)^{\frac{\ell}{2}} \beta^{\ell} \nabla^{\ell} \mathcal{R}_{a} \mathcal{R}_{b} G_{0}^{(\beta)}(t) ) \\
			& = \nabla^{\alpha} 
			(\partial_{t}^{\ell} \mathbb{K}_{00}(t) \ast g -(-1)^{\frac{\ell}{2}} \beta^{\ell} \nabla^{\ell} \mathbb{G}_{0}(t) \ast g
			)
			+
			(-1)^{\frac{\ell}{2}} \beta^{\ell} \nabla^{\alpha+\ell} 
			( \mathbb{G}_{0}(t) \ast g-m_{g} \mathbb{G}_{0}(t)) \\
			& -\Delta \nabla^{\alpha} \partial_{t}^{\ell} \mathbb{K}_{1}(t) \ast g + \sum_{k=M,H} \nabla^{\alpha} 
			\partial^{\ell}_{t} K_{0k}^{(\beta)}(t) \mathcal{R}_{a} \mathcal{R}_{b} g \\
			& - m_{g}(-1)^{\frac{\ell}{2}} \beta^{\ell} \nabla^{\ell} \mathcal{R}_{a} \mathcal{R}_{b} \mathcal{F}^{-1}\left[ 
			\mathcal{G}_{0}^{(\beta)}(t, \xi)(\chi_{M}+\chi_{H}) 
			\right].
		\end{split}
	\end{equation*} 
	Then we take the $L^{p}$ norm to the both sides and apply the estimates \eqref{eq:3.4}, \eqref{eq:3.52}, \eqref{eq:3.49}, \eqref{eq:4.5} and \eqref{eq:2.11}
	to see that 
		\begin{equation*} 
		\begin{split}
			& \| \nabla^{\alpha} 
			(\partial^{\ell}_{t} K_{0}^{(\beta)}(t) \mathcal{R}_{a} \mathcal{R}_{b} g - m_{g}(-1)^{\frac{\ell}{2}} \beta^{\ell} \nabla^{\ell} \mathcal{R}_{a} \mathcal{R}_{b} G_{0}^{(\beta)}(t) ) \|_{p} \\
			& \le \| \nabla^{\alpha} 
			(\partial_{t}^{\ell} \mathbb{K}_{00}(t) \ast g -(-1)^{\frac{\ell}{2}} \beta^{\ell} \nabla^{\ell} \mathbb{G}_{0}(t) \ast g
			) \|_{p}
			+
			C \| \nabla^{\alpha+\ell} 
			( \mathbb{G}_{0}(t) \ast g-m_{g} \mathbb{G}_{0}(t)) \|_{p} \\
			& + \| \Delta \nabla^{\alpha} \partial_{t}^{\ell} \mathbb{K}_{1}(t) \ast g \|_{p} + \sum_{k=M,H} \| \nabla^{\alpha}  
			\partial^{\ell}_{t} K_{0k}^{(\beta)}(t) \mathcal{R}_{a} \mathcal{R}_{b} g \|_{p} \\
			& + C \| \mathcal{R}_{a} \mathcal{R}_{b} \mathcal{F}^{-1}\left[ 
			\mathcal{G}_{0}^{(\beta)}(t, \xi)(\chi_{M}+\chi_{H}) 
			\right] \|_{p} \\ 
			& \le C(1+t)^{-\frac{5}{2}(1-\frac{1}{p})-\frac{\alpha+\ell}{2}} \| g \|_{1}
			+ o(t^{-\frac{5}{2}(1-\frac{1}{p})+\frac{1}{2}-\frac{\alpha+\ell}{2}} )  + C e^{-ct} t^{-\frac{3}{2}(1-\frac{1}{p})-\frac{\alpha+\ell}{2}}\\
			& +  \| \nabla^{\alpha}  
			\partial^{\ell}_{t} K_{0H}^{(\beta)}(t) \mathcal{R}_{a} \mathcal{R}_{b} g \|_{p} \\
			& = o(t^{-\frac{5}{2}(1-\frac{1}{p})+\frac{1}{2}-\frac{\alpha+\ell}{2}} )  +  \| \nabla^{\alpha}  
			\partial^{\ell}_{t} K_{0H}^{(\beta)}(t) \mathcal{R}_{a} \mathcal{R}_{b} g \|_{p} 
		\end{split}
	\end{equation*} 
	as $t \to \infty$, which is the desired estimate \eqref{eq:7.2}.
	We complete the proof of Proposition \eqref{eq:7.1}.
\end{proof}
We can prove Corollary \ref{cor:3.7} by applying the same argument of the proof of Proposition \ref{Prop:3.1}.
\begin{cor} \label{cor:7.2}
	Let $\alpha, \ell, m \ge 0$, $1 \le p \le \infty$ and $g \in L^{1}$.
	Then it holds that 
	\begin{equation} \label{eq:7.6}
		\begin{split}
			& \left\| \nabla^{\alpha} \left(
			\partial^{\ell}_{t} K_{0}^{(\beta)}(t) g - m_{g}(-1)^{\frac{\ell}{2}} \beta^{\ell} G_{0}^{(\beta)}(t) 
			\right)
			\right\|_{p} \\
			&
			\le o(t^{-\frac{3}{2}(1-\frac{1}{p})-(1-\frac{1}{p})+\frac{1}{2}-\frac{\ell+\alpha}{2}}) 
			+  C\| \nabla^{\alpha}  
			\partial^{\ell}_{t} K_{0M}^{(\beta)}(t)  g \|_{p} 
			 +  C\| \nabla^{\alpha}  
			\partial^{\ell}_{t} K_{0H}^{(\beta)}(t)  g \|_{p} 
		\end{split}
	\end{equation}
	for $\ell=2m$,
	\begin{equation} \label{eq:7.7}
		\begin{split}
			& \left\| \nabla^{\alpha} \left(
			\partial^{\ell}_{t} K_{0}^{(\beta)}(t)  g - m_{g} (-1)^{\frac{\ell+1}{2}} \beta^{\ell+1} \nabla^{\ell+1}  G_{1}^{(\beta)}(t) 
			\right)
			\right\|_{p} \\
			&
			\le o(t^{-\frac{3}{2}(1-\frac{1}{p})-(1-\frac{1}{p})+\frac{1}{2}-\frac{\ell+\alpha}{2}}) 
			+  C\| \nabla^{\alpha}  
			\partial^{\ell}_{t} K_{0M}^{(\beta)}(t)  g \|_{p} 
			+ C\| \nabla^{\alpha}  
			\partial^{\ell}_{t} K_{0H}^{(\beta)}(t)  g \|_{p} 
		\end{split}
	\end{equation}
	for $\ell=2m+1$,
	\begin{equation} \label{eq:7.8}
		\begin{split}
			& \left\| \nabla^{\alpha} \left(
			\partial^{\ell}_{t} K_{1}^{(\beta)}(t)   g -  m_{g} (-1)^{\frac{\ell}{2}} \beta^{\ell} \nabla^{\ell} G_{1}^{(\beta)}(t) 
			\right)
			\right\|_{p} \\
			&
			\le o(t^{-\frac{3}{2}(1-\frac{1}{p})-(1-\frac{1}{p})+1-\frac{\ell+\alpha}{2} }) 
			+  C\| \nabla^{\alpha}  
			\partial^{\ell}_{t} K_{1M}^{(\beta)}(t)  g \|_{p} 
			+  C\| \nabla^{\alpha}  
			\partial^{\ell}_{t} K_{1H}^{(\beta)}(t)  g \|_{p} 
		\end{split}
	\end{equation}
	for $\ell=2m$ and 
	\begin{equation} \label{eq:7.9}
		\begin{split}
			& \left\| \nabla^{\alpha} \left(
			\partial^{\ell}_{t} K_{1}^{(\beta)}(t)  g - m_{g} (-1)^{\frac{\ell-1}{2}} \beta^{\ell-1} \nabla^{\ell-1} G_{0}^{(\beta)}(t) 
			\right)
			\right\|_{p} \\
			&
			\le o(t^{-\frac{3}{2}(1-\frac{1}{p})-(1-\frac{1}{p})+1-\frac{\ell+\alpha}{2} })
			+  C\| \nabla^{\alpha}  
			\partial^{\ell}_{t} K_{1M}^{(\beta)}(t)  g \|_{p} 
			+  C\| \nabla^{\alpha}  
			\partial^{\ell}_{t} K_{1H}^{(\beta)}(t)  g \|_{p} 
		\end{split}
	\end{equation}
	for $\ell=2m+1$, 
	as $t \to \infty$.
\end{cor}
\begin{rem}
	Here we note that the estimates \eqref{eq:7.6}-\eqref{eq:7.9} for the case $p=2$ and $\ell=0,1$ is already proved in \cite{I}.
\end{rem}
\begin{proof}[Proof of Corollaries \ref{cor:7.2}]
	The proof of Corollary \ref{cor:7.2} is shown by a similar way to Corollary \ref{cor:3.7}.
	The only difference is that we can deal with the case $\alpha+\ell=0$ and $p=1$, which comes from the facts that 
	\begin{equation} \label{eq:7.10}
		\begin{split}
			\left\| \partial^{\ell}_{t} \nabla^{\alpha} 
			G_{0L}^{(\beta)}(t) 
			\right\|_{p} 
			\le
			C (1+t)^{
				-\frac{3}{2}(1-\frac{1}{p})-(1-\frac{1}{p})+\frac{1}{2}-\frac{\ell+\alpha}{2}},
		\end{split}
	\end{equation}
	\begin{equation}  \label{eq:7.11}
		\begin{split}
			\left\| \partial^{\ell}_{t} \nabla^{\alpha} 
			G_{1L}^{(\beta)}(t) 
			\right\|_{p} 
			\le C (1+t)^{-\frac{3}{2}(1-\frac{1}{p})-(1-\frac{1}{p})+1-\frac{\ell+\alpha}{2} 
			}
		\end{split}
	\end{equation}
	for $t \ge 0$, where $1 \le p \le \infty$ for $\ell+\alpha \ge 0$.
	Indeed, 
	it is easy to see that 
	\begin{equation*} 
		\begin{split}
			\left\| 
			G_{0L}^{(\beta)}(t) 
			\right\|_{1} 
			& =
			\| 
			W_{0}^{(\beta)}(t)  \mathcal{F}^{-1} [e^{-\frac{\nu t|\xi|^{2}}{2}} \chi_{L} ]\|_{1} \\
			& \le 
			C \| 
			\mathcal{F}^{-1} [e^{-\frac{\nu t|\xi|^{2}}{2}} \chi_{L}  ]\|_{1} 
		 +
			C t \| 
			\nabla  \mathcal{F}^{-1} [e^{-\frac{\nu t|\xi|^{2}}{2}} \chi_{L} ]\|_{1} \\
			& \le C  + (1+t)^{\frac{1}{2}} \le C (1+t)^{\frac{1}{2}} 
		\end{split}
	\end{equation*}
	and 
\begin{equation*} 
	\begin{split}
		\left\| 
		G_{1L}^{(\beta)}(t) 
		\right\|_{1} 
		 =
		\| 
		 W_{1}^{(\beta)}(t) \mathcal{F}^{-1} [e^{-\frac{\nu t|\xi|^{2}}{2}} \chi_{L} ]\|_{1}
		\le 
		C t \| 
		 \mathcal{F}^{-1} [e^{-\frac{\nu t|\xi|^{2}}{2}} \chi_{L} ]\|_{1} 
		\le C t
	\end{split}
\end{equation*}
	 by \eqref{eq:2.17} and \eqref{eq:2.18}.
\end{proof}
Now we are in a position to prove the asymptotic behavior of the linear solution corresponding to \eqref{eq:1.1}. 
\begin{cor} \label{cor:7.4}
	Let $(f_{0}, f_{1}) \in \{ \dot{H}^{3} \cap \dot{W}^{1,1} \}^{3} \times  \{H^{1} \cap {L}^{1} \}^{3}$.
	Then the following estimates hold:  
	\begin{align}
		& \| \nabla^{\alpha} (u_{lin}(t)-G_{lin}(t)) \|_{2} = o(t^{-\frac{1}{4}-\frac{\alpha}{2}}), \quad 1 \le \alpha \le 3, \label{eq:7.12} \\
		& \| \nabla^{\alpha} (u_{lin}(t)-G_{lin}(t)) \|_{\infty} =o(t^{-\frac{3}{2}-\frac{\alpha}{2}}), \quad 0 \le \alpha \le 1, \label{eq:7.13}  \\
		& \| \nabla^{\alpha} (\partial_{t}u_{lin}(t) -H_{lin}(t)) \|_{2} =o(t^{-\frac{3}{4}-\frac{\alpha}{2}}), \quad 0 \le \alpha \le 2, \label{eq:7.14}  \\
		& \| \nabla^{2} (\partial_{t} u_{lin}(t) -H_{lin}(t)) \|_{p} =o( t^{-\frac{5}{2}(1-\frac{1}{p})-\frac{1}{2}} ), \quad 2 \le p \le 6, \label{eq:7.15}  \\
		& \| \nabla^{\alpha} (\partial_{t} u_{lin}(t) -H_{lin}(t))\|_{\infty}
		=o( t^{-2-\frac{\alpha}{2}}),\quad 0 \le \alpha \le 1, \label{eq:7.16}  \\
		& \| \partial_{t}^{2} u_{lin}(t) -\tilde{G}_{lin}(t) \|_{2} = o( t^{-\frac{5}{2}(1-\frac{1}{p})}), \quad 2 \le p \le 6 \label{eq:7.17} 
	\end{align}
	as $t \to \infty$.
\end{cor}
\begin{proof}[Proof of Corollary \ref{cor:7.4}]
	At first, we decompose the linear solution $u_{lin}(t)$ into two parts:  
	\begin{equation*} 
		\begin{split}
			u_{lin}(t) =u_{0,lin}(t)+u_{1,lin}(t),
		\end{split}
	\end{equation*}
	where $u_{j,lin}(t) :=(K_{j}^{(\sqrt{\lambda+2 \mu})}(t)-  K_{j}^{(\sqrt{\mu})}(t)) \mathcal{F}^{-1} [\mathcal{P} \hat{f}_{j}] + K_{j}^{(\sqrt{\mu})}(t)f_{j}$ for $j=0,1$.
	When $\alpha \ge 1$, 
	we observe that 
	\begin{equation*} 
		\begin{split}
			\nabla^{\alpha} K_{0}^{(\beta)}(t)\mathcal{F}^{-1} [\mathcal{P} \hat{f}_{0}] 
			= \sum_{k=1}^{3} \nabla^{\alpha-1} K_{0}^{(\beta)}(t) 
			\begin{pmatrix}
				\mathcal{R}_{1} \mathcal{R}_{k} \nabla f_{0k} \\
				\mathcal{R}_{2} \mathcal{R}_{k} \nabla f_{0k} \\
				\mathcal{R}_{3} \mathcal{R}_{k} \nabla f_{0k} 
			\end{pmatrix}
		\end{split}
	\end{equation*}
	and 
	\begin{equation*} 
		\begin{split}
			\nabla^{\alpha} \nabla^{-1}  \mathcal{F}^{-1} \left[ 
			\mathcal{G}^{(\beta)}_{0}(t,\xi) \mathcal{P} m_{0}
			\right]
			= \sum_{k=1}^{3} \nabla^{\alpha-1} 
			\begin{pmatrix}
				m_{0k}\mathcal{R}_{1} \mathcal{R}_{k} G_{0}^{(\beta)}(t,x) \\
				m_{0k} \mathcal{R}_{2} \mathcal{R}_{k}G_{0}^{(\beta)}(t,x) \\
				m_{0k} \mathcal{R}_{3} \mathcal{R}_{k} G_{0}^{(\beta)}(t,x)
			\end{pmatrix}.
		\end{split}
	\end{equation*}
	In what follows, we denote the $j$-th component of the 3-dimensional vector $A$ by $(A)_{(j)}$ for $j=1,2,3$.
	Then we see that 
	\begin{equation} \label{eq:7.18}
		\begin{split}
			& \left\| \nabla^{\alpha} 
			\left\{  (K_{0}^{(\beta)}(t)\mathcal{F}^{-1} [\mathcal{P} \hat{f}_{0}])_{(j)}
			-  
			\left(
			\nabla^{-1}
			\mathcal{F}^{-1} \left[ 
			\mathcal{G}^{(\beta)}_{0}(t,\xi) \mathcal{P} m_{0}
			\right]
			\right)_{(j)}
			\right\} 
			\right\|_{2} \\
			& \le C \sum_{k=1}^{3} 
			\left\|
			\nabla^{\alpha-1} \{ 
			K_{0}^{(\beta)}(t) \mathcal{R}_{j} \mathcal{R}_{k} \nabla f_{0k}-m_{0k} \mathcal{R}_{j} \mathcal{R}_{k}G_{0}^{(\beta)}(t)
			\}
			\right\|_{2} \\
			& \le o(t^{-\frac{1}{4}-\frac{\alpha}{2}}) + C \sum_{k=1}^{3} \left\|
			\nabla^{\alpha} 
			K_{0H}^{(\beta)}(t) \mathcal{R}_{j} \mathcal{R}_{k} f_{0k}
			\right\|_{2}=o(t^{-\frac{1}{4}-\frac{\alpha}{2}})
		\end{split}
	\end{equation}
	as $t \to \infty$ for $0 \le \alpha \le 3$ by \eqref{eq:7.2} and \eqref{eq:4.6}.
	Similarly, noting that 
	\begin{equation*} 
		\begin{split}
			\nabla^{\alpha} K_{1}^{(\beta)}(t)\mathcal{F}^{-1} [\mathcal{P} \hat{f}_{0}] 
			= \sum_{k=1}^{3} \nabla^{\alpha} K_{1}^{(\beta)}(t) 
			\begin{pmatrix}
				\mathcal{R}_{1} \mathcal{R}_{k} f_{1k} \\
				\mathcal{R}_{2} \mathcal{R}_{k} f_{1k} \\
				\mathcal{R}_{3} \mathcal{R}_{k} f_{1k} 
			\end{pmatrix}
		\end{split}
	\end{equation*}
	and 
	\begin{equation*} 
		\begin{split}
			\nabla^{\alpha}  \mathcal{F}^{-1} \left[ 
			\mathcal{G}^{(\beta)}_{1}(t,\xi) \mathcal{P} m_{1}
			\right]
			= \sum_{k=1}^{3} \nabla^{\alpha} 
			\begin{pmatrix}
				m_{1k}\mathcal{R}_{1} \mathcal{R}_{k} G_{1}^{(\beta)}(t,x) \\
				m_{1k} \mathcal{R}_{2} \mathcal{R}_{k}G_{1}^{(\beta)}(t,x) \\
				m_{1k} \mathcal{R}_{3} \mathcal{R}_{k} G_{1}^{(\beta)}(t,x)
			\end{pmatrix},
		\end{split}
	\end{equation*}
	 we have
	\begin{equation} \label{eq:7.19}
		\begin{split}
			& \left\| \nabla^{\alpha} 
			\left\{  (K_{1}^{(\beta)}(t)\mathcal{F}^{-1} [\mathcal{P} \hat{f}_{1}])_{(j)}
			-  
			\left(\mathcal{F}^{-1} \left[ 
			\mathcal{G}^{(\beta)}_{1}(t,\xi) \mathcal{P} m_{1}
			\right]
			\right)_{(j)}
			\right\} 
			\right\|_{2} =o(t^{-\frac{1}{4}-\frac{\alpha}{2}})
		\end{split}
	\end{equation}
	as $t \to \infty$ for $0 \le \alpha \le 3$ by \eqref{eq:7.4} and \eqref{eq:4.7}.
	The remaining parts
	\begin{equation} \label{eq:7.20}
		\begin{split}
			& \left\| \nabla^{\alpha} 
			\left\{  (K_{0}^{(\beta)}(t)f_{0})_{(j)}
			-  
			\left(  \nabla^{-1} \mathcal{F}^{-1} \left[ 
			\mathcal{G}^{(\beta)}_{0}(t,\xi) m_{0}
			\right]
			\right)_{(j)}
			\right\} 
			\right\|_{2} =o(t^{-\frac{1}{4}-\frac{\alpha}{2}})
		\end{split}
	\end{equation}
	and 
	\begin{equation} \label{eq:7.21}
		\begin{split}
			& \left\| \nabla^{\alpha} 
			\left\{  (K_{1}^{(\beta)}(t) f_{1})_{(j)}
			-  
			\left(
			\mathcal{F}^{-1} \left[ 
			\mathcal{G}^{(\beta)}_{1}(t,\xi)  m_{1}
			\right]
			\right)_{(j)}
			\right\} 
			\right\|_{2} =o(t^{-\frac{1}{4}-\frac{\alpha}{2}})
		\end{split}
	\end{equation}
	as $t \to \infty$ are shown by the same way with the use of \eqref{eq:7.9}, \eqref{eq:7.11}, \eqref{eq:4.6} and \eqref{eq:4.7}.
	Therefore we get
	\begin{align} \label{eq:7.22}
		& \| \nabla^{\alpha} (u_{0,lin}(t)-G_{0,lin}(t)) \|_{2} = o(t^{-\frac{1}{4}-\frac{\alpha}{2}}), \quad 1 \le \alpha \le 3
	\end{align}
	by \eqref{eq:7.18} and \eqref{eq:7.20}, and   
	\begin{align} \label{eq:7.23}
		& \| \nabla^{\alpha} (u_{1,lin}(t)-G_{1,lin}(t)) \|_{2} = o(t^{-\frac{1}{4}-\frac{\alpha}{2}}), \quad 1 \le \alpha \le 3
	\end{align}
	by \eqref{eq:7.19} and \eqref{eq:7.21}, as $t \to \infty$.
	The estimates \eqref{eq:7.22} and \eqref{eq:7.23} imply the desired estimate \eqref{eq:7.12}.
	
	%
	The estimates \eqref{eq:7.13}-\eqref{eq:7.15} are proved by the similar way to \eqref{eq:7.12}. 
	We omit the proof. 
%

Finally we prove the estimate \eqref{eq:7.17}.
	Now we observe that
	\begin{equation} \label{eq:7.24}
		\begin{split}
			& \|
			 \partial_{t}^{2} K_{0}^{(\beta)}(t)\mathcal{R}_{a} \mathcal{R}_{b} g
			+ \Delta \nabla^{-1}
			\beta^{2} m_{\nabla g} \mathcal{R}_{a} \mathcal{R}_{b} G^{(\beta)}_{0}(t )
			\|_{p} \\
			& \le \|
			\partial_{t}^{2} \mathbb{K}_{00}^{(\beta)}(t)\ast g
			+ \nabla^{2}
			\beta^{2} \mathbb{G}^{(\beta)}_{0}(t) \ast g
			\|_{p} +\|
			 \nabla
			\beta^{2} (\mathbb{G}^{(\beta)}_{0}(t) \ast \nabla g-m_{\nabla g} \mathbb{G}^{(\beta)}_{0}(t) )
			\|_{p}  \\
			& 
			+C \|
			\nabla^{2}  \partial_{t}^{2} \mathbb{K}_{1}^{(\beta)}(t)\ast g
			\|_{p} \\
			&
			+C \sum_{k=M,H} 
			\{ 
			\|
			\partial_{t}^{2} K_{0k}^{(\beta)}(t)\mathcal{R}_{a} \mathcal{R}_{b} g
			\|_{p}
			+ \|
			\Delta \nabla^{-1}
			\mathcal{R}_{a} \mathcal{R}_{b} G^{(\beta)}_{0k}(t)
			\|_{p} \} \\
			& \le C(1+t)^{-\frac{5}{2}(1-\frac{1}{p})-\frac{1}{2}} \| \nabla g\|_{1} +  o(t^{-\frac{5}{2}(1-\frac{1}{p})}) +Ce^{-ct} \| \nabla^{3}  g \|_{2} =o(t^{-\frac{5}{2}(1-\frac{1}{p})}) 
		\end{split}
	\end{equation}
	as $t \to \infty$,
	where we used the estimates \eqref{eq:3.4}, \eqref{eq:3.52}, \eqref{eq:3.49}, \eqref{eq:4.10} and \eqref{eq:4.14}.
	Thus we obtain 
	\begin{equation*} 
		\begin{split}
			& \left\|
			  ( \partial_{t}^{2} K_{0}^{(\beta)}(t)\mathcal{F}^{-1} [\mathcal{P} \hat{f}_{0}])_{(j)}
			+ \Delta \nabla^{-1}
			\beta^{2} 
			\left(\mathcal{F}^{-1} \left[ 
			\mathcal{G}^{(\beta)}_{0}(t,\xi) \mathcal{P} m_{0}
			\right]
			\right)_{(j)}
			\right\|_{p} \\
			& \le C \sum_{k=1}^{3} 
			\left\|
			\partial_{t}^{2}  K_{0}^{(\beta)}(t) \mathcal{R}_{j} \mathcal{R}_{k} \nabla f_{0k}+m_{0k} \Delta \nabla^{-1}
			\beta^{2} \mathcal{R}_{j} \mathcal{R}_{k}G_{1}^{(\beta)}(t)
			\right\|_{p} =o(t^{-\frac{5}{2}(1-\frac{1}{p})})
		\end{split}
	\end{equation*}
	as $t \to \infty$, by \eqref{eq:7.24}.
	On the other hand, it follows from the estimates \eqref{eq:7.4} and \eqref{eq:4.14} that
	\begin{equation*} 
		\begin{split}
			& \left\| 
			 ( \partial_{t}^{2} K_{1}^{(\beta)}(t)\mathcal{F}^{-1} [\mathcal{P} \hat{f}_{1}])_{(j)}
			+ \beta^{2} \nabla^{2}
			\left(\mathcal{F}^{-1} \left[ 
			\mathcal{G}^{(\beta)}_{1}(t,\xi) \mathcal{P} m_{1}
			\right]
			\right)_{(j)}
			\right\|_{p} \\
			& \le C \sum_{k=1}^{3} 
			\left\|
			\partial_{t}^{2}  K_{1}^{(\beta)}(t) \mathcal{R}_{j} \mathcal{R}_{k} f_{1k}+m_{1k} \beta^{2} \nabla^{2} \mathcal{R}_{j} \mathcal{R}_{k}G_{1}^{(\beta)}(t)
			\right\|_{p} \\
			& \le o(t^{-\frac{5}{2}(1-\frac{1}{p})}) + C \sum_{k=1}^{3} \left\|
			\partial_{t}^{2} K_{1H}^{(\beta)}(t) \mathcal{R}_{j} \mathcal{R}_{k} f_{1k}
			\right\|_{p}=o(t^{-\frac{5}{2}(1-\frac{1}{p})})
		\end{split}
	\end{equation*}
	as $t \to \infty$. 
	Again, we also get the estimates 
	\begin{equation*} 
		\begin{split}
			\left\| 
			  ( \partial_{t}^{2} K_{0}^{(\beta)}(t) f_{0})_{(j)}
			+ \Delta \nabla^{-1}
			\beta^{2} 
			\left(\mathcal{F}^{-1} \left[ 
			\mathcal{G}^{(\beta)}_{0}(t,\xi) m_{0}
			\right]
			\right)_{(j)}
			\right\|_{p} =o(t^{-\frac{5}{2}(1-\frac{1}{p})})
		\end{split}
	\end{equation*}
	and 
	\begin{equation*} 
		\begin{split}
			\left\| 
			 ( \partial_{t}^{2} K_{1}^{(\beta)}(t) f_{1})_{(j)}
			+ \beta^{2} \nabla^{2}
			\left(\mathcal{F}^{-1} \left[ 
			\mathcal{G}^{(\beta)}_{1}(t,\xi) m_{1}
			\right]
			\right)_{(j)} 
			\right\|_{p} =o(t^{-\frac{5}{2}(1-\frac{1}{p})})
		\end{split}
	\end{equation*}
	as $t \to \infty$, by the same method as in \eqref{eq:7.20} and \eqref{eq:7.21}.
	Summing up these estimates, 
	we conclude the estimates \eqref{eq:7.17}, which completes the proof.
\end{proof}

\subsection{Asymptotic profiles of nonlinear terms}
In this subsection, we study the asymptotic behavior of non-homogeneous term, which is useful to obtain the asymptotic profile of the nonlinear term in \eqref{eq:1.1}. 
\begin{prop} \label{prop:7.5}
	Let $\alpha, \ell, m \ge 0$.
	Suppose that $f \in L^{1}(0,\infty;L^{1}(\R^{3}))$ with $\| f(t) \|_{1} \le C (1+t)^{-2}$.
	Then the following estimates hold as $t \to \infty$:
	\begin{equation} \label{eq:7.25}
		\begin{split}
			& \left\| \nabla^{\alpha} \left(\int_{0}^{\frac{t}{2}} \partial_{t}^{\ell} \mathbb{K}^{(\beta)}_{1}(t-\tau) \ast f(\tau) d \tau 
			-(-1)^{\frac{\ell}{2}} \beta^{\ell} \nabla^{\ell} 
			\int_{0}^{\frac{t}{2}} 
			\int_{\R^{3}} f(\tau,y) dy d \tau\,  
			\mathbb{G}_{1}(t)
			\right)
			\right\|_{p} \\
			& = o(t^{-\frac{5}{2}(1-\frac{1}{p})+1-\frac{\alpha+\ell}{2}})
		\end{split}
	\end{equation}
	for $\ell=2m$ and 
	\begin{equation} \label{eq:7.26}
		\begin{split}
			& \left\| \nabla^{\alpha} \left(\int_{0}^{\frac{t}{2}} \partial_{t}^{\ell} \mathbb{K}^{(\beta)}_{1}(t-\tau) \ast f(\tau) d \tau 
			-(-1)^{\frac{\ell-1}{2}} \beta^{\ell-1} \nabla^{\ell-1} 
			\int_{0}^{\frac{t}{2}} 
			\int_{\R^{3}} f(\tau,y) dy d \tau\,  
			\mathbb{G}_{0}(t)
			\right)
			\right\|_{p} \\
			& = o(t^{-\frac{5}{2}(1-\frac{1}{p})+1-\frac{\alpha+\ell}{2}})
		\end{split}
	\end{equation}
	for $\ell=2m+1$, where $1 <p \le \infty$ for $\ell+\alpha=0$ and $1 \le p \le \infty$ for $\ell+\alpha \ge 1$.
\end{prop}
\begin{proof}[Proof of Proposition \ref{prop:7.5}]
	We only prove the estimate \eqref{eq:7.25}.
	The estimate \eqref{eq:7.26} is shown by a similar way.
	At first, we decompose the integrand into three parts:
	\begin{equation} \label{eq:7.27}
		\begin{split}
			& \nabla^{\alpha} \left(\int_{0}^{\frac{t}{2}} \partial_{t}^{\ell} \mathbb{K}^{(\beta)}_{1}(t-\tau) \ast f(\tau,x) d \tau 
			-(-1)^{\frac{\ell}{2}} \beta^{\ell} \nabla^{\ell} 
			\int_{0}^{\frac{t}{2}} 
			\int_{\R^{3}} f(\tau,y) dy d \tau\,  
			\mathbb{G}_{1}(t,x)
			\right)
			 \\
			& =I_{1}+I_{2}+I_{3},
		\end{split}
	\end{equation}
where
	\begin{equation*} 
		\begin{split}
			I_{1}& :=\nabla^{\alpha} \int_{0}^{\frac{t}{2}} \left( \partial_{t}^{\ell} \mathbb{K}^{(\beta)}_{1}(t-\tau) 
			-(-1)^{\frac{\ell}{2}} \beta^{\ell} \nabla^{\ell} 
			\mathbb{G}_{1}(t-\tau) \right) \ast f(\tau,x)
			d \tau, \\
			I_{2}& :=(-1)^{\frac{\ell}{2}} \beta^{\ell} \nabla^{\alpha+\ell} \int_{0}^{\frac{t}{2}} \left( \mathbb{G}^{(\beta)}_{1}(t-\tau) 
			- 
			\mathbb{G}_{1}(t) \right) \ast f(\tau,x) 
			d \tau, \\
			I_{3}& := (-1)^{\frac{\ell}{2}} \beta^{\ell} \nabla^{\alpha+\ell} \int_{0}^{\frac{t}{2}} \left( \mathbb{G}_{1}(t)\ast f(\tau,x) 
			- 
			\int_{\R^{3}} f(\tau, y) dy d \tau\, \mathbb{G}_{1}(t,x) \right).
		\end{split}
	\end{equation*}
	For $I_{1}$, we apply the estimate \eqref{eq:3.6} to have 
	\begin{equation} \label{eq:7.28}
		\begin{split}
			\|I_{1} \|_{p} & \le \left\| \nabla^{\alpha} \int_{0}^{\frac{t}{2}} \left( \partial_{t}^{\ell} \mathbb{K}^{(\beta)}_{1}(t-\tau) 
			-(-1)^{\frac{\ell}{2}} \beta^{\ell} \nabla^{\ell} 
			\mathbb{G}_{1}(t-\tau) \right) \ast f(\tau,x)
			d \tau \right\|_{p} \\
			& \le C  \int_{0}^{\frac{t}{2}} (1+t-\tau)^{-\frac{5}{2}(1-\frac{1}{p})+\frac{1}{2}-\frac{\alpha+\ell}{2}} \|f(\tau) \|_{1}
			d \tau \\
			& \le C (1+t)^{-\frac{5}{2}(1-\frac{1}{p})+\frac{1}{2}-\frac{\alpha+\ell}{2}} \int_{0}^{\frac{t}{2}}  (1+\tau)^{-2} d \tau \\
			& \le (1+t)^{-\frac{5}{2}(1-\frac{1}{p})+\frac{1}{2}-\frac{\alpha+\ell}{2}}.
		\end{split}
	\end{equation}
	To obtain the estimate for $I_{2}$, using the mean value theorem on $t$, we see that 
	\begin{equation*}
		\begin{split}
			\mathbb{G}_{1}^{(\beta)}(t-\tau, x-y)-\mathbb{G}_{1}^{(\beta)}(t,x-y) = (-\tau) \partial_{t} 
			\mathbb{G}_{1}^{(\beta)}(t-\theta \tau, x-y) 
		\end{split}
	\end{equation*}
	for some $\theta \in [0,1]$.
	Therefore we have 
	\begin{equation} \label{eq:7.29}
		\begin{split}
			\| I_{2} \|_{p} & \le \int_{0}^{\frac{t}{2}} \tau \| \nabla^{\alpha+\ell} \partial_{t} 
			\mathbb{G}_{1}^{(\beta)}(t-\theta \tau) \ast f(\tau) \|_{p} d \tau \\
			& \le C\int_{0}^{\frac{t}{2}} \tau (t-\tau)^{-\frac{5}{2}(1-\frac{1}{p})+\frac{1}{2}-\frac{\alpha +\ell}{2}}  \| f(\tau) \|_{1} d \tau \\
			& \le Ct^{-\frac{5}{2}(1-\frac{1}{p})+\frac{1}{2}-\frac{\alpha+\ell}{2}} \int_{0}^{\frac{t}{2}} \tau (1+\tau)^{-2} d \tau \\
			& \le Ct^{-\frac{5}{2}(1-\frac{1}{p})+\frac{1}{2}-\frac{\alpha+\ell}{2}} \log(t+2)
		\end{split}
	\end{equation}
	by \eqref{eq:3.51}.
	The proof is completed by showing the estimate for $I_{3}$.
	Now we divide $I_{3}$ into 2 parts:
	\begin{equation} \label{eq:7.30}
		\begin{split}
			I_{3} & := I_{31}+I_{32},
		\end{split}
	\end{equation}
	where 
	\begin{equation*}
		\begin{split}
			I_{31} & :=  (-1)^{\frac{\ell}{2}} \beta^{\ell} \nabla^{\alpha+\ell} \int_{0}^{\frac{t}{2}}
			\int_{|y| \le t^{\frac{1}{4}}}
			\left( 
			\mathbb{G}^{(\beta)}_{1}(t,x-y)-\mathbb{G}^{(\beta)}_{1}(t,x)
			\right)
			f(\tau,y) dy d \tau, \\ 
			I_{32} & := (-1)^{\frac{\ell}{2}} \beta^{\ell} \nabla^{\alpha+\ell}  \int_{0}^{\frac{t}{2}}
			\int_{|y| \ge t^{\frac{1}{4}}}
			\left( 
			\mathbb{G}^{(\beta)}_{1}(t,x-y)-\mathbb{G}^{(\beta)}_{1}(t,x)
			\right)
			f(\tau,y) dy d \tau.
		\end{split}
	\end{equation*}
	%
	We apply the same argument in the proof of Proposition \ref{Prop:3.6} to have 
	\begin{equation} \label{eq:7.31}
		\begin{split}
			\| I_{31} \|_{p} \le C (1+t)^{-\frac{5}{2}(1-\frac{1}{p})+\frac{3}{4}-\frac{\alpha+\ell}{2}}
		\end{split}
	\end{equation}
and 
\begin{equation} \label{eq:7.32}
		\begin{split}
			& \| I_{32}\|_{p}=o(t^{-\frac{5}{2}(1-\frac{1}{p})+1-\frac{\alpha+\ell}{2}})
		\end{split}
	\end{equation}
	as $t \to \infty$, where we used the fact that 
	$
	\| f \|_{L^{1}(0,\infty;L^{1}(\R^{3}))}  <\infty
	$
	and then
	$$
	\lim_{t\to \infty}
	\int_{0}^{\infty}
	\int_{|y| \ge t^{\frac{1}{4}}}
	| f(\tau,y)| dy d \tau =0.
	$$
	Therefore summing up \eqref{eq:7.27}-\eqref{eq:7.32}, 
	we can conclude the estimate \eqref{eq:7.25}, 
	which proves the proposition.
	%
\end{proof}
The following estimates are also useful to obtain the asymptotic profiles of solutions to \eqref{eq:1.1}.
\begin{cor} \label{cor:7.6}
	Under the assumption on Proposition \ref{prop:7.5}, the following estimates holds:
	\begin{equation} \label{eq:7.33}
		\begin{split}
			& \left\| \nabla^{\alpha} \left(\int_{0}^{\frac{t}{2}} \partial_{t}^{\ell} K^{(\beta)}_{1L}(t-\tau) f(\tau) d \tau 
			-(-1)^{\frac{\ell}{2}} \beta^{\ell} \nabla^{\ell} 
			\int_{0}^{\frac{t}{2}} 
			\int_{\R^{3}} f(\tau,y) dy d \tau\,  
			G_{1L}(t)
			\right)
			\right\|_{p} \\
			& = o(t^{-\frac{5}{2}(1-\frac{1}{p})+1-\frac{\alpha+\ell}{2}})
		\end{split}
	\end{equation}
	for $\ell=2m$ and 
	\begin{equation} \label{eq:7.34}
		\begin{split}
			& \left\| \nabla^{\alpha} \left(\int_{0}^{\frac{t}{2}} \partial_{t}^{\ell} K^{(\beta)}_{1L}(t-\tau) \ast f(\tau) d \tau 
			-(-1)^{\frac{\ell-1}{2}} \beta^{\ell-1} \nabla^{\ell-1} 
			\int_{0}^{\frac{t}{2}} 
			\int_{\R^{3}} f(\tau,y) dy d \tau\,  
			G_{0L}(t)
			\right)
			\right\|_{p} \\
			& = o(t^{-\frac{5}{2}(1-\frac{1}{p})+1-\frac{\alpha+\ell}{2}})
		\end{split}
	\end{equation}
	for $\ell=2m+1$, where $1 \le p \le \infty$.
\end{cor}
	Corollary \ref{cor:7.6} can be shown by the same way as in Proposition \ref{prop:7.5}.
\subsection{Proof of Theorems \ref{thm:1.5}}
From now on, $u_{N}(t)$ denotes the nonlinear term of the integral equation corresponding to \eqref{eq:1.1},
\begin{equation*} 
	\begin{split}
		u_{N}(t) & := \int_{0}^{t}
		(K_{1}^{(\sqrt{\lambda+2 \mu})}(t-\tau)-K_{1}^{(\sqrt{\mu})}(t-\tau)) \mathcal{F}^{-1} [\mathcal{P} \hat{F}(u)(\tau) ] d \tau \\
		& + \int_{0}^{t} K_{1}^{(\sqrt{\mu})}(t-\tau) F(u)(\tau) d \tau
	\end{split}
\end{equation*}
for the simplicity.
\begin{proof}[Proof of Theorem  \ref{thm:1.5}]
	Let $F(u)=\nabla u \nabla^{2} u$.
	For the proof, we firstly claim that 
	\begin{align}
		& \| \nabla^{\alpha} (u_{N}(t)-G_{N}(t)) \|_{2} = o(t^{-\frac{1}{4}-\frac{\alpha}{2}}), \quad 1 \le \alpha \le 3, \label{eq:7.35} \\
		& \| \nabla^{\alpha} (u_{N}(t)-G_{N}(t)) \|_{\infty} =o(t^{-\frac{3}{2}-\frac{\alpha}{2}}), \quad 0 \le \alpha \le 1, \label{eq:7.36}  \\
		& \| \nabla^{\alpha} (\partial_{t}u_{N}(t) -H_{N}(t)) \|_{2} =o(t^{-\frac{3}{4}-\frac{\alpha}{2}}), \quad 0 \le \alpha \le 2, \label{eq:7.37}  \\
		& \| \nabla^{2} (\partial_{t} u_{N}(t) -H_{N}(t)) \|_{p} =o( t^{-\frac{5}{2}(1-\frac{1}{p})-\frac{1}{2}} ), \quad 2 \le p <6, \label{eq:7.38}  \\
		& \| \nabla^{\alpha} (\partial_{t} u_{N}(t) -H_{N}(t))\|_{\infty}
		=o( t^{-2-\frac{\alpha}{2}}),\quad 0 \le \alpha \le 1, \label{eq:7.39}  \\
		& \| \partial_{t}^{2} u_{N}(t) -\tilde{G}_{N}(t) \|_{p} = o(t^{-\frac{5}{2}(1-\frac{1}{p})-\frac{1}{2}}),  \quad 2 \le p <6 \label{eq:7.40} 
	\end{align}
	as $t \to \infty$.
	
	At first,
	we collect the estimates for the nonlinear term.
	Namely, we easily see that 
	\begin{equation} \label{eq:7.41}
		\begin{split}
			\left| \int_{0}^{\frac{t}{2}} \int_{\R^{3}} F(u)(\tau, y) dy d \tau \right| 
			\le C \int_{0}^{\frac{t}{2}} (1+\tau)^{-2} d \tau \le C 
		\end{split}
	\end{equation}
	and
	\begin{equation} \label{eq:7.42}
		\begin{split}
			\left|
			\int_{\frac{t}{2}}^{\infty} \int_{\R^{3}} F(u)(\tau, y) dy d \tau 
			\right| \le C \int_{\frac{t}{2}}^{\infty} (1+\tau)^{-2} d \tau \le C(1+t)^{-1}, 
		\end{split}
	\end{equation}
	where we used \eqref{eq:5.4} for $F(u)=\nabla u \nabla^{2} u$.
	Thus, by \eqref{eq:2.11}, \eqref{eq:2.12} and \eqref{eq:7.41}, we have 
	\begin{equation} \label{eq:7.43}
		\begin{split}
			& \| \partial_{t}^{\ell} \nabla^{\alpha} (G_{0M}(t) +G_{0H}(t)) \|_{p}
			\left|
			\int_{0}^{\frac{t}{2}} \int_{\R^{3}} F(u)(\tau, y) dy d \tau 
			\right| \\
			& +
			\| \partial_{t}^{\ell} \nabla^{\alpha} 
			\mathcal{F}^{-1} [\mathcal{G}^{(\beta)}_{0}(t,\xi)(\chi_{M} +\chi_{H}) \mathcal{P}] \|_{p}
			\left|
			\int_{0}^{\frac{t}{2}} \int_{\R^{3}} F(u)(\tau, y) dy d \tau 
			\right| \\
			& \le Ce^{-ct} t^{-\frac{3}{2}(1-\frac{1}{p})-\frac{\alpha+\ell}{2}}
		\end{split}
	\end{equation}
	and
	\begin{equation} \label{eq:7.44}
		\begin{split}
			& \| \partial_{t}^{\ell} \nabla^{\alpha} (G_{1M}(t) +G_{1H}(t)) \|_{p}
			\left|
			\int_{0}^{\frac{t}{2}} \int_{\R^{3}} F(u)(\tau, y) dy d \tau 
			\right| \\
			& +
			\| \partial_{t}^{\ell} \nabla^{\alpha} 
			\mathcal{F}^{-1} [\mathcal{G}^{(\beta)}_{1}(t,\xi) (\chi_{M} +\chi_{H})  \mathcal{P}] \|_{p}
			\left|
			\int_{0}^{\frac{t}{2}} \int_{\R^{3}} F(u)(\tau, y) dy d \tau 
			\right| \\
			& \le Ce^{-ct} t^{-\frac{3}{2}(1-\frac{1}{p})-\frac{\alpha+\ell}{2}}
		\end{split}
	\end{equation}
	for $1 \le p \le \infty$ and $\alpha, \ell \ge 0$.
	Likewise, we also have 
	\begin{equation} \label{eq:7.45}
		\begin{split}
			& \| \partial_{t}^{\ell} \nabla^{\alpha} G_{0}(t) \|_{p}
			\left|
			\int_{\frac{t}{2}}^{\infty} \int_{\R^{3}} F(u)(\tau, y) dy d \tau 
			\right| \\
			& +
			\| \partial_{t}^{\ell} \nabla^{\alpha} 
			\mathcal{F}^{-1} [\mathcal{G}^{(\beta)}_{0}(t,\xi) \mathcal{P} ] \|_{p}
			\left|
			\int_{\frac{t}{2}}^{\infty} \int_{\R^{3}} F(u)(\tau, y) dy d \tau 
			\right| \le C t^{-\frac{5}{2}(1-\frac{1}{p})-\frac{1}{2}-\frac{\alpha+\ell}{2}}
		\end{split}
	\end{equation}
	by \eqref{eq:2.9}, \eqref{eq:3.50} and \eqref{eq:2.11},
	and

	\begin{equation} \label{eq:7.46}
		\begin{split}
			& \| \partial_{t}^{\ell} \nabla^{\alpha} G_{1}(t) \|_{p}
			\left|
			\int_{\frac{t}{2}}^{\infty} \int_{\R^{3}} F(u)(\tau, y) dy d \tau 
			\right| \\
			& +
			\| \partial_{t}^{\ell} \nabla^{\alpha} 
			\mathcal{F}^{-1} [\mathcal{G}^{(\beta)}_{1}(t,\xi) \mathcal{P} ] \|_{p}
			\left|
			\int_{\frac{t}{2}}^{\infty} \int_{\R^{3}} F(u)(\tau, y) dy d \tau 
			\right| \le C t^{-\frac{5}{2}(1-\frac{1}{p})-\frac{\alpha+\ell}{2}}
		\end{split}
	\end{equation}
	by \eqref{eq:2.10}, \eqref{eq:3.51} and \eqref{eq:2.12}, 
	where $1<p \le \infty$ for $\ell+\alpha=0$ and $1 \le p \le \infty$ for $\ell+\alpha \ge 1$.

	We now turn to the proof of estimates \eqref{eq:7.35}-\eqref{eq:7.40}. 
	We only show the estimate \eqref{eq:7.35}. 
	The other estimates are proved in a similar way.
	For this aim, we use the following decomposition:
	\begin{equation*} 
		\begin{split}
			u_{N}(t)-G_{N}(t) & = J_{0, \mathcal{P}, \lambda+2 \mu}- J_{0,\mathcal{P},\mu}+ J_{0,\phi, \mu},
		\end{split}
	\end{equation*}
where
\begin{equation*} 
	\begin{split}
		&J_{0, \mathcal{P}, \beta} 
		 := \\
		& \int_{0}^{\frac{t}{2}}
		K_{1L}^{(\sqrt{\beta})}(t-\tau) \mathcal{F}^{-1} [\mathcal{P} \hat{F}(u)(\tau) ] d \tau 
		-
		\mathcal{F}^{-1}
		\left[
		\mathcal{G}_{1}^{(\sqrt{\beta})}(t,\xi) \chi_{L} \mathcal{P}
		\right]\, \int_{0}^{\frac{t}{2}} \int_{\R^{3}} F(u)(\tau, y) dy d \tau 
		 \\
		 & - \mathcal{F}^{-1}
		 \left[
		 \mathcal{G}_{1}^{(\sqrt{\beta})}(t,\xi) (\chi_{M}+\chi_{H}) \mathcal{P}
		 \right]\, \int_{0}^{\frac{t}{2}} \int_{\R^{3}} F(u)(\tau, y) dy d \tau \\
		& +\int_{\frac{t}{2}}^{t}
		K_{1L}^{(\sqrt{\beta})}(t-\tau) \mathcal{F}^{-1} [\mathcal{P} \hat{F}(u)(\tau) ] d \tau
		 +\sum_{k=M,H} \int_{0}^{t}
		K_{1k}^{(\sqrt{\beta})}(t-\tau) \mathcal{F}^{-1} [\mathcal{P} \hat{F}(u)(\tau) ] d \tau \\
		& -
		\mathcal{F}^{-1}
		\left[
		\mathcal{G}_{1}^{(\sqrt{\beta})}(t,\xi) \mathcal{P}
		\right]\, \int_{\frac{t}{2}}^{\infty} \int_{\R^{3}} F(u)(\tau, y) dy d \tau 
	\end{split}
\end{equation*}
and 
\begin{equation*} 
	\begin{split}
		J_{0, \phi, \beta} 
		& :=
		\int_{0}^{\frac{t}{2}}
		K_{1L}^{(\sqrt{\beta})}(t-\tau) F(u)(\tau)  d \tau  -
		G_{1L}^{(\sqrt{\beta})}(t)  \int_{0}^{\frac{t}{2}} \int_{\R^{3}} F(u)(\tau, y) dy d \tau 
		\\
		& -
		(G_{1M}^{(\sqrt{\beta})}(t)+G_{1M}^{(\sqrt{\beta})}(t))  \int_{0}^{\frac{t}{2}} \int_{\R^{3}} F(u)(\tau, y) dy d \tau \\
		& +\int_{\frac{t}{2}}^{t}
		K_{1L}^{(\sqrt{\beta})}(t-\tau) F(u)(\tau) d \tau
		 +\sum_{k=M,H} \int_{0}^{t}
		K_{1k}^{(\sqrt{\beta})}(t-\tau) F(u)(\tau) d \tau \\
		& -
		G_{1}^{(\sqrt{\beta})}(t) \int_{\frac{t}{2}}^{\infty} \int_{\R^{3}} F(u)(\tau, y) dy d \tau. 
	\end{split}
\end{equation*}
It is easy to see that 
\begin{equation} \label{eq:7.47}
	\begin{split}
		\| \nabla^{\alpha} J_{0, \mathcal{P}, \beta} \|_{2} 
		& \le o(t^{-\frac{1}{4}-\frac{\alpha}{2}} ) +C \int_{\frac{t}{2}}^{t}
		(1+t-\tau)^{1-\frac{\alpha}{2}} \| F(u)(\tau) \|_{2}  d \tau\\
		& +\sum_{k=M,H} \int_{0}^{t}
		e^{-c(t-\tau)} \| \nabla F(u)(\tau) \|_{2} d \tau
		 +
	  C t^{-\frac{5}{4}-\frac{\alpha}{2}} \\
		 & \le o(t^{-\frac{1}{4}-\frac{\alpha}{2}} ) +C \int_{\frac{t}{2}}^{t}
		 (1+t-\tau)^{1-\frac{\alpha}{2}} (1+\tau)^{-\frac{11}{4}}  d \tau \\
		& +C\int_{0}^{t}
		 e^{-c(t-\tau)} (1+\tau)^{-\frac{13}{4}}d \tau
		 \\
		& \le o(t^{-\frac{1}{4}-\frac{\alpha}{2}} ) +C (1+t)^{-\frac{3}{4}-\frac{\alpha}{2}}=o(t^{-\frac{1}{4}-\frac{\alpha}{2}} ) 
	\end{split}
\end{equation}
as $t \to \infty$ for $1 \le \alpha \le 3$, by \eqref{eq:7.25}, \eqref{eq:3.49}, \eqref{eq:7.44} and \eqref{eq:7.46}.
By the same manner, we also have 
\begin{equation} \label{eq:7.48}
	\begin{split}
		\| \nabla^{\alpha} J_{0, \phi, \beta} \|_{2} =o(t^{-\frac{1}{4}-\frac{\alpha}{2}} ) 
	\end{split}
\end{equation}
as $t \to \infty$ for $1 \le \alpha \le 3$.
%
Noting \eqref{eq:7.47} and \eqref{eq:7.48}, 
we obtian
\begin{equation*} 
	\begin{split}
		\| \nabla^{\alpha} (u_{N}(t)-G_{N}(t))\|_{2} & \le  
		\| \nabla^{\alpha} J_{0, \mathcal{P}, \lambda+2 \mu}\|_{2}+\| \nabla^{\alpha} J_{0,\mathcal{P},\mu} \|_{2}+\| \nabla^{\alpha}  J_{0,\phi,\mu} \|_{2} \\
		& =o(t^{-\frac{1}{4}-\frac{\alpha}{2}} ) 
	\end{split}
\end{equation*}
as $t \to \infty$ for $1 \le \alpha \le 3$, which is the desired estimate \eqref{eq:7.35}.

	Once we have \eqref{eq:7.35}-\eqref{eq:7.40} combined with the estimates \eqref{eq:7.12}-\eqref{eq:7.17},
	we immediately conclude the desired estimates \eqref{eq:1.12}-\eqref{eq:1.17}, 
	since $u(t) =u_{lin}(t) +u_{N}(t)$.
	We complete the proof of Theorem \ref{thm:1.5}.
\end{proof}

%
\section{Appendix}
\begin{proof}[Proof of \eqref{eq:3.36}]
	We first observe that
	\begin{equation} \label{eq:a.1}
		\begin{split}
			\nabla_{\xi}^{k} |\xi|= O(|\xi|^{1-k})
		\end{split}
	\end{equation}
    and 
	\begin{equation} \label{eq:a.2}
		\begin{split}
			 \nabla_{\xi}^{k} \{ ( \phi_{\nu,\beta}-1)|\xi| \}=  \nabla_{\xi}^{k} O(|\xi|^{3-k}) =O(|\xi|^{2})
		\end{split}
	\end{equation}
	as $|\xi| \to 0$  for $k \in \mathbb{N}$ by \eqref{eq:3.18}.
	Now we decompose the integrand into three parts;
	
	\begin{equation} \label{eq:a.3}
		\begin{split}
			& \nabla_{\xi}^{2} \{ \sin (\beta |\xi| \phi_{\nu,\beta} t)- \sin (\beta |\xi| t)- \beta |\xi| t(\phi_{\nu,\beta}-1)\cos (\beta |\xi| t) \}
			=
			\mathcal{B}_{1}+\mathcal{B}_{2}+\mathcal{B}_{3},
		\end{split}
	\end{equation}
	where 
	\begin{equation*} 
		\begin{split}
			\mathcal{B}_{1} & := (\beta t)^{2}
			\{ -\sin (\beta |\xi| \phi_{\nu,\beta} t) (\nabla_{\xi}(|\xi| \phi_{\nu,\beta}) )^{2} +\sin (\beta |\xi|  t) (\nabla_{\xi}|\xi|)^{2} \\
			& +\beta t (\phi_{\nu,\beta}-1)|\xi| \cos (\beta |\xi|  t) (\nabla_{\xi}|\xi|)^{2} \},
		\end{split}
	\end{equation*}
	\begin{equation*}
		\begin{split}
			\mathcal{B}_{2} & := (\beta t)^{2}
			\{ \cos (\beta |\xi| \phi_{\nu,\beta} t) \nabla_{\xi}^{2}(|\xi| \phi_{\nu,\beta}) -\cos (\beta |\xi|  t) \nabla_{\xi}^{2}|\xi|\}, \\
		\end{split}
	\end{equation*}
	\begin{equation*}
		\begin{split}
			\mathcal{B}_{3} & := (-\beta t) \nabla_{\xi}^{2} [(\phi_{\nu,\beta}-1)|\xi|]
			\cos (\beta |\xi| \phi_{\nu,\beta} t) \\ 
			& +2 (\beta t)^{2} \nabla_{\xi}[(\phi_{\nu,\beta}-1)|\xi|] \sin (\beta |\xi|  t) \nabla_{\xi}^{2}|\xi|.
		\end{split}
	\end{equation*}
	%
	Then it follows from \eqref{eq:a.2} and \eqref{eq:3.21} that 
	\begin{equation} \label{eq:a.4}
		\begin{split}
			\mathcal{B}_{1} & = (\beta t)^{2}(-\sin (\beta |\xi| \phi_{\nu,\beta} t)) \{ (\nabla_{\xi}(|\xi| \phi_{\nu,\beta}) )^{2}-(\nabla_{\xi}|\xi| )^{2} \} \\
			& + (\beta t)^{2} \{ -\sin (\beta |\xi| \phi_{\nu,\beta} t)  +\sin (\beta |\xi|  t)  +\beta t (\phi_{\nu,\beta}-1)|\xi| \cos (\beta |\xi|  t) \} (\nabla_{\xi}|\xi|)^{2} \\
			& =to(|\xi|^{2}) +t^{4} O(|\xi|^{6})
		\end{split}
	\end{equation}
	as $|\xi| \to 0$.
	$\mathcal{B}_{2}$ is estimated by \eqref{eq:a.1}, \eqref{eq:a.2} and \eqref{eq:3.26} as follows;
	\begin{equation} \label{eq:a.5}
		\begin{split}
			\mathcal{B}_{2} & = \beta t
			\{ \cos (\beta |\xi| \phi_{\nu,\beta} t) \nabla_{\xi}^{2}(|\xi| (\phi_{\nu,\beta}-1)) +
			(\cos (\beta |\xi| \phi_{\nu,\beta} t) - \cos (\beta |\xi|  t)) \nabla_{\xi}^{2}|\xi|\} \\
			& =t^{2} O(|\xi|^{2}) +t^{4} O(|\xi|^{6})
		\end{split}
	\end{equation}
as $|\xi| \to 0$. In the same manner we can see that
	\begin{equation} \label{eq:a.6}
		\begin{split}
			\mathcal{B}_{3} & = -\beta t \nabla_{\xi}^{2}(|\xi| (\phi_{\nu,\beta}-1)) \cos (\beta |\xi|  t) \\
			& +2(\beta t)^{2} \nabla_{\xi}(|\xi| (\phi_{\nu,\beta}-1))
			 \sin (\beta |\xi|  t)) \nabla_{\xi}|\xi| \\
			& =tO(|\xi|) +t^{2} O(|\xi|^{2})
		\end{split}
	\end{equation}
as $|\xi| \to 0$ by \eqref{eq:a.1} and \eqref{eq:a.2}. 
	Combining \eqref{eq:a.4}-\eqref{eq:a.6}, we have the desired estimate \eqref{eq:3.36},
	since we used the assumption $\xi \in \supp \chi_{L}$.
\end{proof}


%
\vspace*{5mm}
\noindent
\textbf{Acknowledgments. } 

Y. Kagei was supported in part by JSPS Grant-in-Aid for Scientific Research (A) 20H00118.
H. Takeda was partially supported by JSPS Grant-in-Aid for Scientific Research (C) 19K03596.

\bibliographystyle{siam} 
\bibliography{mybibfile}
\end{document}